\documentclass[a4paper,reqno,11pt]{amsart}
\usepackage{amsmath,amsthm,amsfonts,amssymb,color,enumerate,bbm,dsfont,cancel}

\usepackage{hyperref}
\hypersetup{
     colorlinks   = true,
     citecolor    = black,
     linkcolor=black
}

\theoremstyle{plain}
\newtheorem{theorem}{Theorem}[section]

\newtheorem{proposition}[theorem]{Proposition}
\newtheorem{lemma}[theorem]{Lemma}

\theoremstyle{definition}
\newtheorem{remark}[theorem]{Remark}

\newtheorem{definition}[theorem]{Definition}

\numberwithin{equation}{section}

\def\cD{\mathcal{D}}

\def\cF{\mathcal{F}}
\def\cH{\mathcal{H}}
\def\cL{\mathcal{L}}
\def\cN{\mathcal{N}}

\def\e{\varepsilon}
\def\E{\mathbb E}

\def\P{\mathbb P}

\def\t{\boldsymbol t}
\def\h{\boldsymbol h}

\newcommand{\R}{\mathbb{R}}
\newcommand{\N}{\mathbb{N}}
\newcommand{\W}{\dot{W}}
\newcommand{\ud}{\ensuremath{ \mathrm{d}} }

\newcommand{\Var}{\mathrm{Var}}

\newcommand{\FoxH}[5]{H_{#2}^{#1}\left(#3\:\middle\vert\: \begin{subarray}{l}#4\\[0.4em] #5\end{subarray}\right)}

\newcommand*{\one}{{{\rm 1\mkern-1.5mu}\!{\rm I}}}

\makeatletter
\@namedef{subjclassname@2020}{\textup{2020} Mathematics Subject Classification}
\makeatother

\begin{document}

\title[Path properties and small ball probabilities for  Stochastic FDEs]{Sample path properties and small ball probabilities for  stochastic fractional diffusion equations}

\author[Y. Guo]{Yuhui Guo}
\address{Y. Guo: Faculty of Science and Technology, BNU-HKBU United International College, Zhuhai, Guangdong, 519087, China.}
 \email{\url{yuhuiguo@uic.edu.cn}}

\author[J. Song]{Jian Song}
\address{J. Song: Research Center for Mathematics and Interdisciplinary Sciences, Shandong University, Qingdao, Shandong, 266237, China.}
 \email{\url{txjsong@sdu.edu.cn}}

\author[R. Wang]{Ran Wang}
\address{R. Wang: School of Mathematics and Statistics, Wuhan University, Wuhan, 430072, China.}
 \email{\url{rwang@whu.edu.cn}}

\author[Y. Xiao]{Yimin Xiao}
\address{Y. Xiao: Department of Statistics and Probability, Michigan State University, East Lansing, MI 48824, USA.}
 \email{\url{xiaoy@msu.edu}}



\date{}

\begin{abstract}
  We consider the following stochastic space-time fractional diffusion equation with vanishing initial condition:
  \begin{equation*}
    \partial^{\beta} u(t, x)=- \left(-\Delta\right)^{\alpha / 2} u(t, x)+   I_{0+}^{\gamma}\left[\dot{W}(t, x)\right],\quad t\in[0,T],\: x \in \mathbb{R}^d,
  \end{equation*}
  where $\alpha>0$, $\beta\in(0,2)$, $\gamma\in[0,1)$,  $\left(-\Delta\right)^{\alpha / 2}$ is the fractional/power of Laplacian and $\W$
  is a fractional space-time Gaussian noise. We prove the existence and uniqueness of the solution and then focus on various
  sample path regularity properties of the solution. More specifically, we establish the exact uniform and local moduli of continuity and
  Chung’s laws of the iterated logarithm. The small ball probability is also studied.
\end{abstract}

\maketitle

\tableofcontents

\section{Introduction}

In this paper, we consider the following stochastic   fractional diffusion equation (FDE, for short):
\begin{equation}\label{e:fde}
	\begin{cases}
		\partial^{\beta} u(t, x)=-\left(-\Delta\right)^{\alpha / 2} u(t, x)+  I_{0+}^{\gamma}\left[\dot{W}(t, x)\right] , & t\in[0,T], x \in \mathbb{R}^d,   \\
		u(0, \cdot)=0,                                                                                       & \text { if } \beta \in(0,1], \\
		u(0, \cdot)=0, \quad \dfrac{\partial}{\partial t} u(0, \cdot)=0,                                 & \text { if } \beta \in(1,2),
	\end{cases}
\end{equation}
with $\alpha>0$, $\beta\in(0,2)$, and $\gamma\in[0,1)$.

In the above equation, $\left(-\Delta\right)^{\alpha / 2}$ is \textit{fractional} ($0<\alpha\le 2$) or \textit{power} of  ($\alpha>2$) \textit{Laplacian} in space. $\partial^\beta$   denotes the {\em Caputo fractional differential} operator in time:
\begin{equation}\label{e:CFDO}
  \partial^{\beta} f(t):=
\begin{cases}
  \displaystyle \frac{1}{\Gamma(n-\beta)}\int_{0}^{t}\frac{f^{(n)}(\tau)}{(t-\tau)^{\beta+1-n}}\ud\tau, & \mbox{if } \beta \neq n, \vspace{0.2cm} \\
  \dfrac{\ud^n}{\ud t^n}f(t),                                                                                    & \mbox{if }  \beta = n,
\end{cases}
\end{equation}
where $n=\lceil\beta\rceil$ is the smallest integer not smaller than $\beta$ and
 $\Gamma(x) = \int_{0}^{\infty}e^{-t}t^{x-1}\ud t$ is the Gamma function, and
$I_{0_+}^\gamma$  is the \emph{Riemann-Liouville integral} in time:
\begin{equation*}
  (I_{0+}^{\gamma}f)(t):=\frac{1}{\Gamma(\gamma)} \int_{0}^{t}f(r)(t-r)^{\gamma-1}\ud r.
\end{equation*}

In \eqref{e:fde}, the Gaussian noise $\dot W$ is  fractional with Hurst indices $H_0\in[1/2,1)$ in time and
$H_1,\dots,H_d\in(0,1)$ in space.  Namely, the covariance function of the Gaussian noise $\W$ is
\begin{equation*}
  \E[\W(t,x)\W(s,y)]=\Lambda_{H_0}(t-s)\prod_{j=1}^{d}\Lambda_{H_j}(x_j-y_j),
\end{equation*}
where for any $z\in \mathbb R$,
\begin{equation}\label{e:def-lambda-h}
\Lambda_H(z)=
\begin{cases}
H(2H-1) |z|^{2H-2}, & \text{ if } H>1/2,\\
\delta_0(z), & \text{ if } H=1/2,\\
\dfrac12 \left(|z|^{2H}\right)'' , & \text{ if } H<1/2,
\end{cases}
\end{equation}
with $\delta_0$ being the Dirac delta function and $(\cdot)''$ being the second derivative in the sense of distribution.

The class of stochastic FDE \eqref{e:fde} encompasses some well-known models as special cases, such as
stochastic (fractional) heat equations. Stochastic FDEs with multiplicative Gaussian noise have been studied in the
literature, and we refer to \cite{chen2017nonlinear,chen2022interpolating,chen2024moments,chen2019nonlinear,guo2023stochastic}
for recent developments.  Regarding path properties for stochastic FDEs with additive space-time white noise, the
case of  $\alpha=2$, $\beta\in(0,1/2]$, $\gamma=1-\beta$ has been studied by Allouba and Xiao \cite{Xiao2017JDE},
where \eqref{e:fde} is also called a stochastic partial integro-differential equation.  On the other hand, compared with the
space-time white noise, fractional Gaussian noise is more effective to describe correlated random perturbations in
biology, physics and finance, and there has been extensive research on SPDEs with fractional Gaussian noise in
recent two decades. We refer to \cite{balan2016intermittency,balan2015spdes,Balan2008Thestochastic,Balan2010thestocha,
Chen2018Intermittency,Chen2018Temporal,Herrell2020sharp,Hu2012Feynman,hu2009stochastic,
hns2011Feynman,hu2017stochastic,song2017class}, which by no means is complete, for related results.

Path regularity of the solutions to some special cases of stochastic FDE \eqref{e:fde} has been studied. For example,
for the solutions of stochastic heat equations driven by space-time colored noise, Tudor and Xiao \cite{Tudor2017Sample}
established exact moduli of continuity and  Chung’s laws of the iterated logarithm (Chung's LIL) in space and in time,
respectively;  the results were extended in Herrell {\em et al.} \cite{Herrell2020sharp} to  (fractional) heat equations,
where the exact moduli of continuity and Chung’s LIL were obtained in space and time simultaneously; in \cite{Lee2023Chung-type},
Lee and Xiao obtained exact moduli of continuity and Chung’s LIL by establishing harmonizable-type integral representation
and strong local nondeterminism for the solution, improving the results in \cite{Tudor2017Sample, Herrell2020sharp}
by providing more precise information on the limiting constants.  For \eqref{e:fde} with $\alpha=2, \beta\in(0,1/2],
\gamma=1-\beta$ driven by the space-time white noise, Allouba and Xiao \cite{Xiao2017JDE} studied the temporal
and spatial regularity of the solution.

In this paper, we study the path regularity of the solution $\{u(t,x),t\ge0,x\in\R^d\}$ to stochastic FDE \eqref{e:fde}. More precisely,
we study temporal, spatial, and joint regularity properties of the solution $u(t,x)$, including the exact uniform and local moduli
of continuity, Chung's LIL, and small ball probabilities. These results generalize and strengthen the related results in
\cite{Xiao2017JDE,Herrell2020sharp,Tudor2017Sample}. To analyze the path regularity, we decompose the solution process
into the sum of two Gaussian random fields: one with stationary increments and the other with smooth paths. For the stationary-increment
term, we derive an explicit expression for its spectral measure and then prove the property of strong local nondeterminism.
For the smooth term, we prove that its sample path is smoother in space and in time, respectively than that of the stationary-increment
term. Under a stronger condition, we show that the sample path of the smooth term is even almost
surely continuously differentiable. Note that our results are consistent with the known results when
\eqref{e:fde} reduces, for instance, to a (fractional) stochastic heat equation.

In recent years, the small ball probability problem for SPDEs with multiplicative noise has attracted a lot of attention in probability
and related areas. Considering the stochastic heat equation with space-time white noise,  Athreya {\em et al.} \cite{Athreya2021Smallball} 
and Foondun {\em et al.} \cite{Foondun2023Smallball} studied the small ball probabilities under the uniform norm and a H\"older
semi-norm, respectively. Khoshnevisan {\em et al.} \cite{khoshnevisan2023small} obtained the existence of a small-ball constant under 
the uniform norm. Recently, Chen \cite{Chen2024Smallball} investigated the small ball probabilities for the stochastic heat equation 
with Gaussian noise which is white in time and colored in space. We also refer to Li and Shao \cite{Li2001Gaussian} for a survey on 
small ball probabilities for general Gaussian processes.

The rest of the paper is organized as follows. We prove the existence and uniqueness of the solution to \eqref{e:fde} in Section~\ref{se:ex}
and then investigate the regularity properties of the sample paths of the solution in Section~\ref{se:re}. In Section~\ref{se:smallball}, we 
obtain small ball probabilities. Appendix \ref{se:appendix} lists some technical lemmas used in the analysis.

\smallskip
Some notations that will be used throughout this paper are collected here. We use $c$ or $C$ to denote a generic constant whose value 
may change in different places. For functions $\varphi(x)$ and $\psi(x)$, we use $(\varphi\ast \psi)(x)$ to denote the convolution. The Fourier
transform is given by $\cF\varphi(\xi)=\int_{\R^d} e^{-ix\cdot \xi}\varphi(x)\ud x$ for $\varphi\in L^1(\R^d)$. We use $\tilde{\varphi}$ to
denote the reflection of $\varphi$ at zero, i.e. $\tilde{\varphi}(x)=\varphi(-x)$. The set of positive integers is denoted by $\N=\{ 1,2,\dots \}$.

\section{Existence of the solution}\label{se:ex}


\subsection{The fundamental solution of the fractional diffusion equation}
Consider the following deterministic fractional diffusion equation:
\begin{equation}\label{e:fde 0}
	\begin{cases}
		\partial^{\beta} v(t, x)=-\left(-\Delta\right)^{\alpha / 2} v(t, x)+ I_{0+}^{\gamma}\left[f(t, x)\right] , & t\in[0,T], x \in \mathbb{R}^d,   \\
		v(0, \cdot)=0,                                                                                       & \text { if } \beta \in(0,1], \\
		v(0, \cdot)=0, \quad \dfrac{\partial}{\partial t} v(0, \cdot)=0,                                 & \text { if } \beta \in(1,2),
	\end{cases}
\end{equation}
where $\alpha>0$, $\beta\in(0,2)$, and $\gamma\in[0,1)$ are constants.
According to  \cite[Theorem 2.8]{chen2024moments}, the fundamental solution to \eqref{e:fde 0} is
\begin{equation}\label{e:de-ptx}
     p(t,x) := \pi^{-d/2}|x|^{-d}t^{\beta+\gamma-1}\times \FoxH{2,1}{2,3}{\frac{ |x|^\alpha}{2^{\alpha} t^\beta}}
       {(1,1),\:(\beta+\gamma,\beta)}{(d/2,\alpha/2),\:(1,1),\:(1,\alpha/2)}, \quad \text{for } t>0,
\end{equation}
where $H^{m,n}_{p,q}(z)$ is the Fox $H$-function (see, e.g.,  \cite{kilbas2004h}) and   $p(t,x)\equiv0$ for $t\le0$.

The Fourier transform
of $p(t, x)$ in space is given by
\begin{equation}\label{eq Fourier p}
  \cF p(t,\cdot)(\xi)=t^{\beta+\gamma-1}E_{\beta,\beta+\gamma}\big(- t^\beta|\xi|^\alpha\big),
\end{equation}
where  $E_{a,b}(z)$ is the Mittag-Leffler function  (see, e.g., \cite[Sect.1.2]{podlubny1998fractional})
\begin{equation}\label{e:mlf}
  E_{a, b}(z):=\sum_{k=0}^{\infty}\frac{z^k}{\Gamma(a k+b)}\,.
\end{equation}

\begin{remark}
In the literature, it is often assumed that $\alpha\in(0,2]$. In this case, the fractional Laplacian $-\left(-\Delta\right)^{\alpha / 2}$ 
is the infinitesimal generator of the $\alpha$-stable process $X:=\{X_t, t\ge0\}$ with characteristic function 
$\E\left[\exp\left(i\langle \xi, X_t\rangle\right)\right]=\exp\left(-t|\xi|^\alpha\right)$. As a consequence, assuming  $\beta=1$ and 
 $\gamma=0$,  the fundamental solution $p(t,x)$ is the transition density of $X$ and hence one may apply probability results 
 of $X$ and/or the semigroup theory to study the corresponding SPDEs.

  In contrast, when $\alpha>2$, the fundamental solution $p(t,x)$  in general is not necessarily nonnegative (see \cite[Remark 1.2]{Chen2023Interpolating} or
  \cite[Remark 5.3]{chen2024moments}) and hence does not correspond to any stochastic process. This may cause extra difficulties in general studies of 
  SPDEs for the case $\alpha>2$. On the other hand, for all $\alpha>0$, the fundamental solution $p(t,x)$  has a unified Fourier transform \eqref{eq Fourier p} 
  and this allows us to study the path properties of the solutions by using Fourier analysis. It is also worth mentioning that this natural extension from $\alpha 
  \in(0,2]$ to all $\alpha>0$ enables us to cover more physical models such as the beam equation ($\alpha=4$).
\end{remark}

For any $z\in\mathbb{C}\setminus\{0\}, a, b, \lambda\in \mathbb C$ and $n\in\mathbb{N}$, the following identity holds (see
\cite[(1.10.7)]{kilbas2006theory}):
\begin{equation}\label{e:differential-E}
  \left(\frac{\partial}{\partial z}\right)^n \left(z^{b-1}E_{a,b}(\lambda z^\alpha) \right)=z^{b-n-1} E_{a,b-n}(\lambda z^\alpha).
\end{equation}

When $a\in(0,2)$, the Mittag-Leffler function has the following asymptotic property (\cite[(1.8.28)]{kilbas2006theory}): for
$n\in\mathbb{N}$,
  \begin{equation}\label{e:asym-MLF}
    E_{a,b}(z) =-\sum_{k=1}^{n}\frac{1}{\Gamma(b-ak)\cdot z^k} +o\left(\frac{1}{z^n}\right), \text{ as } z\to-\infty.
  \end{equation}
We will use the convention (see, e.g.,~\cite[(5.2.1) on p.~136]{Olver2010NIST})
\begin{equation}\label{e:gamma0}
  \frac{1}{\Gamma(z)}\equiv 0, \quad \text{for $z = 0, -1, -2, \dots$}.
\end{equation}
The following property is a direct consequence of \eqref{e:asym-MLF} (see also \cite[Theorem 1.6]{podlubny1998fractional}): for
all $a\in(0,2)$ and $z<0$,
\begin{equation}\label{e:bound-MLF}
  |E_{a,b}(z)|\le \dfrac{\hat{c}}{1+|z|},
\end{equation}
for some positive constant $\hat c$.

\subsection{The solution of the stochastic FDE}

Let $\cD\big(\R\times\R^d\big)$ denote the space of real-valued infinitely differentiable functions with compact support in
$\R\times\R^d$ and let $W$ be a zero-mean Gaussian family $\{W(\varphi),\varphi\in\cD\big(\R\times\R^d\big)\}$ whose
covariance structure is given by
\begin{equation}\label{e:covariance-noise}
  \E[W(\varphi)W(\psi)]=\int_{\R}\int_{\R}\int_{\R^d} \cF\varphi(t,\xi) \overline{\cF\psi(s,\xi)} \Lambda_{H_0}(t-s) \mu(\ud\xi)\ud t\ud s,
\end{equation}
where $\Lambda_{H_0}$ is defined in \eqref{e:def-lambda-h} and the measure $\mu(\ud\xi)=\prod_{j=1}^{d}c_{H_j}|\xi_j|^{1-2H_j}
ud\xi$ with $c_{H_j}=\frac{1}{2\pi}\Gamma(2H_j+1)\sin(\pi H_j)$. Let $\cH$ be the completion of $\cD\big(\R\times\R^d\big)$ with
the inner product
\begin{equation}\label{e:innerproduct-cH}
  \langle \varphi,\psi \rangle_\cH=\E[W(\varphi)W(\psi)].
\end{equation}
Now, the mapping $\varphi\mapsto W(\varphi)$ can be extended to all $\varphi\in \cH$ and it is called the Wiener integral. For
each $\varphi\in\cH$, we also write  $W(\varphi)$ in the form of integral
\begin{equation*}
 W(\varphi)=\int_{\R}\int_{\R^d}\varphi(t, x)W(\ud t,\ud x).
\end{equation*}

We define the solution of the equation \eqref{e:fde} as follows, adopted from \cite{Balan2008Thestochastic}:
\begin{definition}\label{de:solution}
  A jointly measurable process $\{u(t,x);(t,x)\in[0,T]\times\R^d\}$ is said to be a solution of the stochastic Cauchy problem
  \eqref{e:fde}, if for any $\varphi(t,x)\in\cD\big((0,T)\times\R^d\big)$,
  \begin{equation}\label{e:de-solution}
    \int_{0}^{T}\int_{\R^d} u(t,x)\varphi(t,x)\ud t\ud x=  \int_{0}^{T}\int_{\R^d} (\varphi\ast  \tilde{p})(t,x)W(\ud t,\ud x), \quad \text{a.s.}
  \end{equation}
  Here, we recall that $\tilde{p}(t,x):=p(-t,-x)$ and the symbol $\ast$ denotes the convolution.
\end{definition}

In the following theorem, we prove the existence and uniqueness of the solution to \eqref{e:fde}.
\begin{theorem}\label{th:ex-solution}
  Denote $g_{t,x}(s,y):=  p(t-s,x-y)$ and $H:=\sum_{j=1}^{d} H_j$. If the condition
  \begin{align}\label{e:DL}
     \alpha+\dfrac{\alpha}{\beta}\min\left(\gamma+H_0-1,0\right)>d-H
  \end{align}
  holds, then $\Vert g_{t,x} \Vert_\cH<\infty$ for any $t\in[0,T]$ and $x\in\R^d$. In this case, \eqref{e:fde} has a  unique solution
  $\{u(t,x);(t,x)\in[0,T]\times\R^d\}$ given by
  \begin{equation}\label{e:solution}
    u(t,x)=  \int_{0}^{t}\int_{\R^d} p(t-s,x-y)W(\ud s,\ud y).
  \end{equation}
\end{theorem}
\begin{proof}  The uniqueness of the solution is a direct consequence of the definition \eqref{e:de-solution}.  Now we prove that
$u(t,x)$ given by \eqref{e:solution} is well-defined and it satisfies \eqref{e:de-solution}. By using \eqref{e:innerproduct-cH},
\eqref{eq Fourier p}, Lemma~\ref{le:B.2} and  Minkowski’s inequality,   there exists a positive constant $C_{H_0}$ such that
  \begin{align*}
    &\Vert g_{t,x} \Vert_\cH^2 =   \int_{[0,t]^2}\int_{\R^d} \cF p(t-s,x-\cdot)(\xi) \overline{\cF p(t-r,x-\cdot)(\xi)} |s-r|^{2H_0-2} \ud r\ud s\mu(\ud\xi) \\
    &\le C_{H_0}  \int_{\R^d}  \bigg( \int_{0}^{t} \Big| \cF p(t-s,x-\cdot)(\xi)  \Big|^{\frac{1}{H_0}}  \ud s \bigg)^{2H_0}\mu(\ud\xi)     \\
    &\le C_{H_0}\bigg( \int_{0}^{t}\bigg(\int_{\R^d} \Big| \cF p(t-s,x-\cdot)(\xi)  \Big|^2 \mu(\ud\xi) \bigg)^\frac{1}{2H_0}\ud s \bigg)^{2H_0}\\
    &= C_{H_0}\bigg ( \int_{0}^{t}\bigg( (t-s)^{2\beta+2\gamma-2}\int_{\R^d} E_{\beta,\beta+\gamma}^2\big(-(t-s)^\beta|\xi|^\alpha \big)
    \prod_{j=1}^{d}c_{H_j} |\xi_j|^{1-2H_j} \ud\xi \bigg)^\frac{1}{2H_0}\ud s \bigg)^{2H_0}\\
    &= C_{H_0} \bigg ( \int_{0}^{t}\bigg( (t-s)^{2\beta+2\gamma-2-\frac{\beta}{\alpha}(2d-2H)}  \bigg)^\frac{1}{2H_0}\ud s \bigg)^{2H_0}\
    \times \int_{\R^d} E_{\beta,\beta+\gamma}^2\big(-|\xi|^\alpha \big) \prod_{j=1}^{d}c_{H_j} |\xi_j|^{1-2H_j} \ud\xi,
  \end{align*}
  where we have used the change of variables $\tilde{\xi_j}=(t-s)^{\beta/\alpha}\xi_j$ in the last step. By \eqref{e:bound-MLF}, we have
  \begin{align*}
    \Vert g_{t,x} \Vert_\cH^2
    &\le \hat{c}C_{H_0}\bigg ( \int_{0}^{t}\bigg( (t-s)^{2\beta+2\gamma-2-\frac{\beta}{\alpha}(2d-2H)} \bigg)^\frac{1}{2H_0}\ud s \bigg)^{2H_0}
    \times \int_{\R^d} \frac{\prod_{j=1}^{d}c_{H_j} |\xi_j|^{1-2H_j}}{1+|\xi|^{2\alpha}} \ud\xi.
    \end{align*}
Then, applying Lemma \ref{le:space-integral}, we have $\Vert g_{t,x} \Vert_\cH^2<\infty$ under  the condition \eqref{e:DL}.

Next, to validate \eqref{e:de-solution},  we will show that for all $\varphi\in\cD\big((0,T)\times\R^d\big)$,
  \begin{equation*}
    \E\left|\int_{0}^{T}\int_{\R^d} u(t,x)\varphi(t,x)\ud t\ud x- W(\varphi\ast \tilde{p})  \right|^2 =0,
  \end{equation*}
by verifying that
  \begin{align*}
& \E\left|\int_{0}^{T}\int_{\R^d} u(t,x)\varphi(t,x)\ud t\ud x\right|^2= \E \left(W(\varphi\ast \tilde{p})\int_{0}^{T}\int_{\R^d}
u(t,x)\varphi(t,x)\ud t\ud x  \right)= \E\left| W(\varphi\ast \tilde{p})  \right|^2.
  \end{align*}
  Using \eqref{e:innerproduct-cH} and \eqref{e:solution}, we see that
  \begin{equation}\label{e:co-utx-usy}
    \E\left[u(t,x)u(s,y)\right]=\E[W(g_{t,x}) W(g_{s,y})]= \langle g_{t,x},g_{s,y} \rangle_\cH,
  \end{equation}
  and
  \begin{equation}\label{e:co-Wphip-utx}
    \E\left[W(\varphi\ast \tilde{p})u(t,x)\right]=\E\left[W(\varphi\ast \tilde{p})W(g_{t,x})\right]= \langle \varphi\ast \tilde{p},g_{t,x}\rangle_\cH.
  \end{equation}
  By Fubini’s theorem and \eqref{e:co-utx-usy}, we get
  \begin{align*}
    &\E\left|\int_{0}^{T}\int_{\R^d} u(t,x)\varphi(t,x)\ud t\ud x \right|^2\\
    =&\, \int_{[0,T]^2}\int_{\R^{2d}}\varphi(t,x)\varphi(s,y)\langle g_{t,x},g_{s,y} \rangle_\cH \ud s\ud t\ud x\ud y \\
    =&\,\int_{[0,T]^2}\int_{\R^{2d}}\varphi(t,x)\varphi(s,y) \bigg( \int_{[0,T]^2}\int_{\R^d} \cF g_{t,x}(r_1,\xi) \overline{\cF g_{s,y}(r_2,\xi)}
    \Lambda_{H_0}(r_1-r_2) \mu(\ud\xi)\ud r_1\ud r_2 \bigg)\ud s\ud t\ud x\ud y \\
    =&\, \int_{[0,T]^2}\int_{\R^d}\bigg( \int_{[0,T]^2}\int_{\R^{2d}} \varphi(t,x)\varphi(s,y) \cF g_{t,x}(r_1,\xi) \overline{\cF g_{s,y}(r_2,\xi)}
    \ud s\ud t\ud x\ud y\bigg) \Lambda_{H_0}(r_1-r_2) \mu(\ud\xi)\ud r_1\ud r_2\\
   =&\,  \int_{[0,T]^2}\int_{\R^d} \cF (\varphi\ast \tilde{p})(r_1,\xi) \overline{\cF (\varphi\ast \tilde{p})(r_2,\xi)} \Lambda_{H_0}(r_1-r_2)
   \mu(\ud\xi)\ud r_1\ud r_2\\
    =&\, \Vert \varphi\ast \tilde{p}\Vert_\cH^2.
  \end{align*}
  Similarly, by Fubini’s theorem and \eqref{e:co-Wphip-utx}, we also have
  \begin{align*}
     & \E\left[ W(\varphi\ast \tilde{p})\int_{0}^{T}\int_{\R^d} u(t,x)\varphi(t,x)\ud t\ud x \right]\\
     =&\, \int_{0}^{T}\int_{\R^d} \langle \varphi\ast \tilde{p},g_{t,x}\rangle_\cH\: \varphi(t,x)\ud t\ud x\\
    =&\, \int_{0}^{T}\int_{\R^d} \bigg( \int_{[0,T]^2}\int_{\R^d} \cF (\varphi\ast \tilde{p})(r_1,\xi) \overline{\cF g_{t,x}(r_2,\xi)}
    \Lambda_{H_0}(r_1-r_2) \mu(\ud\xi)\ud r_1\ud r_2 \bigg)   \varphi(t,x)\ud t\ud x\\
     =&\,  \int_{[0,T]^2}\int_{\R^d} \cF (\varphi\ast \tilde{p})(r_1,\xi) \bigg(\int_{0}^{T}\int_{\R^d} \overline{\cF g_{t,x}(r_2,\xi)}
     \varphi(t,x)\ud t\ud x \bigg) \Lambda_{H_0}(r_1-r_2) \mu(\ud\xi)\ud r_1\ud r_2\\
     =&\, \int_{[0,T]^2}\int_{\R^d} \cF (\varphi\ast \tilde{p})(r_1,\xi) \overline{\cF (\varphi\ast \tilde{p})(r_2,\xi)} \Lambda_{H_0}(r_1-r_2)
     \mu(\ud\xi)\ud r_1\ud r_2\\
    =&\, \Vert \varphi\ast \tilde{p}\Vert_\cH^2.
  \end{align*}
  Therefore,
  \begin{align*}
     & \E\left|\int_{0}^{T}\int_{\R^d} u(t,x)\varphi(t,x)\ud t\ud x- W(\varphi\ast\tilde{p})  \right|^2 \\
     &=\E\left|\int_{0}^{T}\int_{\R^d} u(t,x)\varphi(t,x)\ud t\ud x\right|^2 -2\E\left[ W(\varphi\ast\tilde{p})\int_{0}^{T}\int_{\R^d} u(t,x)\varphi(t,x)
     \ud t\ud x \right] + \E\left|W(\varphi\ast\tilde{p}) \right|^2\\
     &= \Vert \varphi\ast\tilde{p}\Vert_\cH^2 -2 \Vert \varphi\ast\tilde{p}\Vert_\cH^2+ \E\left|W(\varphi\ast\tilde{p}) \right|^2=0.
  \end{align*}
  This completes the proof.
\end{proof}

\begin{remark}
  When $\beta=1$ and $\gamma=0$, the stochastic FDE \eqref{e:fde} reduces to the fractional stochastic heat equation, and the
  condition \eqref{e:DL} is equivalent to $\alpha H_0 >d-H$ which coincides with \cite[Theorem 2.2]{Herrell2020sharp}. Note that
  the condition \eqref{e:DL}  is also shown to be necessary in \cite[Theorem 2.2]{Herrell2020sharp}  for the fractional stochastic heat
  equation.
\end{remark}

\begin{proposition}
  When the noise is white in time (i.e. $H_0=\frac12$), the condition \eqref{e:DL} is also necessary.
\end{proposition}
\begin{proof}
  When $H_0=\frac12$, by    \eqref{eq Fourier p} and  a change of variables $\tilde{\xi_j}=(t-s)^{\beta/\alpha}\xi_j$, we have
  \begin{align}\label{e:dalang-g-infty}
    \Vert g_{t,x} \Vert_\cH^2 &= \int_{0}^{t}\int_{\R^d} \Big| \cF p(t-s,x-\cdot)(\xi)  \Big|^2 \mu(\ud\xi) \ud s \notag \\
    &=  \int_{0}^{t} (t-s)^{2\beta+2\gamma-2}\ud s \int_{\R^d} E_{\beta,\beta+\gamma}^2\big(-(t-s)^\beta|\xi|^\alpha \big)
    \prod_{j=1}^{d} c_{H_j}|\xi_j|^{1-2H_j} \ud\xi \notag\\
    &= \int_{0}^{t} (t-s)^{2\beta+2\gamma-2-\frac{\beta}{\alpha}(2d-2H)}\ud s \int_{\R^d} E_{\beta,\beta+\gamma}^2\big(-|\xi|^\alpha \big)
    \prod_{j=1}^{d} c_{H_j} |\xi_j|^{1-2H_j} \ud\xi.
  \end{align}
  For the second integral, by the change of variables $\xi_j=rw_j$ in \eqref{e:changevariable}, we have
  \begin{align*}
    &\int_{\R^d} E_{\beta,\beta+\gamma}^2\big(-|\xi|^\alpha \big) \prod_{j=1}^{d} c_{H_j}|\xi_j|^{1-2H_j} \ud\xi\\
    &=\int_{0}^{\infty} E_{\beta,\beta+\gamma}^2\big(- r^\alpha \big) r^{2d-2H-1}\ud r \times \int_{\mathbb{S}^{d-1}}
     \prod_{j=1}^{d}c_{H_j}|w_j|^{1-2H_j} \sigma(\ud w),
  \end{align*}
  where $\sigma(\ud w)$ is the uniform measure on the $(d-1)$-dimensional unit sphere $\mathbb{S}^{d-1}$ and the second integral
  is finite due to Lemma \ref{le:space-integral}. Then, by \eqref{e:dalang-g-infty}, we know that $\Vert g_{t,x} \Vert_\cH^2<\infty$ is equivalent to
  \begin{equation}\label{e:necess1}
    \begin{cases}
     \displaystyle 2\beta+2\gamma-2-\frac{\beta}{\alpha}(2d-2H)>-1,\\%
      \displaystyle \int_{0}^{\infty} E_{\beta,\beta+\gamma}^2\big(- r^\alpha \big) r^{2d-2H-1}\ud r<\infty.
    \end{cases}
  \end{equation}
  Notice that $E_{\beta,\beta+\gamma}^2\big(- |\cdot|^\alpha \big) |\cdot|^{2d-2H-1}$ is locally integrable. By the asymptotic property
  \eqref{e:asym-MLF}, as $r\to\infty$,
  \begin{equation}\label{e:asym-necess}
    E_{\beta,\beta+\gamma}\big(- r^\alpha \big)= \frac{1}{\Gamma(\gamma) r^\alpha} + o(r^{-1}).
  \end{equation}

  For the case $\gamma>0$, the condition \eqref{e:necess1} is equivalent to
  \begin{equation}\label{e:necess2}
    \begin{cases}
      2\alpha+\frac{\alpha}{\beta}(2\gamma-1)>2d-2H,  \\
      2\alpha>2d-2H,
    \end{cases}
  \end{equation}
  which is straightforward by \eqref{e:asym-necess}.

  For the case $\gamma=0$, the condition \eqref{e:necess1} is equivalent to
  \begin{align}\label{e:necess3}
     & 2\alpha-\frac{\alpha}{\beta}>2d-2H,
  \end{align}
  where the equivalence follows from the fact that the first condition $2\alpha-\frac{\alpha}{\beta}>2d-2H$ in \eqref{e:necess1}
  implies the second condition $\int_{0}^{\infty} E_{\beta,\beta+\gamma}^2\big(- r^\alpha \big) r^{2d-2H-1}\ud r<\infty$ by
  \eqref{e:asym-necess} and \eqref{e:gamma0}.

  Therefore, combining \eqref{e:necess2} with \eqref{e:necess3}, we obtain that  the condition \eqref{e:necess1} is equivalent to
  \begin{align*}
   2\alpha+\frac{\alpha}{\beta}\min\left(2\gamma-1,0\right)>2d-2H.
  \end{align*}
  This completes the proof.
\end{proof}

\section{Regularity of the solution}\label{se:re}

In this section, we investigate the sample path regularity of the solution to stochastic FDE \eqref{e:fde} in space-time
 (Theorem \ref{th:modulus-both}), in time (Theorem \ref{th:modulus-time}), and in space (Theorem \ref{th:modulus-space}),
  respectively. 

\subsection{Space-time joint regularity}\label{se:joint-r}

In this subsection, we study the space-time joint regularity of $u(t,x)$, see Theorem \ref{th:modulus-both}. Similar to
the string processes studied in  Mueller and Tribe \cite{Mueller2002Hitting} (see also \cite{Herrell2020sharp,Tudor2017Sample}),
 the solution $\{u(t,x), (t, x) \in [0, \infty)\times  \R^d\}$ has the following decomposition
\begin{equation}\label{e:split-u}
  u(t,x)=U(t,x)-V(t,x),
\end{equation}
where for $ (t, x) \in [0, \infty)\times  \R^d$
\begin{equation}\label{e:de-U}
\begin{split}
  U(t,x)  = & \int_{-\infty}^{0}\int_{\R^d} \left[p(t-r,x-y)-p(-r,-y)\right] W(\ud r,\ud y)\\
   &+\int_{0}^{t}\int_{\R^d} p(t-r,x-y) W(\ud r,\ud y)
\end{split}
\end{equation}
and
\begin{equation}\label{e:de-V}
 V(t,x)=\int_{-\infty}^{0}\int_{\R^d}  [p(t-r,x-y)-p(-r,-y)] W(\ud r,\ud y).
\end{equation}
Notice that if we rewrite $U(t,x) $ in \eqref{e:de-U} as
\begin{equation}\label{e:de-U2}
U(t,x) = \int_{\R}\int_{\R^d} [p\left((t-r)_+,x-y\right)-p\left((-r)_+,-y\right)] W(\ud r,\ud y),
\end{equation}
where $a_+ = \max\{0, a\}$ for every $a \in \R$, then  $U(t, x)$ is well defined for all $(t, x)  \in \R \times \R^d$. 

We will show that the sample path of the random field $V(t,x)$ is smoother than that of $U(t,x)$ on $(0,\infty)\times\R^d$,
hence the regularity properties of $u(t,x)$ on  $(0,\infty)\times\R^d$ solely depend on those of $U(t,x)$. To derive the sharp
regularity, we first prove that the random field $\{U(t,x), (t, x)  \in \R \times \R^d\}$ has stationary increments and it satisfies the property
of strong local nondeterminism (SLND, for short) (see Lemma \ref{le:sLND-U}). Then, we show the sample path of $V(t,x)$ is
smoother than that of $U(t,x)$ on $(0,\infty)\times\R^d$ in Proposition \ref{prop:reg-V} and Proposition \ref{prop:reg-V-add}.
Finally, we obtain the joint regularity in Theorem \ref{th:modulus-both} for the solution $u(t,x)$.

To study the regularity properties of the Gaussian random field $\{U(t,x), (t, x) \in \R \times \R^d\}$
we recall some facts about Gaussian random fields with stationary increments.  A real-valued  random
field $X=\{X(\t), \t \in\R^N\}$ is
said to have stationary increments if for every $\h\in\R^N$,
  \begin{equation*}
  \{X(\t+h)-X(\h), \, \t \in\R^N\} \overset{d}{=} \{X(\t) - X(0),\, \t \in\R^N\},
  \end{equation*}
  where $\overset{d}{=}$ means equality in distribution.  According to Yaglom \cite{yaglom1957some},  if $X$ satisfies $X(0)=0$
  and its covariance function $R(\t,\t')=\E[X(\t )X(\t')]$ is continuous,  then $R(\t,\t')$ can be represented as
  \begin{equation}\label{Eq:R}
  R(\t, \t')= \int_{\R^N}\left(e^{i\langle \t,\xi\rangle}-1\right)\left(e^{-i\langle \t',\xi\rangle}-1\right) F(\ud\xi)+\langle \tau,Q\tau' \rangle,
  \end{equation}
  where $Q=(q_{ij})$ is an $N\times N$ non-negative definite matrix and $F(\ud\xi)$  is a nonnegative symmetric measure on
  $\R^d\setminus\{0\}$ satisfying
  \begin{equation}\label{Eq:F}
    \int_{\R^N}\frac{|\xi|^2}{1+|\xi|^2}F(\ud\xi)<\infty.
  \end{equation}
The measure $F$ and its density $f(\xi)$ (if it exists) are called the spectral measure and spectral density of $X$, respectively.
The converse is also true. Namely, if a Gaussian random field $X$ satisfies $X(0) = 0$ and its covariance function is given by (\ref{Eq:R})
with $F$ satisfying (\ref{Eq:F}), then $X$ has stationary increments with spectral measure $F$. We will apply these facts to
$\{U(t,x), (t, x) \in \R \times \R^d\}$ defined by \eqref{e:de-U2}.  In this case, $\t = (t, x)$ and $N = 1+d$.

Denote
\begin{equation}\label{e:theta12}
\begin{split}
 & \theta_1:=\beta+\gamma+H_0-1-\frac{\beta}{\alpha}(d-H);\\
 & \theta_2:=\frac\alpha\beta\theta_1=\alpha-d+H+\frac{\alpha}{\beta}(\gamma+H_0-1).
 \end{split}
 \end{equation}
Clearly, both $\theta_1$ and $\theta_2$ are strictly positive under the condition \eqref{e:DL}.

The following proposition is a key ingredient to study the sample path regularity of $U(t,x)$.
\begin{lemma}\label{le:sLND-U}
Assume the condition  \eqref{e:DL} holds and
\begin{equation}\label{e:U-exist}
  \begin{cases}
    \theta_1<1,  \\
    \theta_2<1,  \\
    \gamma<1-H_0.
  \end{cases}
\end{equation}
\begin{enumerate}
  \item The Gaussian random field $U=\{U(t,x),(t,x)\in\R\times\R^d\}$ has stationary increments with spectral measure
  \begin{equation}\label{e:spectral-measure-U}
  F_U(\ud\tau,\ud\xi)=\frac{C_{H_0}|\tau|^{1-2\gamma-2H_0}}{|\tau|^{2\beta}+2|\xi|^\alpha|\tau|^\beta\cos\left(\frac{\pi\beta}{2}\right)
  +|\xi|^{2\alpha}} \mu(\ud\xi) \ud \tau.
  \end{equation}
  \item The Gaussian random field $U$ satisfies the strong local nondeterminism property: there exists
  a constant $C>0$ such that for all $n\in\N$ and $(t,x),(t^1,x^1),\dots,(t^n,x^n)\in\R\times\R^d$,
  \begin{equation}\label{e:sLND}
    \Var\big(U(t,x)|U(t^1,x^1),\dots,U(t^n,x^n)  \big)\ge C\min_{k=0,1,\dots,n}\bigg( |t-t^k|^{\theta_1}+\sum_{i=1}^d|x_i-x^{k}_i|^{\theta_2} \bigg)^2,
  \end{equation}
  where $t^0=x^0_i=0$, $i=1,\dots, d$.
  \item 
  There exist constants $c, C>0$ such that
  \begin{equation}\label{e:bounds-U}
    c\left(|t-s|^{2\theta_1}+|x-y|^{2\theta_2}\right) \le \E\left[ |U(t,x)-U(s,y)|^2 \right]\le C\left(|t-s|^{2\theta_1}+|x-y|^{2\theta_2}\right)
  \end{equation}
  for all $(t,x),(s,y)\in\R\times\R^d$.
\end{enumerate}
\end{lemma}

\begin{proof}
  {\em Step 1:} Let $(t,x),(s,y)\in\R\times\R^d$. By \eqref{e:covariance-noise} and \eqref{e:de-U}, we have
  \begin{align*}
    &\E\left[ |U(t,x)-U(s,y)|^2 \right] \\
    &=\E\Bigg[ \bigg|\int_{\R}\int_{\R^d} [p\left((t-r)_+,x-z\right)-p\left((s-r)_+,y-z\right)] W(\ud r,\ud z)\bigg|^2 \Bigg]\\
    &=  \int_{\R^2}   |r_1-r_2|^{2H_0-2} \ud r_1\ud r_2  \int_{\R^d} \cF\left[  p\left((t-r_1)_+,x-\cdot\right)-p\left((s-r_1)_+,y-\cdot\right)\right](\xi)\\
    &\quad \times \overline{ \cF\left[  p\left((t-r_2)_+,x-\cdot\right)-p\left((s-r_2)_+,y-\cdot\right)\right](\xi)} \mu(\ud\xi) ,
  \end{align*}
  where by \eqref{eq Fourier p},
  \begin{equation*}
    \cF p\left((t-r)_+,x-\cdot\right)(\xi)=e^{-i\langle x,\xi\rangle}(t-r)^{\beta+\gamma-1}E_{\beta,\beta+\gamma}\big( - (t-r)^\beta|\xi|^\alpha \big)\one_{\{t-r>0\}}.
  \end{equation*}
Denote
  \begin{align*}
     \phi(r,\xi) &:= \cF\left[  p\left((t-r)_+,x-\cdot\right)-p\left((s-r)_+,y-\cdot\right)\right](\xi) \\
     & =e^{-i\langle x,\xi\rangle}(t-r)^{\beta+\gamma-1}E_{\beta,\beta+\gamma}\big( - (t-r)^\beta|\xi|^\alpha \big)\one_{\{t-r>0\}}\\
     &\quad -e^{-i\langle y,\xi\rangle}(s-r)^{\beta+\gamma-1}E_{\beta,\beta+\gamma}\big( - (s-r)^\beta|\xi|^\alpha \big)\one_{\{s-r>0\}}.
  \end{align*}
  Using the Parseval–Plancherel identity, we have
  \begin{align*}
    &\E\left[ |U(t,x)-U(s,y)|^2 \right] \\
    &= \int_{\R^2}|r_1-r_2|^{2H_0-2}\ud r_1\ud r_2 \int_{\R^d} \phi(r_1,\xi) \overline{\phi(r_2,\xi)}\mu(\ud\xi)\\
    &=C_{H_0}\int_{\R^d} \mu(\ud\xi) \int_\R \left| \hat{\phi}(\tau,\xi) \right|^2 |\tau|^{1-2H_0}\ud \tau,
  \end{align*}
with
  \begin{align*}
       \hat{\phi}(\tau,\xi) & :=\int_\R e^{-ir\tau}\phi(r,\xi)\ud r\\
     &= e^{-i\langle x,\xi\rangle}\int_{-\infty}^{t} e^{-ir\tau}(t-r)^{\beta+\gamma-1}E_{\beta,\beta+\gamma}\big( - (t-r)^\beta|\xi|^\alpha \big)\ud r\\
     &\quad-e^{-i\langle y,\xi\rangle}\int_{-\infty}^{s} e^{-ir\tau}(s-r)^{\beta+\gamma-1}E_{\beta,\beta+\gamma}\big( - (s-r)^\beta|\xi|^\alpha \big)\ud r\\
     &= e^{-i(\langle x,\xi\rangle+\tau t)}\int_{0}^{\infty} e^{iv\tau} v^{\beta+\gamma-1}E_{\beta,\beta+\gamma}\big( - v^\beta|\xi|^\alpha \big)\ud v\\
     &\quad-e^{-i(\langle y,\xi\rangle+\tau s)}\int_{0}^{\infty} e^{iv\tau} v^{\beta+\gamma-1}E_{\beta,\beta+\gamma}\big( - v^\beta|\xi|^\alpha \big)\ud v\\
     &= \left(e^{-i(\langle x,\xi\rangle+\tau t)}-e^{-i(\langle y,\xi\rangle+\tau s)}\right) \frac{(-i\tau)^{-\gamma}}{(-i\tau)^\beta+|\xi|^\alpha},
  \end{align*}
  where we have used Lemma \ref{le:Fourier-MLF} in the last step. Then
  \begin{align}\label{e:E-utx-usy}
    &\E\left[ |U(t,x)-U(s,y)|^2 \right] \notag\\
    &=C_{H_0} \int_{\R^{d+1}}\left| e^{-i(\langle x,\xi\rangle+\tau t)}-e^{-i(\langle y,\xi\rangle+\tau s)}\right|^2 \left|\frac{(-i\tau)^{-\gamma}}
    {(-i\tau)^\beta+|\xi|^\alpha}\right|^2|\tau|^{1-2H_0}\mu(\ud\xi)\ud \tau \notag\\
    &= C_{H_0} \int_{\R^{d+1}}\left| e^{-i(\langle x,\xi\rangle+\tau t)}-e^{-i(\langle y,\xi\rangle+\tau s)}\right|^2 \frac{|\tau|^{-2\gamma}} 
    {|\tau|^{2\beta}+2|\xi|^\alpha|\tau|^\beta\mathfrak{R}[(-i)^\beta] +|\xi|^{2\alpha}} |\tau|^{1-2H_0} \mu(\ud\xi) \ud \tau \notag\\
    &=2C_{H_0} \int_{\R^{d+1}}\big[ 1-\cos\big( \langle x-y,\xi\rangle+(t-s)\tau \big) \big] \frac{|\tau|^{1-2\gamma-2H_0}}
    {|\tau|^{2\beta}+2|\xi|^\alpha|\tau|^\beta\cos\left(\frac{\pi\beta}{2}\right)+|\xi|^{2\alpha}} \mu(\ud\xi) \ud \tau,
  \end{align}
  which only depends on $t-s$ and $x-y$.  Therefore, $U(t,x)$ has stationary increments with the spectral measure given by \eqref{e:spectral-measure-U}.

  {\em Step 2:}  Next, we prove the SLND property. By \eqref{e:spectral-measure-U}, we have the spectral density
  \begin{equation*}
    f_U(\tau,\xi)= \frac{2C_{H_0}|\tau|^{1-2\gamma-2H_0}\prod_{j=1}^{d}(c_j|\xi_j|^{1-2H_j})} {|\tau|^{2\beta}
    +2|\xi|^\alpha|\tau|^\beta\cos\left(\frac{\pi\beta}{2}\right)+|\xi|^{2\alpha}}.
  \end{equation*}
  Then for every $c>0$, we have 
  \begin{align*}
    &f_U(c^{1/\theta_1}\tau,c^{1/\theta_2}\xi)\\
     = &\, c^{\frac{1-2\gamma-2H_0}{\theta_1}+\frac{d-2H}{\theta_2}} \frac{2C_{H_0}|\tau|^{1-2\gamma-2H_0}\prod_{j=1}^{d}(c_j|\xi_j|^{1-2H_j})} 
     {c^{\frac{2\beta}{\theta_1}}|\tau|^{2\beta}+c^{\frac{\beta}{\theta_1}+\frac{\alpha}{\theta_2}}2|\xi|^\alpha|\tau|^\beta\cos\left(\frac{\pi\beta}{2}\right) 
     +c^\frac{2\alpha}{\theta_2}|\xi|^{2\alpha}}\\
     =&\, c^{\frac{1-2\gamma-2H_0}{\theta_1}+\frac{d-2H}{\theta_2}-\frac{2\beta}{\theta_1}}f_U(\tau,\xi)\\
    =&\, c^{\frac{2-2\beta-2\gamma-2H_0}{\theta_1}+\frac{2d-2H}{\theta_2}-\frac{1}{\theta_1}-\frac{d}{\theta_2}}f_U(\tau,\xi)
    =c^{-(2+Q)}f_U(\tau,\xi),
  \end{align*}
  where $Q=\frac{1}{\theta_1}+\frac{d}{\theta_2}$ with $\theta_1$ and $\theta_2$ being given in \eqref{e:theta12}.   The SLND property follows
  directly from Theorem 4.1 in \cite[Appendix A]{Herrell2020sharp}.

  {\em Step 3:} By \eqref{e:E-utx-usy}, using the change of variables $\xi =|\tau|^{\frac\beta\alpha} \tilde{\xi}$ and $\tilde\tau=(t-s)\tau$,
 and then changing the notations $\tilde{\xi}$ and  $\tilde\tau$ back to $\xi$ and  $\tau$, we have
  \begin{align}\label{e:proof-spectral-tilde-U}
 &\E\left[ |U(t,x)-U(s,x)|^2 \right]\notag\\
&=2C_{H_0}\int_{\R^{d+1}}\big[ 1-\cos\big( (t-s)\tau \big) \big] \frac{|\tau|^{1-2\gamma-2H_0}\prod_{j=1}^{d}c_{H_j}|\xi_j|^{1-2H_j}}{|\tau|^{2\beta}+2|\xi|^\alpha|\tau|^\beta\cos\left(\frac{\pi\beta}{2}\right)+|\xi|^{2\alpha}} \ud\xi \ud \tau\notag\\
   &= 2C_{H_0}\int_{\R^{d+1}} \frac{ 1-\cos\big( (t-s)\tau \big) }{|\tau|^{2\beta+2\gamma+2H_0-1-\frac\beta\alpha(2d-2H)}} \frac{\prod_{j=1}^{d}c_{H_j} |\xi_j|^{1-2H_j}}{1+2|\xi|^\alpha\cos\left(\frac{\pi\beta}{2}\right)+|\xi|^{2\alpha}}\ud\xi\ud \tau\\
   &= 2C_{H_0}|t-s|^{2\theta_1}  \int_{\R}\frac{1-\cos\big(\tau \big)}{|\tau|^{2\theta_1+1}}\ud \tau\times \int_{\R^d}\frac{\prod_{j=1}^{d}c_{H_j}|\xi_j|^{1-2H_j}}{1+2|\xi|^\alpha\cos\left(\frac{\pi\beta}{2}\right)+|\xi|^{2\alpha}}\ud\xi, \notag
  \end{align}
  where the first integral is finite owing to $\theta_1\in(0,1)$ and the second integral
  \begin{equation}\label{e:time-rgular-2-integral}
    \begin{split}
        &\int_{\R^d}\frac{\prod_{j=1}^{d}c_{H_j}|\xi_j|^{1-2H_j}}{1+2|\xi|^\alpha\cos\left(\frac{\pi\beta}{2}\right)+|\xi|^{2\alpha}}\ud\xi\\
     &\le \int_{\R^d}\frac{\prod_{j=1}^{d}c_{H_j}|\xi_j|^{1-2H_j}}{\left(1+|\xi|^{2\alpha}\right)\wedge
     \left(\left[1+\cos\left(\frac{\pi\beta}{2}\right)\right]\times\left(1+|\xi|^{2\alpha}\right) \right)}\ud\xi
     \\
     &\le \frac{1}{1\wedge\left[1+\cos\left(\frac{\pi\beta}{2}\right)\right]}\times\int_{\R^d}\frac{\prod_{j=1}^{d}c_{H_j}|\xi_j|^{1-2H_j}}{1+|\xi|^{2\alpha}}\ud\xi,
    \end{split}
  \end{equation}
   which is also finite  noting that $\alpha>d-H$ under \eqref{e:DL} by Lemma \ref{le:space-integral} (i).  Then,
    \begin{equation}\label{e:up-bound-t}
      \E\left[ |U(t,x)-U(s,x)|^2 \right] = C_1 |t-s|^{2\theta_1},
    \end{equation}
 where
 \begin{equation*}
 C_1=2C_{H_0}\times\int_{\R}\frac{1-\cos\big(\tau \big)}{|\tau|^{2\theta_1+1}}\ud \tau\times \int_{\R^d}\frac{\prod_{j=1}^{d}c_{H_j}|\xi_j|^{1-2H_j}}
 {1+2|\xi|^\alpha\cos\left(\frac{\pi\beta}{2}\right)+|\xi|^{2\alpha}}\ud\xi.
    \end{equation*}

By  \eqref{e:E-utx-usy} and using the change of variables ${\tau}=|\xi|^\frac{\alpha}{\beta}\tilde\tau$, we get
  \begin{align*}
 &\E\left[ |U(t,x)-U(t,y)|^2 \right]\\
 &=2C_{H_0}\int_{\R^{d+1}}\big[ 1-\cos\big( \langle x-y,\xi\rangle \big) \big] \frac{|\tau|^{1-2\gamma-2H_0}\prod_{j=1}^{d}c_{H_j}|\xi_j|^{1-2H_j}}
    {|\tau|^{2\beta}+2|\xi|^\alpha|\tau|^\beta\cos\left(\frac{\pi\beta}{2}\right)+|\xi|^{2\alpha}} \ud\xi \ud \tau\\
 &=2C_{H_0}\int_{\R^{d+1}}\big[ 1-\cos\big( \langle x-y,\xi\rangle \big) \big] \frac{\prod_{j=1}^{d}c_{H_j}|\xi_j|^{1-2H_j}}
 {|\xi|^{2\alpha+\frac\alpha\beta(2\gamma+2H_0-2)}} \frac{|\tau|^{1-2\gamma-2H_0}}{|\tau|^{2\beta}+2|\tau|^\beta\cos\left(\frac{\pi\beta}{2}\right)+1} \ud\xi \ud \tau \\
    &=2C_{H_0} \int_{\R^d} \frac{\big[ 1-\cos\big( \langle x-y,\xi\rangle \big) \big]\prod_{j=1}^{d}c_{H_j}|\xi_j|^{1-2H_j}}{|\xi|^{2\alpha+\frac\alpha\beta(2\gamma+2H_0-2)}} \ud\xi
     \int_{\R} \frac{|\tau|^{{1-2\gamma-2H_0}}}{|\tau|^{2\beta}+2|\tau|^\beta\cos\left(\frac{\pi\beta}{2}\right)+1}\ud \tau.
  \end{align*}
  The last condition in \eqref{e:U-exist} implies $1-2\gamma-2H_0>-1$ and \eqref{e:DL} implies $2\beta+2H_0+2\gamma-1>1$. Thus,
  the last integral in the above equation is finite, namely,
  \begin{equation}\label{c31}
   \begin{split}
     c_{3,1} &:=\int_{\R} \frac{|\tau|^{{1-2\gamma-2H_0}}}{|\tau|^{2\beta}+2|\tau|^\beta\cos\left(\frac{\pi\beta}{2}\right)+1}\ud \tau \\
  &\le\frac{1}{1\wedge\left[1+\cos\left(\frac{\pi\beta}{2}\right)\right]}\int_{\R} \frac{|\tau|^{{1-2\gamma-2H_0}}}{|\tau|^{2\beta}+1}\ud \tau<\infty.
   \end{split}
  \end{equation}
  Then, using changes of variables $\xi_j=rw_j$ in \eqref{e:changevariable} and $w_{x-y}=\frac{x-y}{|x-y|}$, we have
  \begin{align}\label{e:proof-spectral-hat-U}
    &\E\left[ |U(t,x)-U(t,y)|^2 \right]\notag\\
    &=2C_{H_0} c_{3,1}\int_{\R^d} \frac{\big[ 1-\cos\big( \langle x-y,\xi\rangle \big) \big]\prod_{j=1}^{d}c_{H_j}|\xi_j|^{1-2H_j}}
    {|\xi|^{2\alpha+\frac\alpha\beta(2\gamma+2H_0-2)}} \ud\xi\\
    &= 2C_{H_0} c_{3,1}\int_{\mathbb{S}^{d-1}} \prod_{j=1}^{d}c_{H_j}|w_j|^{1-2H_j} \sigma(\ud w) \int_0^\infty \frac{ 1-\cos\big( r |x-y|\langle w_{x-y}, w\rangle\big)} {r^{2\alpha-2d+2H+1+\frac\alpha\beta(2\gamma+2H_0-2)}} \ud r\notag\\
    &= 2C_{H_0} c_{3,1}\int_{\mathbb{S}^{d-1}} \prod_{j=1}^{d}c_{H_j}|w_j|^{1-2H_j} \sigma(\ud w) \int_0^\infty \frac{ 1-\cos\big( r|x-y|\langle w_{x-y}, w\rangle \big)}{r^{2\theta_2+1}} \ud r\notag\\
    &= 2C_{H_0} c_{3,1}|x-y|^{2\theta_2}\int_{\mathbb{S}^{d-1}} \prod_{j=1}^{d}c_{H_j}|w_j|^{1-2H_j} \sigma(\ud w) \int_0^\infty \frac{ 1-\cos\big( r \langle w_{x-y}, w\rangle\big)}{r^{2\theta_2+1}} \ud r.\notag
  \end{align}
 Note that  $\langle w_{x-y}, w\rangle\in[-d,d]$
  and
  \begin{align*}
    &\int_{\mathbb{S}^{d-1}} \prod_{j=1}^{d}c_{H_j}|w_j|^{1-2H_j} \sigma(\ud w) \int_0^\infty \frac{ 1-\cos\big( r \langle w_{x-y}, w\rangle\big)}{r^{2\theta_2+1}} \ud r\\
    &\le \int_{\mathbb{S}^{d-1}} \prod_{j=1}^{d}c_{H_j}|w_j|^{1-2H_j} \sigma(\ud w) \left(\int_{0}^{1/d}\frac{ 1-\cos\big( r d \big)}{r^{2\theta_2+1}}\ud r+ \int_{1/d}^{\infty}\frac{2}{r^{2\theta_2+1}}\ud r \right)=:c_{3,2},
  \end{align*}
  which is finite due to $\theta_2\in(0,1)$ and Lemma \ref{le:space-integral}. Note that $c_{3,2}$ does not depend on $x-y$.  Thus,
  \begin{equation}\label{e:up-bound-x}
    \E\left[ |U(t,x)-U(t,y)|^2 \right]\le C |x-y|^{2\theta_2}.
  \end{equation}

  Combining \eqref{e:up-bound-t} with \eqref{e:up-bound-x}, we have
  \begin{equation*}
    \E\left[ |U(t,x)-U(s,y)|^2 \right]\le C\left(|t-s|^{2\theta_1}+|x-y|^{2\theta_2}\right),
  \end{equation*}
  and $U(t,x)$ is well defined since $U(0,0)=0$. The lower bound in \eqref{e:bounds-U} is a direct consequence of the
  SLND property \eqref{e:sLND}. The proof is complete.
\end{proof}

\begin{remark}[On the condition for the parameter $\gamma$]
To obtain the existence and uniqueness of the solution to \eqref{e:fde},  Dalang's condition \eqref{e:DL}  in
Theorem \ref{th:ex-solution} does not require an upper bound for the parameter $\gamma$, which is natural since the
Riemann-Liouville integral $I_{0_+}^\gamma$  makes the noise $\dot W$ more regular with bigger $\gamma$. On the
other hand, we need to impose an upper bound for $\gamma$ in the study of the regularity of the solution, due to the following 
technical reasons:
\begin{itemize}
\item[(1)]
Our calculation relies on the Mittag-Leffler Fourier transform \eqref{e:fourier-MLF}, which holds only when $\gamma<1$.

 \item[(2)] Instead of dealing with $u(t,x)$ directly, we study the regularity of the increment-stationary field $
 \{U(t,x), (t, x) \in \R\times \R^d\}$ given in \eqref{e:de-U}.
To ensure that $U(t,x)$ is well-defined, we need the three conditions in \eqref{e:U-exist} so that the integrals in
\eqref{e:proof-spectral-tilde-U},  \eqref{c31}, and \eqref{e:proof-spectral-hat-U} are convergent.
 \end{itemize}
\end{remark}

Now, we are ready to study the regularity of $U(t,x)$.

\begin{proposition}\label{prop:modulus-both-U}
   Assume the conditions of Lemma \ref{le:sLND-U} hold.  Let  $I=[-M,M]^{d+1}$ with $0<M<\infty$ and let   $\rho(t,x)=|t|^{\theta_1}+|x|^{\theta_2}$.
  \begin{enumerate}
    \item (Uniform modulus of continuity). There exists a constant $k_1\in(0,\infty)$ such that
    \begin{equation*}
      \lim_{\e\to0^+}\sup_{\substack{(t,x),(s,y)\in I;\\ \rho(t-s,x-y)\le\e}}\dfrac{|U(t,x)-U(s,y)|}{\rho(t-s,x-y)\sqrt{\log\left(1+\rho(t-s,x-y)^{-1}\right)}}=k_1,
      \text{ a.s. }
    \end{equation*}
    \item (Local modulus of continuity). There exists a constant $k_2\in(0,\infty)$ such that for any $(t,x)\in I$,     \begin{equation*}
      \lim_{\e\to0^+}\sup_{\rho(s,y)\le \e} \dfrac{|U(t+s,x+y)-U(t,x)|}{\rho(s,y)\sqrt{\log\log\left(1+\rho(s,y)^{-1}\right)}}=k_2,
      \text{ a.s. }
    \end{equation*}
   \item (Chung's LIL). There exists a constant $k_3\in(0,\infty)$ such that for all $(t,x)\in I$,
  \begin{equation*}
    \liminf_{\e\to0^+}\sup_{\rho(s,y)\le \e} \dfrac{|U(t+s,x+y)-U(t,x)|}{\e\left(\log\log 1/\e\right)^{-1/Q}}=k_3, \text{ a.s. }
  \end{equation*}
  where $Q=\frac{1}{\theta_1}+\frac{d}{\theta_2}$.
  \end{enumerate}
\end{proposition}

\begin{proof}
  Thanks to Lemma \ref{le:sLND-U}, by using \cite[Theorem 4.1, Theorem 5.6]{Xiao2013Fernique} for modulus of continuity
  and \cite[Theorem 1.1]{Xiao2010Chung} for Chung's LIL,  we can derive the regularity results for $U(t,x)$.
\end{proof}

Now, we study the regularity of $V(t,x)$. The following proposition shows that the sample path of $V(t,x)$ is smoother 
than $U(t,x)$ in $(0,\infty)\times\R^d$.

\begin{proposition}\label{prop:reg-V}
Assume the conditions of Lemma \ref{le:sLND-U} hold. Then, $\{V(t,x),t\ge0,x\in\R^d\}$ is well-defined.
\begin{enumerate}
  \item[(1)] There exists a finite constant $c > 0$  such that for all $x\in\R^d$ and $0< s < t<\infty$, 
  \begin{equation}\label{e:d1-pre}
    \E\left[|V(t,x)-V(s,x)|^2\right]^{1/2}\le
    \begin{cases}
      \dfrac{c|t-s|}{s^{1-\theta_1}}, & \mbox{if } \beta+\gamma-2\le0, \\
      \dfrac{c\: t^{\beta+\gamma-2}|t-s|}{s^{-H_0+\frac{\beta}{\alpha}(d-H)}}, & \mbox{if }\beta+\gamma-2>0 .
    \end{cases}
  \end{equation}
   In particular, for any $0<a<b<\infty$, there exists a finite constant $C > 0$ such that for all $x\in\R^d$ and 
   $s,t\in[a,b]$, 
  \begin{equation}\label{e:upb-vtx-vsx}
  \E\left[ \left|V(t,x)-V(s,x)\right|^2 \right] \le C |t-s|^2.
  \end{equation}
   Consequently,  for any $x\in\R^d$, the  function $t\mapsto V(t,x)$ is smoother than $t\mapsto U(t,x)$ on $[a,b]$.

  \item[(2)]
    There exists a finite constant $C > 0$ such that for all $t>0$ and $x,y\in\R^d$ with $|x-y|\le e^{-1}$,
  \begin{equation}\label{e:upb-vtxvty}
    \begin{split}
       &\E\left[ \left|V(t,x)-V(t,y)\right|^2 \right] \\
         &\qquad\le
    \begin{cases}
      C\:t^{2H_0+2\gamma-2}\times|x-y|^{2\alpha-2d+2H}, & \mbox{if } 0<\alpha-d+H<1, \\
       C\:t^{2H_0+2\gamma-2}\times|x-y|^2 \log\left(|x-y|^{-1}\right), & \mbox{if } \alpha-d+H=1, \\
      C\:t^{2\beta+2\gamma+2H_0-2-\frac{\beta(2d-2H+2)}{\alpha}}\times|x-y|^2, & \mbox{if } \alpha-d+H>1.
    \end{cases}
    \end{split}
  \end{equation}
Consequently, for any $t>0$,  the  function $x\mapsto V(t,x)$ is smoother than $x\mapsto U(t,x)$ on $\R^d$.
\end{enumerate}
\end{proposition}

\begin{proof}
By \eqref{e:split-u}, we have that for $(t,x)\in[0,\infty)\times\R^d$,
 \begin{equation*}
   \E[|V(t,x)|^2]=\E[|U(t,x)-u(t,x)|^2]\le 2\E[|U(t,x)|^2]+2\E[|u(t,x)|^2]<\infty,
 \end{equation*}
 under conditions \eqref{e:DL} and \eqref{e:U-exist}, due to Theorem \ref{th:ex-solution} and Lemma \ref{le:sLND-U}. Then, 
 $\{V(t,x),t\ge0,x\in\R^d\}$ is well-defined.

 (1).  We prove the regularity of $t\mapsto V(t,x)$ on $(0, T]$. Using \eqref{e:innerproduct-cH}, \eqref{e:de-V}, Lemma \ref{le:B.2} and  
 Minkowski's inequality, we have
    \begin{align}\label{e:vtx-vsy0}
    &\E\left[ \left|V(t,x)-V(s,x)\right|^2 \right] \notag\\
    &=\int_{-\infty}^0 \int_{-\infty}^0 |r_1-r_2|^{2H_0-2}\ud r_1\ud r_2 \int_{\R^d} \cF\left[ p\left(t-r_1,x-\cdot\right)-p\left(s-r_1,x-\cdot\right)\right](\xi) \notag\\
    &\quad \times \overline{ \cF\left[  p\left(t-r_2,x-\cdot\right)- p\left(s-r_2,x-\cdot\right)\right](\xi) } \mu(\ud\xi). \notag\\
    &\le C_{H_0} \bigg( \int_{-\infty}^{0}\bigg(\int_{\R^d} \bigg|\cF\left[ p\left(t-r,x-\cdot\right)- p\left(s-r,x-\cdot\right)\right](\xi)\bigg|^2 \mu(\ud\xi)\bigg)^\frac{1}{2H_0}\ud r\bigg)^{2H_0},
  \end{align}
   where
  \begin{equation}\label{e:fourier-r-v}
   \cF \left[p\left(t-r,x-\cdot\right)\right](\xi)=e^{-i\langle x,\xi\rangle}(t-r)^{\beta+\gamma-1}E_{\beta,\beta+\gamma}\big( -(t-r)^\beta|\xi|^\alpha \big).
  \end{equation}

  For any $x\in\R^d$ and $0< s< t<\infty$, by \eqref{eq Fourier p}, \eqref{e:differential-E} and \eqref{e:bound-MLF}, we have
  \begin{align}\label{e:cFptxpsx}
    &\bigg|\cF\left[  p\left(t-r,x-\cdot\right)-p\left(s-r,x-\cdot\right)\right](\xi)\bigg| \notag\\
    &= \bigg|(t-r)^{\beta+\gamma-1}E_{\beta,\beta+\gamma}\big( -(t-r)^\beta|\xi|^\alpha \big)- (s-r)^{\beta+\gamma-1}E_{\beta,\beta+\gamma}\big(-(s-r)^\beta|\xi|^\alpha \big)\bigg| \notag\\
    &\le \sup_{\delta\in[s,t]} \left|\frac{\partial }{\partial \delta}\left[(\delta-r)^{\beta+\gamma-1}E_{\beta,\beta+\gamma}\big( -(\delta-r)^\beta|\xi|^\alpha \big) \right]\right| \times|t-s| \notag\\
    &= \sup_{\delta\in[s,t]}\left|(\delta-r)^{\beta+\gamma-2}E_{\beta,\beta+\gamma-1}\big( -(\delta-r)^\beta|\xi|^\alpha \big)\right| \times |t-s| \notag\\
    &\le [1\vee(t/s)^{\beta+\gamma-2}]\hat{c} \left| \frac{(s-r)^{\beta+\gamma-2}}{1+(s-r)^\beta|\xi|^\alpha} \right|  \times |t-s|,
  \end{align}
  where $\hat{c}$ is given in \eqref{e:bound-MLF}. Plugging \eqref{e:cFptxpsx} into \eqref{e:vtx-vsy0}, we can see that
  \begin{align*}
     & \E\left[ \left|V(t,x)-V(s,x)\right|^2 \right] \\
     &\le [1\vee(t/s)^{2\beta+2\gamma-4}]  \\
     &\quad \times C_{H_0}\hat{c}^2\bigg( \int_{-\infty}^{0}\bigg(\int_{\R^d} \left| \frac{(s-r)^{\beta+\gamma-2}}{1+(s-r)^\beta|\xi|^\alpha} \right| ^2  \prod_{j=1}^{d}c_{H_j}|\xi_j|^{1-2H_j} \ud\xi\bigg)^\frac{1}{2H_0}\ud r\bigg)^{2H_0}\times |t-s|^2\\
     &\le [1\vee(t/s)^{2\beta+2\gamma-4}]\\
     &\quad\times C_{H_0}\hat{c}^2\bigg(\int_{-\infty}^{0}(s-r)^{\frac{1}{2H_0}\left[2\beta+2\gamma-4-\frac{\beta(2d-2H)}{\alpha}\right]}\ud r\bigg)^{2H_0} \times \int_{\R^d}  \frac{\prod_{j=1}^{d}c_{H_j}|\xi_j|^{1-2H_j}}{1+|\xi|^{2\alpha}} \ud\xi \times|t-s|^2,
  \end{align*}
  where in the second inequality we applied change of variables. Noting that
  $2H_0+2\beta+2\gamma-4-\frac{\beta(2d-2H)}{\alpha}<0$ due to the first condition in \eqref{e:U-exist}, we have
  \begin{equation*}
    \\ \bigg(\int_{-\infty}^{0}(s-r)^{\frac{1}{2H_0}\left[2\beta+2\gamma-4-\frac{\beta(2d-2H)}{\alpha}\right]}\ud r\bigg)^{2H_0}= C s^{2\beta+2\gamma+2H_0-4-\frac{\beta(2d-2H)}{\alpha}}.
  \end{equation*}
  By Lemma \ref{le:space-integral} and \eqref{e:DL}, we have
  \begin{equation}\label{e:de-c1}
    c_1:=\int_{\R^d} \frac{\prod_{j=1}^{d}c_{H_j} |\xi_j|^{1-2H_j}}{1+|\xi|^{2\alpha}} \ud\xi<\infty.
  \end{equation}
  This proves \eqref{e:d1-pre} and then \eqref{e:upb-vtx-vsx} follows.

(2). Next, we prove the regularity of $V(t,x)$ in space. For any $t\in[a,b]$, we have
  \begin{align}\label{e:vtx-vsy-space 1}
    &\E\left[ \left|V(t,x)-V(t,y)\right|^2 \right] \notag\\
    &=\int_{-\infty}^0 \int_{-\infty}^0 |r_1-r_2|^{2H_0-2}\ud r_1\ud r_2 \int_{\R^d} \cF\left[  p\left(t-r_1,x-\cdot\right)- p\left(t-r_1,y-\cdot\right)\right](\xi) \notag\\
    &\quad \times \overline{ \cF\left[  p\left(t-r_2,x-\cdot\right)- p\left(t-r_2,y-\cdot\right)\right](\xi) } \mu(\ud\xi). \notag\\
    &\le C_{H_0} \bigg( \int_{-\infty}^{0}\bigg(\int_{\R^d} \bigg|\cF\left[ p\left(t-r,x-\cdot\right)-  p\left(t-r,y-\cdot\right)\right](\xi)\bigg|^2 \mu(\ud\xi)\bigg)^\frac{1}{2H_0}\ud r\bigg)^{2H_0}.
  \end{align}
 By \eqref{e:fourier-r-v}, we have
  \begin{align*}
     & \left|\cF\left[ p\left(t-r,x-\cdot\right)-  p\left(t-r,y-\cdot\right)\right](\xi) \right|^2 \notag \\
     &= (t-r)^{2\beta+2\gamma-2}E_{\beta,\beta+\gamma}^2 \big( -(t-r)^\beta|\xi|^\alpha \big) \times \left|e^{-i\langle x,\xi\rangle}-e^{-i\langle y,\xi\rangle}\right|^2.
  \end{align*}

  Note that \eqref{e:DL} implies that $\alpha-d+H>0$ which is split into three cases:
  \begin{itemize}
  \item[(i)] $0<\alpha-d+H<1$;
  \item[(ii)] $\alpha-d+H=1$;
  \item[(iii)] $\alpha-d+H>1$.
\end{itemize}

  (i) $0<\alpha-d+H<1$: Since 
  \begin{align}\label{eq Cos}
  \left|e^{-i\langle x,\xi\rangle}-e^{-i\langle y,\xi\rangle}\right|^2=2[1-\cos(\langle x-y,\xi\rangle)],
  \end{align} 
  by using \eqref{e:bound-MLF} and the  change  of variables $\xi_j=rw_j$ in \eqref{e:changevariable}, $w_{x-y}=\frac{x-y}{|x-y|}$, we can bound \eqref{e:vtx-vsy-space 1} by
  \begin{align*}
     &\E\left[ \left|V(t,x)-V(t,y)\right|^2 \right]\\
     &\le 2C_{H_0} \bigg(\int_{-\infty}^{0}\bigg(\int_{\R^d} [1-\cos(\langle x-y,\xi\rangle)] (t-r)^{2\beta+2\gamma-2}E_{\beta,\beta+\gamma}^2 \big( -(t-r)^\beta|\xi|^\alpha \big) \mu(\ud\xi)\bigg)^\frac{1}{2H_0} \ud r\bigg)^{2H_0}\\
     &\le 2C_{H_0}\hat{c}^2 \bigg(\int_{-\infty}^{0}\bigg((t-r)^{2\gamma-2}\int_{\R^d}  \frac{[1-\cos(\langle x-y,\xi\rangle)]}{|\xi|^{2\alpha}} \mu(\ud\xi)\bigg)^\frac{1}{2H_0} \ud r\bigg)^{2H_0}\\
     &\le 2C_{H_0}\hat{c}^2\bigg(\int_{-\infty}^{0}(t-r)^{\frac{2\gamma-2}{2H_0}} \ud r\bigg)^{2H_0}\times \int_{\R^d}  \frac{1-\cos(\langle x-y,\xi\rangle)}{|\xi|^{2\alpha}} \mu(\ud\xi)\\
     &=2C_{H_0}\hat{c}^2\times t^{2H_0+2\gamma-2}\times \int_{0}^{\infty}\frac{1-\cos(r|x-y|\langle w_{x-y},w\rangle)}{r^{2\alpha-2d+2H+1}} \ud r \times \int_{\mathbb{S}^{d-1}}\prod_{j=1}^{d}c_{H_j}w_j^{1-2H_j} \sigma(\ud w)\\
     &=2C_{H_0}\hat{c}^2\times t^{2H_0+2\gamma-2}\times \int_{0}^{\infty}\frac{1-\cos(r\langle w_{x-y},w\rangle)}{r^{2\alpha-2d+2H+1}} \ud r \times \int_{\mathbb{S}^{d-1}}\prod_{j=1}^{d}c_{H_j}w_j^{1-2H_j} \sigma(\ud w)\\
     &\quad \times |x-y|^{2\alpha-2d+2H}\\
     & \le C\:t^{2H_0+2\gamma-2}\times|x-y|^{2\alpha-2d+2H},
  \end{align*}
   where  we have used the fact $2H_0+2\gamma-2<0$ due to \eqref{e:U-exist} and the facts  that $C$ is a finite constant independent of  $x-y$ and  
  \begin{align*}
     & \int_{0}^{\infty}\frac{1-\cos(r\langle w_{x-y},w\rangle)}{r^{2\alpha-2d+2H+1}} \ud r \le \int_{0}^{1/d}\frac{1-\cos(rd)}{r^{2\alpha-2d+2H+1}} \ud r+ \int_{1/d}^{\infty}\frac{1}{r^{2\alpha-2d+2H+1}} \ud r<\infty,
  \end{align*}
  because of    $0<\alpha-d+H<1$.
  
  (ii) $\alpha-d+H=1$: By \eqref{eq Cos} and  \eqref{e:bound-MLF}, we can bound \eqref{e:vtx-vsy-space 1} by
  \begin{align*}
     &\E\left[ \left|V(t,x)-V(t,y)\right|^2 \right]\\
     &\le 2C_{H_0} \bigg(\int_{-\infty}^{0}\bigg(\int_{\R^d} [1-\cos(\langle x-y,\xi\rangle)] (t-r)^{2\beta+2\gamma-2}E_{\beta,\beta+\gamma}^2 \big( -(t-r)^\beta|\xi|^\alpha \big) \mu(\ud\xi)\bigg)^\frac{1}{2H_0} \ud r\bigg)^{2H_0}\\
     &\le 2C_{H_0}\hat{c}^2 \bigg(\int_{-\infty}^{0}\bigg((t-r)^{2\gamma-2}\int_{\R^d}  \frac{[1-\cos(\langle x-y,\xi\rangle)]}{1+|\xi|^{2\alpha}} \mu(\ud\xi)\bigg)^\frac{1}{2H_0} \ud r\bigg)^{2H_0}\\
     &\le 2C_{H_0}\hat{c}^2\bigg(\int_{-\infty}^{0}(t-r)^{\frac{2\gamma-2}{2H_0}} \ud r\bigg)^{2H_0}\times \int_{\R^d}  \frac{1-\cos(\langle x-y,\xi\rangle)}{1+|\xi|^{2\alpha}} \mu(\ud\xi)\\
     &=2C_{H_0}\hat{c}^2\times t^{2H_0+2\gamma-2}\times \int_{0}^{\infty}\frac{1-\cos(r|x-y|\langle w_{x-y},w\rangle)}{1+r^{2\alpha-2d+2H+1}} \ud r \times \int_{\mathbb{S}^{d-1}}\prod_{j=1}^{d}c_{H_j}w_j^{1-2H_j} \sigma(\ud w),
  \end{align*}
  where
  \begin{align*}
     & \int_{0}^{\infty}\frac{1-\cos(r|x-y|\langle w_{x-y},w\rangle)}{1+r^{2\alpha-2d+2H+1}} \ud r=\int_{0}^{\infty}\frac{1-\cos(r|x-y|\langle w_{x-y},w\rangle)}{1+r^{3}} \ud r\\
     &\le \int_{0}^{1}1-\cos(r|x-y|\langle w_{x-y},w\rangle)\ud r+\int_{1}^{\infty}\frac{1-\cos(r|x-y|\langle w_{x-y},w\rangle)}{r^{3}} \ud r\\
     &\le \frac{d^2}{2}\int_{0}^{1}r^2|x-y|^2 \ud r+|x-y|^2\left(\int_{|x-y|}^{1}\frac{1-\cos(r\langle w_{x-y},w\rangle)}{r^3} \ud r +\int_{1}^{\infty}\frac{1}{r^3}\ud r\right)\\
     &\le\frac{d^2}{6}|x-y|^2+|x-y|^2\left(\frac{d^2}{2}\int_{|x-y|}^{1}r^{-1} \ud r +\int_{1}^{\infty}\frac{1}{r^3}\ud r\right)\\
     &=\frac{d^2+3}{6}|x-y|^2+\frac{d^2}{2}|x-y|^2\log(|x-y|^{-1}).
  \end{align*}

  (iii) $\alpha-d+H>1$: since $\left|e^{-i\langle x,\xi\rangle}-e^{-i\langle y,\xi\rangle}\right|\le |\xi||x-y|$, we can bound \eqref{e:vtx-vsy-space 1} by
  \begin{align*}
     &\E\left[ \left|V(t,x)-V(t,y)\right|^2 \right]\\
     &\le C_{H_0}|x-y|^2 \bigg(\int_{-\infty}^{0}\bigg(\int_{\R^d} (t-r)^{2\beta+2\gamma-2}E_{\beta,\beta+\gamma}^2 \big( -(t-r)^\beta|\xi|^\alpha \big) |\xi|^2 \mu(\ud\xi)\bigg)^\frac{1}{2H_0} \ud r\bigg)^{2H_0}\\
     &= C_{H_0} |x-y|^2 \bigg(\int_{-\infty}^{0}(t-r)^{\frac{1}{2H_0}\left[2\beta+2\gamma-2-\frac{\beta(2d-2H+2)}{\alpha}\right]} \ud r\bigg)^{2H_0} \times \int_{\R^d}  E_{\beta,\beta+\gamma}^2 \big( - |\xi|^\alpha \big)|\xi|^2 \mu(\ud\xi).
  \end{align*}
  Using \eqref{e:bound-MLF}, we can bound the above term by
   \begin{align*}
   &\E\left[ \left|V(t,x)-V(t,y)\right|^2 \right]\\
     &\le C_{H_0}\hat{c}^2 |x-y|^2 \bigg(\int_{-\infty}^{0}(t-r)^{\frac{1}{2H_0}\left[2\beta+2\gamma-2-\frac{\beta(2d-2H+2)}{\alpha}\right]} \ud r\bigg)^{2H_0} \times \int_{\R^d} \frac{ \prod_{j=1}^{d}c_{H_j}|\xi_j|^{1-2H_j}}{1+|\xi|^{2\alpha}}|\xi|^2 \ud \xi \\
     &=C_{H_0}\hat{c}^2\: c_2\times t^{2\beta+2\gamma+2H_0-2-\frac{\beta(2d-2H+2)}{\alpha}}\times |x-y|^2,
  \end{align*}
  where the last step is from $\frac{1}{2H_0}\left[2\beta+2\gamma-2-\frac{\beta(2d-2H+2)}{\alpha}\right]<-1$ due to the condition $\theta_2<1$ in \eqref{e:U-exist} and
  \begin{equation*}
    c_2:=\int_{\R^d} \frac{ \prod_{j=1}^{d}c_{H_j}|\xi_j|^{1-2H_j}}{1+|\xi|^{2\alpha}}|\xi|^2 \ud \xi<\infty,
  \end{equation*}
  due to $\alpha-d+H>1$ by Lemma \ref{le:space-integral}.   This completes the proof of \eqref{e:upb-vtxvty}.
  \end{proof}

With an extra condition \eqref{e:con-space-add}, the regularity of $V(t,x)$ can be improved.
\begin{proposition}\label{prop:reg-V-add}
 Denote $I=[a,b]\times[-M,M]^d$ with $0<a<b<\infty$ and $0<M<\infty$. Under conditions of Lemma \ref{le:sLND-U} and
  \begin{equation}\label{e:con-space-add}
    \alpha-d+H>1,
  \end{equation}
  there is a  modification of $V(t,x)$ such that its sample function is almost surely continuously (partially) differentiable on $(t,x)\in I$.  In particular, the  function $(t,x)\mapsto V(t,x)$ is smoother than $(t,x)\mapsto U(t,x)$ on $I$.
\end{proposition}

\begin{proof}
 {\em Step 1:} We first prove the mean square partial derivative of $V(t,x)$ in time has  a continuous modification. Using the dominated convergence theorem, we obtain that the mean square derivative in time is
\begin{equation*}
  \frac{\partial}{\partial t} V(t,x)=\int_{-\infty}^{0}\int_{\R^d}\frac{\partial}{\partial t} p(t-r,x-z) W(\ud r,\ud z).
\end{equation*}
Notice that $\E\left|\frac{\partial}{\partial t}V(t,x)\right|^2<\infty$ under \eqref{e:DL} and \eqref{e:U-exist}. Thus,  $\frac{\partial}{\partial t}V(t,x)$ is well-defined. Using \eqref{e:innerproduct-cH}, \eqref{e:de-V},  Lemma \ref{le:B.2} and  Minkowski's inequality, we see that
    \begin{align}\label{e:vtx-vsy-2}
    &\E\left[ \left|\frac{\partial}{\partial t}V(t,x)-\frac{\partial}{\partial s}V(s,x)\right|^2 \right] \notag\\
    &= \int_{-\infty}^0 \int_{-\infty}^0 |r_1-r_2|^{2H_0-2}\ud r_1\ud r_2 \int_{\R^d} \cF\left[  \frac{\partial}{\partial t}p\left(t-r_1,x-\cdot\right)-\frac{\partial}{\partial s}p\left(s-r_1,x-\cdot\right)\right](\xi) \notag\\
    &\quad \times \overline{ \cF\left[ \frac{\partial}{\partial t} p\left(t-r_2,x-\cdot\right)- \frac{\partial}{\partial s}p\left(s-r_2,x-\cdot\right)\right](\xi) } \mu(\ud\xi). \notag\\
    &\le C_{H_0} \bigg( \int_{-\infty}^{0}\bigg(\int_{\R^d} \bigg|\cF\left[ \frac{\partial}{\partial t}p\left(t-r,x-\cdot\right)- \frac{\partial}{\partial s}p\left(s-r,x-\cdot\right)\right](\xi)\bigg|^2 \mu(\ud\xi)\bigg)^\frac{1}{2H_0}\ud r\bigg)^{2H_0}.
  \end{align}
   Using Lemma \ref{le:partial-fourier}, we have
  \begin{align*}
    \cF \left[\frac{\partial}{\partial t}p\left(t-r,x-\cdot\right)\right](\xi)=&\, \frac{\partial}{\partial t} \cF p(t-r,x-\cdot)(\xi)\\
   &= e^{-i\langle x,\xi\rangle}\frac{\partial}{\partial t} \left[(t-r)^{\beta+\gamma-1}E_{\beta,\beta+\gamma}\big( - (t-r)^\beta|\xi|^\alpha \big)\right]\\
   &=e^{-i\langle x,\xi\rangle}(t-r)^{\beta+\gamma-2}E_{\beta,\beta+\gamma-1}\big( - (t-r)^\beta|\xi|^\alpha \big),
  \end{align*}
  where the last step is due to \eqref{e:differential-E}.
By using \eqref{e:differential-E} again,   we have
  \begin{align}\label{e:phitx-sx}
    &\bigg|\cF\left[ \frac{\partial}{\partial t} p\left(t-r,x-\cdot\right)-\frac{\partial}{\partial s}p\left(s-r,x-\cdot\right)\right](\xi)\bigg| \notag\\
    &= \bigg|(t-r)^{\beta+\gamma-2}E_{\beta,\beta+\gamma-1}\big( -(t-r)^\beta|\xi|^\alpha \big)- (s-r)^{\beta+\gamma-2}E_{\beta,\beta+\gamma-1}\big( -(s-r)^\beta|\xi|^\alpha \big)\bigg| \notag\\
    &\le \sup_{\delta\in[a,b]} \left| \frac{\partial }{\partial \delta}\left[ (\delta-r)^{\beta+\gamma-2}E_{\beta,\beta+\gamma-1}\big( -(\delta-r)^\beta|\xi|^\alpha \big) \right]\right| \times|t-s| \notag\\
    &= \sup_{\delta\in[a,b]}\left|(\delta-r)^{\beta+\gamma-3}E_{\beta,\beta+\gamma-2}\big( -(\delta-r)^\beta|\xi|^\alpha \big)\right| \times |t-s| \notag\\
    &\le \left| \frac{C(a-r)^{\beta+\gamma-3}}{1+(a-r)^\beta|\xi|^\alpha} \right|  \times |t-s|.
  \end{align}
  Recall that $c_1$ is given in \eqref{e:de-c1}. Combining \eqref{e:vtx-vsy-2} with \eqref{e:phitx-sx}, we can get
  \begin{align*}
     & \E\left[ \left|\frac{\partial}{\partial t}V(t,x)-\frac{\partial}{\partial s}V(s,x)\right|^2 \right] \\
     &\le C_{H_0}\hat{c}^2 \bigg( \int_{-\infty}^{0}\bigg(\int_{\R^d} \left| \frac{(a-r)^{\beta+\gamma-3}}{1+(a-r)^\beta|\xi|^\alpha} \right| ^2  \prod_{j=1}^{d}c_{H_j}|\xi_j|^{1-2H_j} \ud\xi\bigg)^\frac{1}{2H_0}\ud r\bigg)^{2H_0} \times |t-s|^2\\
     &\le C_{H_0}\hat{c}^2 \bigg(\int_{-\infty}^{0}(a-r)^{\frac{1}{2H_0}\left[2\beta+2\gamma-6-\frac{\beta(2d-2H)}{\alpha}\right]}\ud r\bigg)^{2H_0}\times \int_{\R^d}  \frac{\prod_{j=1}^{d}c_{H_j}|\xi_j|^{1-2H_j}}{1+|\xi|^{2\alpha}} \ud\xi \times|t-s|^2\\
     &=C_{H_0}\:\hat{c}^2\times a^{2\beta+2\gamma+2H_0-6-\frac{\beta(2d-2H)}{\alpha}}\times c_1 |t-s|^2,
  \end{align*}
 where  $2\beta+2\gamma+2H_0-6-\frac{\beta(2d-2H)}{\alpha}<0$ due to \eqref{e:U-exist}, and $c_1<\infty$ due to Lemma \ref{le:space-integral} and \eqref{e:DL}.
  Therefore, for all $s,t\in[a,b]$ and $x\in\R^d$,
  \begin{align}\label{e:up-b-partial-t-v-t}
    &\E\left[ \left|\frac{\partial}{\partial t}V(t,x)-\frac{\partial}{\partial s}V(s,x)\right|^2 \right] \le C |t-s|^2.
  \end{align}
  Similarly, we can verify that for  $t\in[a,b]$ and $x,y\in\R^d$,
  \begin{align*}
    &\E\left[ \left|\frac{\partial}{\partial t}V(t,x)-\frac{\partial}{\partial t}V(t,y)\right|^2 \right]\\
    &\le C_{H_0} \bigg( \int_{-\infty}^{0}\bigg(\int_{\R^d} \bigg|\cF\left[ \frac{\partial}{\partial t}p\left(t-r,x-\cdot\right)- \frac{\partial}{\partial t}p\left(t-r,y-\cdot\right)\right](\xi)\bigg|^2 \mu(\ud\xi)\bigg)^\frac{1}{2H_0}\ud r\bigg)^{2H_0},
  \end{align*}
  where
  \begin{align*}
     &\bigg|\cF\left[ \frac{\partial}{\partial t}p\left(t-r,x-\cdot\right)- \frac{\partial}{\partial t}p\left(t-r,y-\cdot\right)\right](\xi)\bigg|^2\\
     &=\bigg|(t-r)^{\beta+\gamma-2}E_{\beta,\beta+\gamma-1}\big( -(t-r)^\beta|\xi|^\alpha \big)\bigg|^2 \times \left|e^{-i\langle x,\xi\rangle}-e^{-i\langle y,\xi\rangle} \right|^2\\
     &\le \bigg|(t-r)^{\beta+\gamma-2}E_{\beta,\beta+\gamma-1}\big( -(t-r)^\beta|\xi|^\alpha \big)\bigg|^2 \times |\xi|^2 |x-y|^2\\
     &\le \bigg|\frac{(t-r)^{\beta+\gamma-2}}{1+(t-r)^\beta|\xi|^\alpha} \bigg|^2 \times |\xi|^2 |x-y|^2.
  \end{align*}
 Then, we have
 \begin{align}\label{e:up-b-partial-t-v-x}
    &\E\left[ \left|\frac{\partial}{\partial t}V(t,x)-\frac{\partial}{\partial t}V(t,y)\right|^2 \right]\notag\\
    &\le  C_{H_0} \bigg( \int_{-\infty}^{0}\bigg(\int_{\R^d} \bigg|\frac{(t-r)^{\beta+\gamma-2}}{1+(t-r)^\beta|\xi|^\alpha} \bigg|^2 \times |\xi|^2 \mu(\ud\xi)\bigg)^\frac{1}{2H_0}\ud r\bigg)^{2H_0} \times|x-y|^2 \notag\\
    &\le C |x-y|^2,
 \end{align}
 where $C<\infty$ due to conditions \eqref{e:U-exist} and \eqref{e:con-space-add}.

 Therefore, combining \eqref{e:up-b-partial-t-v-t} with \eqref{e:up-b-partial-t-v-x}, we see that for any $(t,x),(s,y)\in[a,b]\times\R^d$,
 \begin{equation*}
   \E\left[ \left|\frac{\partial}{\partial t}V(t,x)-\frac{\partial}{\partial s}V(s,y)\right|^2 \right]\le C \left(|t-s|^2+ |x-y|^2\right).
 \end{equation*}
 Then by Kolmogorov’s continuity theorem, we can find a modification of $\frac{\partial}{\partial t}V(t,x)$ such that its sample function is continuous in $I$.

  {\em Step 2:} Next, we prove that the mean square partial derivatives of $V(t,x)$ in $x_i, i=1, \dots, d$,  have continuous modifications. For $i=1,\dots,d$,
  \begin{equation*}
    \frac{\partial}{\partial x_i} V(t,x)=\int_{-\infty}^{0}\int_{\R^d}\frac{\partial}{\partial x_i} p(t-r,x-z) W(\ud r,\ud z),
  \end{equation*}
  is   well-defined under \eqref{e:DL} and  \eqref{e:U-exist}. Then,
  \begin{align*}
    &\E\left[ \left|\frac{\partial}{\partial x_i}V(t,x)-\frac{\partial}{\partial y_i}V(t,y)\right|^2 \right] \notag\\
    &=  \int_{-\infty}^0 \int_{-\infty}^0 |r_1-r_2|^{2H_0-2}\ud r_1\ud r_2 \int_{\R^d} \cF\left[  \frac{\partial}{\partial x_i}p\left(t-r_1,x-\cdot\right)-\frac{\partial}{\partial y_i} p\left(t-r_1,y-\cdot\right)\right](\xi) \notag\\
    &\quad \times \overline{ \cF\left[ \frac{\partial}{\partial x_i} p\left(t-r_2,x-\cdot\right)- \frac{\partial}{\partial y_i}p\left(t-r_2,y-\cdot\right)\right](\xi) } \mu(\ud\xi) \notag\\
    &\le C_{H_0} \bigg( \int_{-\infty}^{0}\bigg(\int_{\R^d} \bigg|\cF\left[ \frac{\partial}{\partial x_i}p\left(t-r,x-\cdot\right)- \frac{\partial}{\partial y_i} p\left(t-r,y-\cdot\right)\right](\xi)\bigg|^2 \mu(\ud\xi)\bigg)^\frac{1}{2H_0}\ud r\bigg)^{2H_0}.
  \end{align*}
  Note that \eqref{e:con-space-add} implies $\alpha>1$. According to Lemma \ref{le:partial-fourier}, the Fourier transform is given by
  \begin{align*}
    \cF \left[\frac{\partial}{\partial x_i}p\left(t-r,x-\cdot\right)\right](\xi) &=\frac{\partial}{\partial x_i} \cF p\left(t-r,x-\cdot\right)(\xi)\\
   &=-i\xi_i e^{-i\langle x,\xi\rangle}(t-r)^{\beta+\gamma-1}E_{\beta,\beta+\gamma}\big( -(t-r)^\beta|\xi|^\alpha \big),
  \end{align*}
  and
  \begin{align*}
     & \left|\cF\left[ \frac{\partial}{\partial x_i}p\left(t-r,x-\cdot\right)- \frac{\partial}{\partial y_i} p\left(t-r,y-\cdot\right)\right](\xi) \right|^2 \notag \\
     &= \left|\xi_i(t-r)^{\beta+\gamma-1}E_{\beta,\beta+\gamma}\big( -(t-r)^\beta|\xi|^\alpha \big) \right|^2 \times \left|e^{-i\langle x,\xi\rangle}-e^{-i\langle y,\xi\rangle}\right|^2.
  \end{align*}
  Now, we separate the condition \eqref{e:con-space-add}  into three cases:
  \begin{itemize}
  \item[(i)] $1<\alpha-d+H<2$;
  \item[(ii)] $\alpha-d+H=2$;
  \item[(iii)] $\alpha-d+H>2$.
  \end{itemize}
   Let
  \begin{equation*}
    \sigma(r)=
    \begin{cases}
      r^{2\alpha-2d+2H-2}, & \mbox{if } 1<\alpha-d+H<2; \\
      r^2|\log r|, & \mbox{if } \alpha-d+H=2; \\
      r^2, & \mbox{if }\alpha-d+H>2.
    \end{cases}
  \end{equation*}
  Using a similar argument in the proof of Proposition \ref{prop:reg-V}, we can obtain that
  \begin{equation*}
    \E\left[ \left|\frac{\partial}{\partial x_i}V(t,x)-\frac{\partial}{\partial y_i}V(t,y)\right|^2 \right]\le \sigma(|x-y|),
  \end{equation*}
  and
  \begin{equation*}
    \E\left[ \left|\frac{\partial}{\partial x_i}V(t,x)-\frac{\partial}{\partial x_i}V(s,x)\right|^2 \right]\le C|t-s|^2.
  \end{equation*}
Consequently,
  \begin{equation*}
    \E\left[ \left|\frac{\partial}{\partial x_i}V(t,x)-\frac{\partial}{\partial y_i}V(s,y)\right|^2 \right]\le C \left(|t-s|^2+ \sigma|x-y|\right).
  \end{equation*}
  By the Gaussian property of $\frac{\partial}{\partial x_i}V(t,x)$ and  Kolmogorov’s continuity theorem, we can find a modification of $\frac{\partial}{\partial x_i}V(t,x)$ such that its sample function is continuous on $I$.

 {\em Step 3:} We will construct a modification of $V(t,x)$ such that its sample function has continuous partial derivatives almost surely. In Step 1, we derive that there is a continuous version $\tilde{V}$ of $V$ such that $\frac{\partial \tilde{V}}{\partial t}(t)$ is continuous. Then, we apply Step 2 to $x_1$ and obtain a version $V^{(1)}$ of $\tilde{V}$ defined by
 \begin{equation*}
   V^{(1)}(t,x)=  \tilde{V}(t,-M,x_2,\dots,x_d)+\int_{-M}^{x} \frac{\partial}{\partial y}  \tilde{V}(t,y,x_2,\dots,x_d)\ud y.
 \end{equation*}
  Then, $\frac{\partial }{\partial t}V^{(1)}(t,x)$ and $\frac{\partial }{\partial x_1}V^{(1)}(t,x)$  are almost surely continuous. Repeating this procedure for $x_i,i=2,\dots,d$, we get a continuous version $V^{(d)}(t,x)$ such that all first-order partial derivatives of $V^{(d)}(t,x)$ are continuous almost surely. Hence, the sample function of $V^{(d)}(t,x)$ is almost surely differentiable. The proof is complete.
\end{proof}

\begin{remark}
Indeed, the  condition \eqref{e:con-space-add} is not necessary for Proposition \ref{prop:reg-V-add} in case of the stochastic heat equation (see also \cite[Theorem 3.6]{Herrell2020sharp}), since the Fourier transform of the fundamental solution has exponential decay in this case (i.e., $\beta=1$ and $\gamma=0$):
  \begin{equation*}
    \cF p(t,\cdot)(\xi)=E_{1,1}\big(- t|\xi|^\alpha\big)=e^{- t|\xi|^\alpha}.
  \end{equation*}
 While for the other cases (i.e., $\beta\neq1$ or $\gamma\neq0$), the Fourier transform of the fundamental solution only has   polynomial decay by \eqref{e:asym-MLF}:
  \begin{equation*}
    E_{\beta,\beta+\gamma}(- t|\xi|^\alpha) = \dfrac{1}{\Gamma(\gamma)t|\xi|^\alpha}-\dfrac{1}{\Gamma(\gamma-\beta)t^2|\xi|^{2\alpha}}+o\left(\frac{1}{|\xi|^{2\alpha}}\right),\quad \text{as } \xi\to\infty.
  \end{equation*}
Thus, when $\beta\neq1$ or $\gamma\neq0$, the condition \eqref{e:con-space-add} is needed for Proposition \ref{prop:reg-V-add}.
\end{remark}

Now, we state the main result in this subsection.
\begin{theorem}[Space-time joint regularity]\label{th:modulus-both}
   Assume the conditions of Proposition \ref{prop:reg-V-add} hold. Denote $I=[a,b]\times[-M,M]^d$ with $0<a<b<\infty$ and $0<M<\infty$.  Denote $\rho(t,x)=|t|^{\theta_1}+|x|^{\theta_2}$.
  \begin{enumerate}
    \item (Uniform modulus of continuity). There exists a constant $k_1\in(0,\infty)$ such that
    \begin{equation*}
      \lim_{\e\to0^+}\sup_{\substack{(t,x),(s,y)\in I;\\ \rho(t-s,x-y)\le\e}}\dfrac{|u(t,x)-u(s,y)|}{\rho(t-s,x-y)\sqrt{\log\left(1+\rho(t-s,x-y)^{-1}\right)}}=k_1,
      \text{ a.s. }
    \end{equation*}
    \item (Local modulus of continuity). There exists a constant $k_2\in(0,\infty)$ such that for any $(t,x)\in I$,
    \begin{equation*}
      \lim_{\e\to0^+}\sup_{\rho(s,y)\le \e} \dfrac{|u(t+s,x+y)-u(t,x)|}{\rho(s,y)\sqrt{\log\log\left(1+\rho(s,y)^{-1}\right)}}=k_2,
      \text{ a.s. }
    \end{equation*}
   \item (Chung's LIL). There exists a constant $k_3\in(0,\infty)$ such that for all $(t,x)\in I$,
  \begin{equation*}
    \liminf_{\e\to0^+}\sup_{\rho(s,y)\le \e} \dfrac{|u(t+s,x+y)-u(t,x)|}{\e\left(\log\log 1/\e\right)^{-1/Q}}=k_3, \text{ a.s. }
  \end{equation*}
  where $Q=\frac{1}{\theta_1}+\frac{d}{\theta_2}$.
  \end{enumerate}
\end{theorem}

\begin{proof}
  Since   $V(t,x)$  is smoother than $U(t,x)$ on $I=[a,b]\times[-M,M]^d$   by  Propositions \ref{prop:reg-V} and \ref{prop:reg-V-add},  the regularity properties of $u(t,x)=U(t,x)-V(t,x)$ are determined  by the increment-stationary field $U(t,x)$ (see Proposition  \ref{prop:modulus-both-U}), and thus we have the above regularity results for $u(t,x)$.
\end{proof}

\begin{remark}[A comparison of Theorem \ref{th:modulus-both} with known results]\label{re:comp-both}\hfill
\begin{enumerate}
  \item For the fractional heat equation (i.e., $\beta=1$ and $\gamma=0$), the joint uniform and local moduli of continuity are consistent with
   \cite[Proposition 3.7]{Herrell2020sharp}, and Chung's LIL coincides with \cite[Proposition 3.10]{Herrell2020sharp}.
  \item  To our best knowledge, the joint regularity properties are new for the case when $\beta\neq1$.
\end{enumerate}
\end{remark}

\subsection{Temporal regularity}

In this subsection, we obtain the temporal regularity of $u(t,x)$ in Theorem \ref{th:modulus-time}. Recall that $U(t,x)$ and $V(t,x)$ are defined in \eqref{e:de-U} and \eqref{e:de-V}, respectively. For any fixed $x\in\R^d$, define
  \begin{equation}\label{e:de-tilde-U}
    \tilde{U}_x(t):=U(t,x)-U(0,x),
  \end{equation}
  and
  \begin{equation}\label{e:de-tilde-V}
    \tilde{V}_x(t):=V(t,x)-V(0,x).
  \end{equation}
  Note that $U(0,x)=V(0,x)$. By the definition of $\tilde{U}_x(t)$ and $\tilde{V}_x(t)$, we have
     \begin{equation}\label{e:u=tildeUV}
       u(t,x)=\tilde{U}_x(t)-\tilde{V}_x(t).
     \end{equation}
We will show that the $\tilde{V}_x(t)$ has smoother sample paths than $\tilde{U}_x(t)$ in $(0,\infty)$, and
hence the regularity properties of $t \mapsto u(t,x)$ solely depend on those of $\tilde{U}_x(t)$ in $(0,\infty)$.

The following proposition is a key ingredient to study the sample path regularity of $\tilde{U}_x(t)$.
\begin{lemma}\label{le:sLND-tilde-U}
Recall that $\theta_1$ is given in \eqref{e:theta12}. Assume  \eqref{e:DL} holds and
\begin{equation}\label{e:tilde-U-exist}
    \theta_1<1.
\end{equation}
\begin{enumerate}
  \item The Gaussian process $\tilde U_x=\left\{\tilde U_x(t),t\in\R\right\}$ has stationary increments with spectral measure
  \begin{equation*}
    F_{\tilde U}(\ud\tau)=\frac{C\ud \tau}{|\tau|^{2\theta_1+1}}.
  \end{equation*}
  \item The Gaussian process $\tilde U_x$ satisfies the strong local nondeterminism  property: there exists
  a constant $C>0$ such that for any $n\in\N$ and $t,t^1,\dots,t^n\in\R$,
  \begin{equation}\label{e:sLND-tilde-U}
    \Var\big(\tilde U_x(t)|\tilde U_x(t^1),\dots,\tilde U_x(t^n)  \big)\ge C\min_{k=0,\dots,n} |t-t^k|^{2\theta_1} ,
  \end{equation}
  where $t^0=0$.
  \item 
  For all $t,s \ge0$,
  \begin{equation*}
    \E\left[ |\tilde{U}_x(t)-\tilde{U}_x(s)|^2 \right]= C_1|t-s|^{2\theta_1},
  \end{equation*}
  where
  \begin{equation}\label{e:de-C1}
      C_1=2C_{H_0}\int_{\R}\frac{1-\cos\big(\tau \big)}{|\tau|^{2\theta_1+1}}\ud \tau \int_{\R^d}\frac{\prod_{j=1}^{d}c_{H_j}|\xi_j|^{1-2H_j}}
      {1+2|\xi|^\alpha\cos\left(\frac{\pi\beta}{2}\right)+|\xi|^{2\alpha}}\ud\xi.
  \end{equation}
Consequently, $\left\{\frac{\tilde{U}_x(t)}{\sqrt{C_1}},t\ge0\right\}$ is a (standard) fractional Brownian motion with Hurst index $\theta_1\in(0,1)$.
\end{enumerate}
\end{lemma}

\begin{proof}
  By the definition of $\tilde{U}_x(t)$ and \eqref{e:proof-spectral-tilde-U}, we have
  \begin{align*}
    &\E\Big[\big|\tilde{U}_x(t)-\tilde{U}_x(s)\big|^2\Big]=\E\left[ |U(t,x)-U(s,x)|^2 \right]\\
    &=2C_{H_0}\int_{\R^{d+1}} \frac{ 1-\cos\big( (t-s)\tau \big) }{|\tau|^{2\beta+2\gamma+2H_0-1-\frac\beta\alpha(2d-2H)}}
    \frac{\prod_{j=1}^{d}c_{H_j} |\xi_j|^{1-2H_j}}{1+2|\xi|^\alpha\cos\left(\frac{\pi\beta}{2}\right)+|\xi|^{2\alpha}}\ud\xi\ud \tau\\
    &=C\int_{\R}  \frac{ 1-\cos\big( (t-s)\tau \big) }{|\tau|^{2\theta_1+1}}\ud \tau,
  \end{align*}
  where the integral with respect to $\xi$ is finite due to \eqref{e:time-rgular-2-integral} and then the first assertion follows.

  The spectral density of $\big\{\tilde U_x(t),t\in\R\big\}$ is
  \begin{equation*}
    f_{\tilde U}(\tau)=\frac{C }{|\tau|^{2\theta_1+1}}.
  \end{equation*}
  Since for every constant $c>0$, 
  \begin{equation*}
    f_{\tilde U}(c^\frac{1}{\theta_1}\tau)=c^{-(2+\frac{1}{\theta_1})} f_{\tilde U}(\tau), 
  \end{equation*}
 the SLND property follows directly from Theorem 4.1 in \cite[Appendix A]{Herrell2020sharp}.

By \eqref{e:up-bound-t}, we have
\begin{align*}
    &\E\Big[\big|\tilde{U}_x(t)-\tilde{U}_x(s)\big|^2\Big]=\E\left[ |U(t,x)-U(s,x)|^2 \right]=C_1 |t-s|^{2\theta_1},
  \end{align*}
where $C_1$ is the constant in \eqref{e:de-C1}.
Noting $\tilde{U}_x(0)=0$,  we have $\E\big[ |\tilde{U}_x(t)|^2 \big] = C_1 |t|^{2\theta_1}$ and $\tilde{U}_x(t)$ is well defined for all $t\in\R$.
Then,
    \begin{equation*}
      \frac{1}{C_1} \E\left[ \tilde{U}_x(t) \tilde{U}_x(s)\big) \right] =\frac{1}{2}\left( |t|^{2\theta_1}+|s|^{2\theta_1}-|t-s|^{2\theta_1} \right).
    \end{equation*}
Consequently,  $\left\{\frac{\tilde{U}_x(t)}{\sqrt{C_1}},t\ge0\right\}$ is a standard fractional Brownian motion with Hurst index $\theta_1\in(0,1)$.
    This finishes the proof.
\end{proof}

Now we are ready to study the regularity of $\{ \tilde{U}_x(t),\, t\in \R \}$.

\begin{proposition}\label{prop:modulus-time-U}
   Assume the conditions of Lemma \ref{le:sLND-tilde-U} hold. Recall that $\theta_1$ and $C_1$ are given in  \eqref{e:theta12} and
   \eqref{e:de-C1}, respectively. For any  fixed $x\in\R^d$, we have the following temporal regularity results:
  \begin{enumerate}
    \item (Uniform modulus of continuity). For any $0\le a<b<\infty$,
    \begin{equation*}
      \lim_{\e\to0^+}\sup_{\substack{s,t\in[a,b];\\ |t-s|\le\e}}\dfrac{|\tilde U_x(t)-\tilde U_x(s)|}{|t-s|^{\theta_1}\sqrt{\log\left(1+|t-s|^{-1}\right)}}=\sqrt{2C_1},
      \text{ a.s. }
    \end{equation*}
    \item (Local modulus of continuity). For any $t\ge0$,
    \begin{equation*}
      \lim_{\e\to0^+}\sup_{|s|\le \e} \dfrac{|\tilde U_x(t+s)-\tilde U_x(t)|}{|s|^{\theta_1} \sqrt{\log\log\left(1+|s|^{-1}\right)}}=\sqrt{2C_1},
      \text{ a.s. }
    \end{equation*}
    \item (Chung's LIL). For any $t\ge0$,
    \begin{equation*}
      \liminf_{\e\to0^+} \sup_{|s|\le \e} \dfrac{|\tilde U_x(t+s)-\tilde U_x(t)|}{\e^{\theta_1}\left(\log\log \e^{-1}\right)^{-\theta_1}}= \sqrt{C_1}  \kappa^{\theta_1}, \text{ a.s. }
    \end{equation*}
    where $\kappa$ is the small ball constant of a fractional Brownian motion with index $\theta_1$ (e.g. \cite[Theorem 6.9]{Li2001Gaussian}).
   \end{enumerate}
\end{proposition}

\begin{proof}
As shown in Lemma \ref{le:sLND-tilde-U} (3), $\left\{\frac{\tilde U_x(t)}{\sqrt{C_1}},t\ge0\right\}$ is a fractional Brownian motion.
Hence, the uniform and local moduli of continuity of $\{ \tilde{U}_x(t),\, t\in \R \}$ with the
exact constants follow from \cite[(7.5), (7.6)]{Li2001Gaussian}. Finally, Chung's LIL follows from \cite[Theorem 7.1]{Li2001Gaussian}.
\end{proof}

The following proposition shows that the sample function of $\tilde{V}_x(t)$ is more regular than $\tilde{U}_x(t)$ in $(0,\infty)$.

\begin{proposition}\label{prop:reg-tilde-V}
  Assume the conditions of Lemma \ref{le:sLND-tilde-U} hold. Then, for any fixed $x\in\R^d$, $\big\{\tilde{V}_x(t),t\ge0\big\}$
  is well-defined. For  any  $0<a<b<\infty$, there exists a finite constant $C > 0$ such that for all $s,t\in[a,b]$,
  \begin{equation}\label{e:upb-tilde-vtx-vsx}
  \E\left[ \big|\tilde{V}_x(t)-\tilde{V}_x(s)\big|^2 \right] \le C |t-s|^2.
  \end{equation}
   Consequently, the function $t\mapsto \tilde{V}_x(t)$ is smoother than $t\mapsto \tilde{U}_x(t)$ on $[a,b]$.
\end{proposition}

\begin{proof}
  Note that $U(0,x)=V(0,x)$. By \eqref{e:split-u} and definition of $\tilde{U}_x(t)$ \eqref{e:de-tilde-U} and $\tilde{V}_x(t)$ \eqref{e:de-tilde-V},
  we have
 \begin{align*}
   &\E[|\tilde{V}_x(t)|^2]=\E[|U(t,x)-u(t,x)-V(0,x)|^2]\\
   &=\E[|\tilde{U}_x(t)-u(t,x)|^2]\le 2\E[|\tilde{U}_x(t)|^2]+2\E[|u(t,x)|^2]<\infty,
 \end{align*}
 under conditions \eqref{e:DL} and \eqref{e:tilde-U-exist}, due to Theorem \ref{th:ex-solution} and Lemma \ref{le:sLND-tilde-U}. Then,
 $\tilde{V}_x(t)$ is well-defined.

By \eqref{e:de-tilde-V} and \eqref{e:upb-vtx-vsx}, we have
 \begin{equation*}
 \E\left[ \big|\tilde{V}_x(t)-\tilde{V}_x(s)\big|^2 \right]  = \E\left[ \big|V(t,x)-V(s,x)\big|^2 \right]\le C |t-s|^2.
 \end{equation*}
 This completes the proof.
\end{proof}

Now, we state the main result in this subsection.
\begin{theorem}[Temporal regularity]\label{th:modulus-time}
   Assume the conditions of Lemma \ref{le:sLND-tilde-U} hold  and recall that $\theta_1$ and $C_1$ are given in  \eqref{e:theta12} and
   \eqref{e:de-C1}, respectively. For any  fixed $x\in\R^d$, we have the following temporal regularity results:
  \begin{enumerate}
    \item (Uniform modulus of continuity). For any $0<a<b$,
    \begin{equation*}
      \lim_{\e\to0^+}\sup_{\substack{s,t\in[a,b];\\ |t-s|\le\e}}\dfrac{|u(t,x)-u(s,x)|}{|t-s|^{\theta_1}\sqrt{\log\left(1+|t-s|^{-1}\right)}}=\sqrt{2C_1},
      \text{ a.s. }
    \end{equation*}
    \item (Local modulus of continuity). For any $t\in(0,\infty)$,
    \begin{equation*}
      \lim_{\e\to0^+}\sup_{|s|\le \e} \dfrac{|u(t+s,x)-u(t,x)|}{|s|^{\theta_1} \sqrt{\log\log\left(1+|s|^{-1}\right)}}=\sqrt{2C_1},
      \text{ a.s. }
    \end{equation*}
    \item (Chung's LIL). For any $t\in(0,\infty)$,
    \begin{equation*}
      \liminf_{\e\to0^+} \sup_{|s|\le \e} \dfrac{|u(t+s,x)-u(t,x)|}{\e^{\theta_1}\left(\log\log \e^{-1}\right)^{-\theta_1}}=\sqrt{C_1}  \kappa^{\theta_1}, \text{ a.s.},
    \end{equation*}
    where $\kappa$ is the small ball constant of a fractional Brownian motion with index $\theta_1$ (e.g. \cite[Theorem 6.9]{Li2001Gaussian}).
   \end{enumerate}
   \end{theorem}

   \begin{proof}
     Since   $\tilde{V}_x(t)$  is smoother than $\tilde{U}_x(t)$ on $I=[a,b]$ with $0<a<b$  by  Proposition \ref{prop:reg-tilde-V},  the regularity properties
     of $u(t,x)$ are determined by the increment-stationary field $\tilde{U}_x(t)$ (see Proposition  \ref{prop:modulus-time-U}), and thus we have the
     above regularity results for $u(t,x)$.
   \end{proof}

\begin{remark}[A comparison of Theorem \ref{th:modulus-time} with the known results]\label{re:comp-time}\hfill
 \begin{enumerate}
   \item For the fractional heat equation (i.e., $\beta=1$ and $\gamma=0$), the uniform and local moduli of continuity in time are consistent with
    \cite[Corollary 3.8]{Herrell2020sharp}.
   \item If $\alpha=2$, $\gamma=1-\beta$ and the noise $\W$ is a space-time white noise (i.e., $H_0=H_1=\cdots=H_d=\frac{1}{2}$, $d=1,2,3$),
   then the  uniform and local moduli of continuity in time are consistent with
   \cite[Theorem 1.4]{Xiao2017JDE}.
 \end{enumerate}
\end{remark}

\subsection{Spatial regularity}

In this subsection, we obtain the spatial regularity of $u(t,x)$ in Theorem \ref{th:modulus-space}. Recall that $U(t,x)$ and $V(t,x)$
are defined in \eqref{e:de-U} and \eqref{e:de-V}, respectively. For any fixed $t\ge0$, define
  \begin{equation}\label{e:de-hat-U}
    \hat{U}_t(x):=U(t,x)-U(t,0), \ \hbox{ for } x \in \R^d
  \end{equation}
  and
  \begin{equation}\label{e:de-hat-V}
    \hat{V}_t(x):=V(t,x)-V(t,0),\  \hbox{ for } x \in \R^d.
  \end{equation}
Then define
\begin{equation}\label{e:de-hat-u}
    \hat{u}_t(x):=u(t,x)-u(t,0)=\hat{U}_t(x)-\hat{V}_t(x).
  \end{equation}

The following proposition is a key ingredient to study the sample path regularity of $\hat{U}_t(x)$.
\begin{lemma}\label{le:sLND-hat-U}
Recall that $\theta_2$ is given in \eqref{e:theta12}. Assume  \eqref{e:DL} holds and
\begin{equation}\label{e:hat-U-exist}
     \begin{cases}
       \theta_2<1,  \\
       \gamma<1-H_0.
     \end{cases}
  \end{equation}
\begin{enumerate}
  \item The Gaussian random field $\hat U_t =\big\{\tilde U_t(x),x\in\R^d\big\}$ has stationary increments with spectral measure
  \begin{equation*}
    F_{\hat U}(\ud\xi)= \frac{C\prod_{j=1}^{d}c_{H_j}|\xi_j|^{1-2H_j}}{|\xi|^{2\alpha+\frac\alpha\beta(2\gamma+2H_0-2)}} \ud\xi.
  \end{equation*}
  \item The Gaussian random field $\hat U_t $ satisfies the strong local nondeterminism
  property: there exists a constant $C>0$ such that for all $n\in\N$ and $x,x^1,\dots,x^n\in\R^d$,
  \begin{equation}\label{e:sLND-hat-U}
    \Var\big(\hat U_t(x)|\hat U_t(x^1),\dots,\hat U_t(x^n)  \big)\ge C\min_{k=0,\dots,n} \left( \sum_{i=1}^d|x_i-x^{k}_i|^{\theta_2} \right)^2 ,
  \end{equation}
  where $x^0_i=0$, $i=1,\dots,d$.
  \item 
  There exist constants $c,C>0$ such that
        \begin{equation}\label{e:bound-hat-U-space}
          c|x-y|^{2\theta_2} \le \E\left[ |\hat U_t(x)-\hat U_t(y)|^2 \right]\le C|x-y|^{2\theta_2}
        \end{equation}
        for all $x,y\in\R^d$.
\end{enumerate}
\end{lemma}

\begin{proof}
 By the definition of $\hat U_t(x)$ and \eqref{e:proof-spectral-hat-U}, we have
 \begin{align}\label{e:indepence-t}
   &\E\left[\big|\hat U_t(x)-\hat U_t(y)\big|^2\right]=\E\left[ |U(t,x)-U(t,y)|^2 \right]\notag\\
   &=2C_{H_0}c_{3,1}\int_{\R^d} \frac{\big[ 1-\cos\big( \langle x-y,\xi\rangle \big) \big]\prod_{j=1}^{d}c_{H_j}|\xi_j|^{1-2H_j}}
   {|\xi|^{2\alpha+\frac\alpha\beta(2\gamma+2H_0-2)}} \ud\xi.
 \end{align}
Hence the Gaussian random field $\hat U_t $ has stationary increments with spectral density 
  \begin{equation*}
    f_{\hat U}(\xi)= \frac{C\prod_{j=1}^{d}c_{H_j}|\xi_j|^{1-2H_j}}{|\xi|^{2\alpha+\frac\alpha\beta(2\gamma+2H_0-2)}}.
  \end{equation*}
  This proves the first assertion. Moreover, for every constant $c > 0$,
  \begin{equation*}
  f_{\hat U}(c^\frac{1}{\theta_2}\xi)=c^{-(2+\frac{d}{\theta_2})}f_{\hat U}(\xi).
  \end{equation*}
  Thus, the SLND property follows directly from Theorem 4.1 in \cite[Appendix A]{Herrell2020sharp}.

 By \eqref{e:up-bound-x}, we have
 \begin{equation*}
   \E\left[\left|\hat U_t(x)-\hat U_t(y)\right|^2\right]=\E\left[ |U(t,x)-U(t,y)|^2 \right]\le C |x-y|^{2\theta_2}.
 \end{equation*}
 Since $\hat U_t(0)=0$, then $\E\left[ |\hat U_t(x)|^2 \right] \le C |x|^{2\theta_2}$ and $\hat U_t(x)$ is well defined for all $x\in\R^d$.
 The lower bound is a direct consequence of the SLND property. The proof is complete.
\end{proof}

From Lemma \ref{le:sLND-hat-U}, we prove the following regularity properties of the Gaussian random field 
$\hat U_t =\big\{\tilde U_t(x),x\in\R^d\big\}$.
 
\begin{proposition}\label{prop:modulus-space-U}
   Assume the conditions of Lemma \ref{le:sLND-hat-U} hold and recall that $\theta_2$ is given in \eqref{e:theta12}. For any
   fixed $t\ge0$, we have the following spatial regularity results:
  \begin{enumerate}
    \item (Uniform modulus of continuity). For any $M>0$, there exists a constant $k_4\in(0,\infty)$  such that
    \begin{equation*}
      \lim_{\e\to0^+}\sup_{\substack{x,y\in [-M,M]^d;\\ |x-y|\le\e}}\dfrac{|\hat U_t(x)-\hat U_t(y)|}{|x-y|^{\theta_2}\sqrt{\log\left(1+|x-y|^{-1}\right)}}=k_4,
      \text{ a.s. }
    \end{equation*}
    \item (Local modulus of continuity). There exists a constant $k_5\in(0,\infty)$ such that for any $x\in\R^d$,
    \begin{equation*}
      \lim_{\e\to0^+}\sup_{|y|\le \e} \dfrac{|\hat U_t(x+y)-\hat U_t(x)|}{|y|^{\theta_2} \sqrt{\log\log\left(1+|y|^{-1}\right)}}=k_5,
      \text{ a.s. }
    \end{equation*}
    \item (Chung's LIL). There exists a constant $k_6\in(0,\infty)$ such that for any $x\in\R^d$,
    \begin{equation*}
    \liminf_{\e\to0^+}\sup_{|y|\le \e} \dfrac{|\hat U_t(x+y)-\hat U_t(x)|}{\e^{\theta_2}\left(\log\log \e^{-1}\right)^{-\theta_2/d}}=k_6, \text{ a.s. }
  \end{equation*}
  \end{enumerate}
\end{proposition}

\begin{proof}
  Thanks to Lemma \ref{le:sLND-hat-U}, by using \cite[Theorem 4.1, Theorem 5.6]{Xiao2013Fernique} for modulus of continuity and
   \cite[Theorem 1.1]{Xiao2010Chung} for Chung's LIL,  we can derive the above regularity results for $\hat U_t(x)$. 
   This finishes the proof.
\end{proof}

The following proposition shows that, as a function of the space variable $x$,  $\hat V_t(x)$ is more regular than $ \hat U_t(x)$.
\begin{proposition}\label{prop:reg-hat-V}
Assume the conditions of Lemma \ref{le:sLND-hat-U} hold. Then the 
Gaussian random field $\{\hat V_t(x), x\in\R^d\} $ is well-defined and there exists a finite constant $C > 0$ such that 
for all $t >0$ and  $x,y\in\R^d$ with $|x-y|\le e^{-1}$,   
  \begin{equation*}
    \begin{split}
       &\E\left[ \big|\hat V_t(x)-\hat V_t(y)\big|^2 \right] \\
         &\qquad\le
    \begin{cases}
      C\:t^{2H_0+2\gamma-2}\times|x-y|^{2\alpha-2d+2H}, & \mbox{if } 0<\alpha-d+H<1, \\
       C\:t^{2H_0+2\gamma-2}\times|x-y|^2 \log\left(|x-y|^{-1}\right), & \mbox{if } \alpha-d+H=1, \\
      C\:t^{2\beta+2\gamma+2H_0-2-\frac{\beta(2d-2H+2)}{\alpha}}\times|x-y|^2, & \mbox{if } \alpha-d+H>1.
    \end{cases}
    \end{split}
  \end{equation*}
Consequently, the function $x\mapsto \hat V_t(x)$ is smoother than $x\mapsto \hat U_t(x)$ on $\R^d$.
\end{proposition}
\begin{proof}
  By \eqref{e:de-hat-u}, we have
  \begin{align*}
    &\E\left[\big|\hat V_t(x)\big|^2\right]=\E\left[\big|\hat{U}_t(x)+u(t,0)-u(t,x)\big|^2\right]\\
    &\le \E\left[\big|\hat{U}_t(x)\big|^2\right]+\E\left[\left|u(t,0)\right|^2\right]+\E\left[\left|u(t,x)\right|^2\right]<\infty,
  \end{align*}
  under conditions \eqref{e:DL} and \eqref{e:hat-U-exist}, due to Theorem \ref{th:ex-solution} and Lemma \ref{le:sLND-hat-U}.
  Then, $\hat{V}_t(x)$ is well-defined.
  By \eqref{e:de-hat-V}, we obtain that
  \begin{equation*}
    \E\left[ \big|\hat V_t(x)-\hat V_t(y)\big|^2 \right]=\E\left[ \left| V(t,x)-V(t,y)\right|^2 \right],
  \end{equation*}
  and then the desired result follows from \eqref{e:upb-vtxvty}. This completes the proof.
\end{proof}

Now, we state the main result in this subsection.
\begin{theorem}[Spatial regularity]\label{th:modulus-space}
   Assume the conditions of Lemma \ref{le:sLND-hat-U}  hold and recall that $\theta_2$ is given in \eqref{e:theta12}. For any
    fixed $t>0$, we have the following spatial regularity results:
  \begin{enumerate}
    \item (Uniform modulus of continuity). For any $M>0$, there exists a constant $k_4\in(0,\infty)$ such that
    \begin{equation*}
      \lim_{\e\to0^+}\sup_{\substack{x,y\in [-M,M]^d;\\ |x-y|\le\e}}\dfrac{|u(t,x)-u(t,y)|}{|x-y|^{\theta_2}\sqrt{\log\left(1+|x-y|^{-1}\right)}}=k_4,
      \text{ a.s. }
    \end{equation*}
    \item (Local modulus of continuity). There exists a constant $k_5\in(0,\infty)$ such that for any $x\in\R^d$,
    \begin{equation*}
      \lim_{\e\to0^+}\sup_{|y|\le \e} \dfrac{|u(t,x+y)-u(t,x)|}{|y|^{\theta_2} \sqrt{\log\log\left(1+|y|^{-1}\right)}}=k_5,
      \text{ a.s. }
    \end{equation*}
    \item (Chung's LIL). There exists a constant $k_6\in(0,\infty)$ such that for any $x\in\R^d$,
    \begin{equation*}
    \liminf_{\e\to0^+}\sup_{|y|\le \e} \dfrac{|u(t,x+y)-u(t,x)|}{\e^{\theta_2}\left(\log\log \e^{-1}\right)^{-\theta_2/d}}=k_6, \text{ a.s. }
  \end{equation*}
  \end{enumerate}
\end{theorem}

\begin{proof}
  Since  $\hat V_t(x)$  is smoother than $\hat U_t(x)$ in space   by  Proposition  \ref{prop:reg-hat-V}.  The regularity properties of
  $\hat u_t(x)=\hat U_t(x)-\hat V_t(x)$ on $\R^d$ are determined  by the increment-stationary field $\hat U_t(x)$ (see Proposition
  \ref{prop:modulus-space-U}). By the definition of $\hat u_t(x)$, we have
   \begin{equation*}
     \hat u_t(x)-\hat u_t(y)=u(t,x)-u(t,y).
   \end{equation*}
   Thus, we have the above regularity results for $u(t,x)$ which are the same as $\hat u_t(x)$ and $\hat U_t(x)$. The proof is complete.
\end{proof}

\begin{remark}[A comparison of Theorem \ref{th:modulus-space} with the known results]\label{re:comp-space}\hfill
 \begin{enumerate}
   \item For the fractional heat equation (i.e., $\beta=1$ and $\gamma=0$), the uniform and local moduli of continuity in space
   are consistent with \cite[Corollary 3.9]{Herrell2020sharp}.
   \item If $\alpha=2$, $\gamma=1-\beta$ and the noise is space-time white noise (i.e., $H_0=H_1=\cdots=H_d=\frac{1}{2}, \,
   d=1,2,3$),  condition \eqref{e:U-exist} implies $\gamma=1-\beta<1-H_0\le\frac{1}{2}$ and hence $\beta>\frac{1}{2}$.   Thus,
   the spatial regularity results presented in Theorem~\ref{th:modulus-space} complement  Theorem 1.5 and Theorem 1.6 in
   \cite{Xiao2017JDE} where the case $\beta\le1/2$ is studied.
   \end{enumerate}
\end{remark}

%
%

\section{Small ball probability}\label{se:smallball}

In this section, we study the small ball probabilities for the solution $u(t,x)$ of \eqref{e:fde} as stated in Theorem \ref{th:small-ball}.
In order to prove small ball probabilities for $u(t,x)=U(t,x)-V(t,x)$, we first study small ball probabilities for $U(t,x)$ and $V(t,x)$,
respectively. Using \cite[Theorem 5.1]{Xiao2009Sample}, we can get the following results for $\{U(t,x), (t, x)  \in \R\times \R^d\}$.

\begin{proposition}\label{prop:small-ball-U}
Let $\{U(t,x), (t, x)  \in \R\times \R^d\}$ be the Gaussian random field defined in \eqref{e:de-U} and recall that $\theta_1$ and $\theta_2$
are given in \eqref{e:theta12}.
  \begin{enumerate}[(1)]
    \item (Small ball probability in time).
    Assume the conditions \eqref{e:DL} and \eqref{e:tilde-U-exist} hold. Then, we have
    \begin{equation*}
       \lim_{\e\to0}\e^{1/\theta_1}\log\P\left\{\max_{t\in[0,1]}|U(t,x)-U(0,x)|\le\e\right\}= -C_1^{1/2\theta_1}\kappa ,
    \end{equation*}
    where $C_1$ is defined in \eqref{e:de-C1} and $\kappa$ is the small ball constant of a fractional Brownian motion with index
    $\theta_1$ (e.g. \cite[Theorem 6.9]{Li2001Gaussian}).
    \item (Small ball probability in space). Assume the conditions \eqref{e:DL} and \eqref{e:hat-U-exist} hold. There exist positive
    constants $c$ and $C$ independent of $t$ such that for all $t > 0$ and $\e>0$,
    \begin{equation*}
      \exp\left(-\frac{c}{\e^{d/\theta_2}}\right)\le \P\left\{\max_{x\in[-1,1]^d}|U(t,x)-U(t,0)|\le\e\right\}\le \exp\left(-\frac{C}{\e^{d/\theta_2}}\right).
    \end{equation*}
    \item (Space-time joint small ball probability). Assume the condition \eqref{e:DL} and \eqref{e:U-exist} hold. There exist
    positive constants $c$ and $C$ such that for $\e>0$,
    \begin{equation*}
      \exp\left(-\frac{c}{\e^{Q}}\right)\le \P\left\{\max_{t\in[0,1], x\in[-1,1]^d}|U(t,x)|\le\e\right\}\le \exp\left(-\frac{C}{\e^Q}\right),
    \end{equation*}
    where $Q=\frac{1}{\theta_1}+\frac{d}{\theta_2}$.
  \end{enumerate}
\end{proposition}

\begin{proof}
  (1) Since $\left\{\frac{\tilde{U}_x(t)}{\sqrt{C_1}},t\ge0\right\}=\left\{\frac{U(t,x)-U(0,x)}{\sqrt{C_1}},t\ge0\right\}$ is a fractional
  Brownian motion by Lemma \ref{le:sLND-tilde-U}, the small ball probability in time can be deduced from \cite[Theorem 6.9]{Li2001Gaussian}.

  (2) Next, we prove the small ball probability in space. By the norm equivalence, i.e.,
  \begin{equation*}
   c \sum_{i=1}^{d}|x_i-x^k_i|^{2\theta_2}\le |x-y|^{2\theta_2}\le C\sum_{i=1}^{d}|x_i-x^k_i|^{2\theta_2},
  \end{equation*}
  and \eqref{e:bound-hat-U-space}, we have
  \begin{equation}\label{e:con-sb-space-c1}
   c \sum_{i=1}^{d}|x_i-x^k_i|^{2\theta_2}\le \E\left[ \left|\hat{U}_t(x)-\hat{U}_t(y)\right|^2 \right]\le C\sum_{i=1}^{d}|x_i-x^k_i|^{2\theta_2}.
  \end{equation}
 Now, the small ball probability in space is a direct consequence of \eqref{e:con-sb-space-c1}, \eqref{e:sLND-hat-U}  and
 \cite[Theorem 5.1]{Xiao2009Sample}. Since $U(t,x)-U(t,0)=\hat{U}_t(x)-\hat{U}_t(0)$ and the fact that \eqref{e:indepence-t}
 does not contain $t$, the distribution of $U(t,x)-U(t,0)$ is independent of $t$.

 (3) Finally, using \eqref{e:bounds-U} and the norm equivalence, we can obtain that
  \begin{equation}\label{e:con-sb-both-c1}
     c\bigg(|t-s|^{2\theta_1}+\sum_{i=1}^{d}|x_i-x^k_i|^{2\theta_2}\bigg) \le \E\left[ |U(t,x)-U(s,y)|^2 \right]\le C\bigg(|t-s|^{2\theta_1}
     +\sum_{i=1}^{d}|x_i-x^k_i|^{2\theta_2}\bigg),
  \end{equation}
and  the space-time joint small ball probability follows from  \eqref{e:sLND}, \eqref{e:con-sb-both-c1}  and \cite[Theorem~5.1]{Xiao2009Sample}.
This finishes the proof.
\end{proof}

Next, we study the small ball probability of $V(t,x)$.

\begin{proposition}\label{prop:V-sb}
Recall that $\{V(t,x), t \ge 0, x \in \R^d\}$ and $\theta_2$ are defined in \eqref{e:de-V} and \eqref{e:theta12}, respectively.

\begin{enumerate}[(1)]
 \item (Small ball probability in time). Assume the conditions \eqref{e:DL} and \eqref{e:tilde-U-exist} hold. Assume additionally either $\beta+\gamma-2\le0$
 or \begin{equation*}
      \begin{cases}
        \beta+\gamma-2>0, \\
        -H_0+\frac{\beta}{\alpha}(d-H)<1.
      \end{cases}
    \end{equation*}
  Then, there exists a constant $C>0$ such that for every $x \in \R^d$,
  \begin{equation*}
    \mathbb{P}\bigg\{\sup_{t\in[0,1]}|V(t,x)-V(0,x)|\le\e \bigg\}\ge \exp\left(-\frac{C}{\e}\right) \text{ for all }\e>0.
  \end{equation*}
  \item (Small ball probability in space). Assume the conditions \eqref{e:DL} and \eqref{e:hat-U-exist} hold. For $\lambda=(\alpha-d+H)\wedge1$,
  there exists a constant $C>0$ depending on $t>0$ such that
  \begin{equation*}
    \mathbb{P}\bigg\{\sup_{x\in[-1,1]^d}|V(t,x)-V(t,0)|\le\e \bigg\}\ge \exp\left(-\frac{C}{\e^{d/\lambda}}\right) \ \text{ for all }\e>0.
  \end{equation*}
  \item  (Space-time joint small ball probability). Assume the condition \eqref{e:DL} and \eqref{e:U-exist}. Assume additionally either $\beta+\gamma-2\le0$ or
  \begin{equation*}
      \begin{cases}
        \beta+\gamma-2>0, \\
        -H_0+\frac{\beta}{\alpha}(d-H)<1.
      \end{cases}
    \end{equation*}
    For $\lambda=(\alpha-d+H)\wedge1$, there exists a constant $C>0$ such that
  \begin{equation*}
    \mathbb{P}\bigg\{\sup_{(t,x)\in[0,1]\times[-1,1]^d}|V(t,x)|\le\e \bigg\}\ge \exp\left(-\frac{C}{\e^{1+d/\lambda}}\right)\ \text{ for all }\e>0.
  \end{equation*}
\end{enumerate}
\end{proposition}

\begin{proof}
(1) Recall that $\tilde{V}_x(t)$ is defined in \eqref{e:de-tilde-V}. Let $\cN_1(\e)$ denote the metric entropy
of the process $\{\tilde{V}_x(t)\}_{t\in[0,1]}$, i.e.,  the smallest number of open balls of radius $\e>0$ that are needed
 to cover $[0,1]$ which is endowed with the distance
\begin{equation}\label{e:de-d1}
  d_1(s,t):=\E\left[|\tilde{V}_x(t)-\tilde{V}_x(s)|^2\right]^{1/2}=\E\left[|V(t,x)-V(s,x)|^2\right]^{1/2}.
\end{equation}
Recall that $\theta_1$ is given in \eqref{e:theta12}. Let
  \begin{equation*}
    \tilde{\lambda}=
    \begin{cases}
      1-\theta_1, & \mbox{if } \beta+\gamma-2\le0, \\
      -H_0+\frac{\beta}{\alpha}(d-H), & \mbox{if } \beta+\gamma-2>0.
    \end{cases}
  \end{equation*}
  Note that $\tilde{\lambda}\in(0,1)$.  Define $a_0=0$, $a_1=1$ and $a_{n+1}=a_n+ca_n^{\tilde{\lambda}}$ for every
  $n\in\N$, where $c$ is given in \eqref{e:d1-pre}. Also, define
  \begin{equation*}
       t_j=c^{-1} (2\e/c^{1+\tilde{\lambda}})^\frac{1}{1-\tilde{\lambda}}\, a_j, \text{ for } j=0,1,2,\cdots.
  \end{equation*}
  According to \eqref{e:d1-pre} and \eqref{e:de-d1}, we have
  \begin{equation*}
    d_1(t_j,t_{j+1})\le c|t_{j+1}-t_j|/t_j^{\tilde{\lambda}}=2\e, \text{ for } j=0,1,2,\cdots.
  \end{equation*}
Noting that $t_0=0$ and $t_j$ is increasing in $j$,  it follows readily from the definition of $\cN_1(\e)$  that
$\cN_1(\e)\le 1+\max\{j\ge0:a_j\le c[ c^{1+\tilde{\lambda}}/(2\e)]^\frac{1}{1-\tilde{\lambda}}\}$. Noting that by
Lemma~\ref{le:an-asym}, $a_j\ge C j^\frac{1}{1-\tilde{\lambda}}$ for all $j$ for some constant $C$, we can
find $\hat{C}>0$ such that $\cN_1(\e)\le \hat{C}/\e$. Then  applying Lemma~\ref{le:talagrand} with
$\Psi(\e)=\hat{C}/\e$, we obtain the small ball probability in time
\begin{align*}
  &\mathbb{P}\bigg\{\sup_{t\in[0,1]}|V(t,x)-V(0,x)|\le\e \bigg\}=\mathbb{P}\bigg\{\sup_{t\in[0,1]}|\tilde{V}_x(t)|\le\e \bigg\}  \\
   &\ge\mathbb{P}\bigg\{\sup_{s,t\in[0,1]}|\tilde{V}_x(t)-\tilde{V}_x(s)|\le\e \bigg\} \ge \exp\left(-\frac{C}{\e}\right).
\end{align*}

(2) Recall that $\hat{V}_t(x)$ is defined in \eqref{e:de-hat-V}. Let $\cN_2(\e)$ denote   the smallest number of
open balls of radius $\e>0$ that are needed to cover $[-1,1]^d$  which is endowed with the distance
\begin{equation*}
  d_2(x,y):=\E\left[|\hat{V}_t(x)-\hat{V}_t(y)|^2\right]^{1/2}=\E\left[|V(t,x)-V(t,y)|^2\right]^{1/2}.
\end{equation*}
Let $\lambda=(\alpha-d+H)\wedge1$. By \eqref{e:upb-vtxvty}, we have  $d_2(x,y)\le c|x-y|^\lambda$,
where  $c>0$ depending on $t$.
  If $d_2(x,y)=\e$, then $|x-y|\ge (c^{-1}\e)^{1/\lambda}$. It follows readily that there exists $C>0$ depending on $t$ such that
  \begin{equation*}
    \cN_2(\e)\le \frac{2^d}{|x-y|^d}\le \frac{2^d}{\e^{d/\lambda}}\le \frac{C}{\e^{d/\lambda}}.
  \end{equation*}
  Applying Lemma \ref{le:talagrand} with $\Psi(\e)=C/\e^{d/\lambda}$, we have
  \begin{align*}
   &\mathbb{P}\bigg\{\sup_{x\in[-1,1]^d}|V(t,x)-V(t,0)|\le\e \bigg\}=\mathbb{P}\bigg\{\sup_{x\in[-1,1]^d}|\hat{V}_t(x)|\le\e \bigg\}\\
   &\ge\mathbb{P}\bigg\{\sup_{x,y\in[-1,1]^d}|\hat{V}_t(x)-\hat{V}_t(y)|\le\e \bigg\} \ge \exp\left(-\frac{C}{\e^{d/\lambda}}\right).
  \end{align*}

  (3) Let $\cN(\e)$  the smallest number of open balls of radius $\e>0$ that are needed to cover $[0,1]\times[-1,1]^d$ which is endowed
  with the distance
   $$ d((t,x),(s,y)):=\E\left[|V(t,x)-V(s,y)|^2\right]^{1/2}.$$
Let $\lambda=(\alpha-d+H)\wedge1$. Then by the triangle inequality, we have
$$ d((t,x),(s,y))\le d_1(s,t)+d_2(x,y).$$
  Note that $\cN_1(\e)\le \frac{C}{\e}$ and $\cN_2(\e)\le \frac{C}{\e^{d/\lambda}}$. Thus, $\cN(\e) \le \frac{C}{\e^{1+d/\lambda}}$.
  Applying Lemma \ref{le:talagrand} with $\Psi(\e)=\frac{C}{\e^{1+d/\lambda}}$, we obtain that
  \begin{align*}
     & \mathbb{P}\bigg\{\sup_{(t,x)\in[0,1]\times[-1,1]^d}|V(t,x)|\le\e \bigg\} \\
     &\ge\mathbb{P}\bigg\{\sup_{(t,x),(s,y)\in[0,1]\times[-1,1]^d}|V(t,x)-V(s,y)|\le\e \bigg\}\ge \exp\left(-\frac{C}{\e^{1+d/\lambda}}\right).
  \end{align*}
  This completes the proof.
\end{proof}

Now, we can state and prove the main result of this section.

\begin{theorem}\label{th:small-ball}
Let $\{u(t,x), t \ge 0, x \in \R^d\} $ be the solution of \eqref{e:fde} and recall that $\theta_1$ and $\theta_2$  are given in \eqref{e:theta12}.
  \begin{enumerate}[(1)]
    \item (Small ball probability in time). Assume the conditions \eqref{e:DL} and \eqref{e:tilde-U-exist} hold. Assume additionally
    either $\beta+\gamma-2\le0$ or
    \begin{equation*}
      \begin{cases}
        \beta+\gamma-2>0, \\
        -H_0+\frac{\beta}{\alpha}(d-H)<1.
      \end{cases}
    \end{equation*}
        Then,
    \begin{equation*}
      \lim_{\e\to0} \e^{1/\theta_1}\log\P\bigg\{\max_{t\in[0,1]}|u(t,x)|\le\e\bigg\} = -C_1^{1/2\theta_1}\kappa ,
    \end{equation*}
    where $C_1$ is defined in \eqref{e:de-C1} and $\kappa$ is the small ball constant of a fractional Brownian motion
    with index $\theta_1$ (e.g. \cite[Theorem 6.9]{Li2001Gaussian}).
    \item (Small ball probability in space). Assume the conditions \eqref{e:DL} and \eqref{e:hat-U-exist} hold. There exist
    positive constants $c$ and $C$ independent of $t>0$ such that for sufficiently small $\e>0$,
    \begin{equation*}
      \exp\left(-\frac{c}{\e^{d/\theta_2}}\right)\le \P\bigg\{\max_{x\in[-1,1]^d}|u(t,x)-u(t,0)|\le\e\bigg\}\le \exp\left(-\frac{C}{\e^{d/\theta_2}}\right).
    \end{equation*}
    \item (Space-time joint small ball probability). Assume the conditions \eqref{e:DL} and \eqref{e:U-exist} hold. Assume additionally
    either $\beta+\gamma-2\le0$ or
    \begin{equation*}
      \begin{cases}
        \beta+\gamma-2>0, \\
        -H_0+\frac{\beta}{\alpha}(d-H)<1.
      \end{cases}
    \end{equation*}
    Then, there exist positive constants $c$ and $C$ such that for sufficiently small $\e>0$,
    \begin{equation*}
      \exp\left(-\frac{c}{\e^{Q}}\right)\le \P\left\{\max_{(t,x)\in [0,1]\times[-1,1]^d}|u(t,x)|\le\e\right\}\le \exp\left(-\frac{C}{\e^{Q}}\right),
    \end{equation*}
  where $Q=\frac{1}{\theta_1}+\frac{d}{\theta_2}$.
  \end{enumerate}
\end{theorem}

\begin{proof}
{\em Step 1:} Let $\rho\in(0,1)$ be an arbitrary real number. By \eqref{e:u=tildeUV}, we have
\begin{align*}
   & \P\left\{\max_{t\in[0,1]}|u(t,x)|\le\e\right\}\\
   &\ge \P\left\{\max_{t\in[0,1]}|\tilde{U}_x(t)|\le\rho\e,\: \max_{t\in[0,1]}|\tilde{V}_x(t)|\le(1-\rho)\e\right\}\\
   &\ge \P\left\{\max_{t\in[0,1]}|\tilde{U}_x(t)|\le\rho\e\right\}\times\P\left\{\max_{t\in[0,1]}|\tilde{V}_x(t)|\le(1-\rho)\e\right\}\\
   &=\P\left\{\max_{t\in[0,1]}|U(t,x)-U(0,x)|\le\rho\e\right\}\times\P\left\{\max_{t\in[0,1]}|V(t,x)-V(0,x)|\le(1-\rho)\e\right\},
\end{align*}
where in the second inequality we use the Gaussian correlation inequality (see Remark \ref{re:GCI}).  Then, by Proposition
\ref{prop:small-ball-U} (1)  and Proposition \ref{prop:V-sb} (1), we have
\begin{align}\label{e:upb-sb-u-t}
   &\liminf_{\e\to0}\e^{1/\theta_1}\log\P\left\{\max_{t\in[0,1]}|u(t,x)|\le\e\right\}\ge -\frac{ C_1^{1/2\theta_1}\kappa}{\rho^{1/\theta_1}}.
\end{align}
  Let $\tilde\rho>1$ be an arbitrary real number. For the upper bound, applying \eqref{e:u=tildeUV} and the Gaussian correlation inequality again, we get
  \begin{align*}
   & \P\left\{\max_{t\in[0,1]}|\tilde{U}_x(t)|\le\tilde\rho\e\right\}\ge \P\left\{\max_{t\in[0,1]}|u(t,x)|\le\e,\: \max_{t\in[0,1]}|\tilde{V}_x(t)|\le(\tilde\rho-1)\e\right\}\\
   &\ge \P\left\{\max_{t\in[0,1]}|u(t,x)|\le\e\right\}\times \P\left\{\max_{t\in[0,1]}|\tilde{V}_x(t)|\le(\tilde\rho-1)\e\right\}.
\end{align*}
Then, using Proposition \ref{prop:small-ball-U} (1)  and Proposition \ref{prop:V-sb} (1), we have
\begin{align}\label{e:lowb-sb-u-t}
   & \limsup_{\e\to0}\e^{1/\theta_1}\log \P\left\{\max_{t\in[0,1]}|u(t,x)|\le\e\right\}\le -\frac{C_1^{1/2\theta_1}\kappa }{\tilde\rho^{1/\theta_1}}.
\end{align}
Combining \eqref{e:upb-sb-u-t} with \eqref{e:lowb-sb-u-t}, we have
\begin{align*}
     -\frac{ C_1^{1/2\theta_1}\kappa}{\rho^{1/\theta_1}} &\le \liminf_{\e\to0}\e^{1/\theta_1}\log\P\left\{\max_{t\in[0,1]}|u(t,x)|\le\e\right\}\\
     &\le \limsup_{\e\to0}\e^{1/\theta_1}\log\P\left\{\max_{t\in[0,1]}|u(t,x)|\le\e\right\}\le -\frac{C_1^{1/2\theta_1}\kappa }{\tilde\rho^{1/\theta_1}}.
\end{align*}
The desired small ball probability in time follows from the fact that both $\rho\in(0,1)$ and $\tilde\rho>1$ can be arbitrarily close to 1.

Similarly, using Proposition \ref{prop:small-ball-U}, Proposition \ref{prop:V-sb}, and the Gaussian correlation inequality, we can also
prove the rest two assertions. This finishes the proof.
\end{proof}

\begin{remark}[A comparison of Theorem \ref{th:small-ball} with the known result]
Considering $1$-d stochastic heat equation with space-time white noise  (i.e., $\alpha=2$, $\beta=1$, $\gamma=0$, $H_0=H=1/2$, $d=1$),
there exist positive constants $c$ and $C$ such that for sufficiently small $\e>0$,
    \begin{equation*}
      \exp\left(-\frac{c}{\e^{6}}\right)\le \P\left\{\max_{(t,x)\in [0,1]\times[-1,1]^d}|u(t,x)|\le\e\right\}\le \exp\left(-\frac{C}{\e^{6}}\right).
    \end{equation*}
This coincides with \cite[Theorem 1.1]{Athreya2021Smallball}.
\end{remark}

\begin{remark}\label{re:GCI}
The Gaussian correlation inequality (GCI), which had been a long-standing conjecture until it was proved rigorously by Royen \cite{Royen2014simpleproof}
(see also \cite{Lata2017Royen}),  has been actively investigated and applied in the study of small ball probabilities. For example, Li \cite{Li1999Gaussian}
proved a weak-version of GCI and applied it to estimate small ball probabilities;   Shao \cite{Shao2003Gaussian} obtained another weak-version of
GCI and then proved the existence of the small ball constant for the fractional Brownian motion; as for the applications of GCI for the stochastic heat equation,
we refer to  \cite{Athreya2021Smallball,Carfagnini2022application}.
\end{remark}

\begin{appendix}
\section{Some miscellaneous results}\label{se:appendix}

In this section, we provide some technical lemmas. The following result is borrowed from \cite[Lemma 5.1.1]{nualart06} (see also \cite{mmv01}).

\begin{lemma}\label{le:B.2}
  For any $\varphi\in L^{1/{H_0}}(\R^n)$ with $H_0\in(1/2,1)$,
  \begin{equation*}
    \int_{\R}\int_{\R}\varphi( t)\varphi(s)|t-s|^{2H_0-2}\ud t\ud s
    \le C_{H_0} \left( \int_{\R}\left|\varphi(t)\right|^{1/H_0}\ud t \right)^{2H_0},
  \end{equation*}
where $C_{H_0}>0$ is a constant depending on $H_0$.
\end{lemma}

\begin{lemma}\label{le:space-integral}
Let $H_j\in(0,1)$ with $j=1,\dots,d$ and $H=\sum_{j=1}^{d}H_j$.
\begin{enumerate}[(i)]
  \item The integral
  \begin{equation*}
    \int_{\R^d} \frac{\prod_{j=1}^{d}c_{H_j} |\xi_j|^{1-2H_j}}{1+|\xi|^{2\alpha}} \ud\xi
  \end{equation*}
  is convergent if and only if $\alpha>d-H$.
  \item   The integral
  \begin{equation*}
    \int_{\R^d} \frac{\prod_{j=1}^{d}c_{H_j} |\xi_j|^{1-2H_j}}{1+|\xi|^{2\alpha}}|\xi|^2 \ud\xi
  \end{equation*}
  is convergent if and only if $\alpha>d-H+1$.
  \item The following inequalities
  \begin{equation*}
    \int_{\mathbb{S}^{d-1}} \prod_{j=1}^{d}c_{H_j}|w_j|^{1-2H_j} \sigma(\ud w)<\infty,
  \end{equation*}
  and
  \begin{equation*}
    \int_{\mathbb{S}^{d-1}} |w_i|^2\prod_{j=1}^{d}c_{H_j}|w_j|^{1-2H_j} \sigma(\ud w)<\infty, \quad \text{for } i=1,\dots,d,
  \end{equation*}
  hold, where $\sigma(\ud w)$ is the uniform measure on the $(d-1)$-dimensional unit sphere $\mathbb{S}^{d-1}$.
\end{enumerate}
\end{lemma}

\begin{proof}
 Using the $d$-dimensional spherical coordinate: $r\in[0,\infty)$, $\theta\in[0,2\pi)$ and $\phi_1, \dots, \phi_{d-2}\in[0, \pi]$,
\begin{equation}\label{e:changevariable}
 \begin{cases}
  &\xi_1=rw_1 = r\cos(\phi_1),\\
  &\xi_2=rw_2 = r\sin(\phi_1)\cos(\phi_2),\\
  &\xi_3=rw_3 =r \sin(\phi_1)\sin(\phi_2)\cos(\phi_3),\\
  & \vdots\\
  &\xi_{d-2}=rw_{d-2}=r\sin(\phi_1)\dots\sin(\phi_{d-3})\cos(\phi_{d-2}),\\
  &\xi_{d-1}=rw_{d-1}  = r \sin(\phi_1)\dots\sin(\phi_{d-3})\sin(\phi_{d-2})\cos(\theta),\\
  &\xi_{d}=rw_{d}  = r\sin(\phi_1)\dots\sin(\phi_{d-3})\sin(\phi_{d-2})\sin(\theta),
 \end{cases}
\end{equation}
we have
\begin{align*}
   &\int_{\R^d} \frac{\prod_{j=1}^{d}c_{H_j} |\xi_j|^{1-2H_j}}{1+|\xi|^{2\alpha}} \ud\xi= \int_{0}^{\infty}\frac{r^{2d-2H-1}}{1+r^{2\alpha}} \ud r
   \times \int_{\mathbb{S}^{d-1}} \prod_{j=1}^{d}c_{H_j}|w_j|^{1-2H_j} \sigma(\ud w),
\end{align*}
where the first integral is finite if and only if $\alpha>d-H$, and for the second one we have
\begin{equation}\label{e:sphere-int}
\begin{aligned}
   & \int_{\mathbb{S}^{d-1}} \prod_{j=1}^{d}c_{H_j} |w_j|^{1-2H_j}\sigma(\ud w)\\
   & = \prod_{j=1}^{d}c_{H_j}\times\prod_{k=1}^{d-2} \int_{0}^{\pi} \cos^{1-2H_k}(\phi_k) \sin^{\sum_{j=k+1}^{d}1-2H_j}(\phi_k) \sin^{d-1-k}(\phi_k) \ud \phi_k \\
    &\qquad \times \int_{0}^{2\pi} \cos^{1-2H_{d-1}}(\theta) \sin^{1-2H_d}(\theta) \ud\theta.
\end{aligned}
\end{equation}
Since $H_1,H_2,\dots,H_d\in(0,1)$, we have
 \begin{equation*}
   \sum_{j=k+1}^{d} (1-2H_j)+d-1-k>-1,
 \end{equation*}
 for $k=1,2,\dots,d-2$, and
 \begin{equation*}
   1-2H_j>-1,
 \end{equation*}
for $j=1,2,\dots,d$. Then, the integrals with respect to $\phi_1,\dots,\phi_{d-2}$ and $\theta$ are finite. This proves (i). By using a similar argument,
we can also prove the assertions in (ii)-(iii). This completes the proof.
\end{proof}

Recall that the Mittag-Leffler function $E_{a,b}$ is given in \eqref{e:mlf}. The following lemma is an extension of \cite[Lemma 2.2]{Xiao2017JDE}.
\begin{lemma}\label{le:Fourier-MLF}
 (Mittag–Leffler Fourier transform). Assume  $x>0$.
 \begin{enumerate}[(i)]
   \item If $\beta\in(0,2]$ and $\gamma\in\R$, then for any $\sigma>0$, $\tau\in\R$,
   \begin{equation}\label{e:lap-MLF}
     \cF\left[\one_{t>0}\:t^{\beta+\gamma-1}E_{\beta,\beta+\gamma}\big(-xt^\beta\big)e^{-t\sigma}\right](\tau)=\frac{(\sigma+i\tau)^{-\gamma}}{(\sigma+i\tau)^\beta+x}.
   \end{equation}

   \item  If $\beta\in(0,2)$ and $\gamma\in[0,1)$, then for any real number $\tau\neq0$,
   \begin{equation}\label{e:fourier-MLF}
     \cF\left[\one_{t>0}\:t^{\beta+\gamma-1}E_{\beta,\beta+\gamma}\big(-xt^\beta\big)\right](\tau)=\frac{(i\tau)^{-\gamma}}{(i\tau)^\beta+x}.
   \end{equation}
 \end{enumerate}
\end{lemma}
\begin{proof}
  {\em (i):} The equality \eqref{e:lap-MLF}
is a direct consequence of \cite[eq. (7.1)]{hms11}. For the reader's convenience, we also include our proof below.  Assume that $\beta\in(0,2]$ and $\sigma>0$. Then,
  \begin{equation}\label{e:lhs-a1}
    \cF\left[\one_{t>0}\:t^{\beta+\gamma-1}E_{\beta,\beta+\gamma}\big(-xt^\beta\big)e^{-t\sigma}\right](\tau)
  \end{equation}
  is well defined by the asymptotic property of $E_{\beta,\beta+\gamma}\big(-xt^\beta\big)$, see \eqref{e:asym-MLF} for $\beta\in(0,2)$
  and \cite[(1.8.31)]{kilbas2006theory} for $\beta=2$. Applying the Cauchy–Riemann theorem, we can prove that \eqref{e:lhs-a1} is analytic.

  Let $\theta=\sigma+i\tau$. Using Euler’s formula, we have
  \begin{equation*}
    \theta^{\beta+\gamma}+x\theta^\gamma=\left(|\theta|e^{i\omega}\right)^{\beta+\gamma}+x\left(|\theta|e^{i\omega}\right)^\gamma,
  \end{equation*}
  where $\omega=\text{arg}(\theta)\in(-\pi/2,\pi/2)$.  Since $\beta\omega\in(-\pi,\pi)$ and $x>0$, it follows that
  $\theta^{\beta+\gamma}+x\theta^\gamma\neq0$. Hence, the fraction $\dfrac{(\sigma+i\tau)^{-\gamma}}{(\sigma+i\tau)^\beta+x}$
  does not have poles and hence is analytic.

 It follows from \cite[(1.10.9)]{kilbas2006theory} on the Laplace transform of the Mittag–Leffler function that,
 if $|\theta|^{\beta}>x$,
 then
  \begin{equation}\label{e:lap-MlF-proof}
    \begin{split}
      \frac{(\sigma+i\tau)^{-\gamma}}{(\sigma+i\tau)^\beta+x} &=\cL\left[t^{\beta+\gamma-1}E_{\beta,\beta+\gamma}\big(-xt^\beta\big)\right](\theta)  \\
      & =\int_{0}^{\infty}t^{\beta+\gamma-1}E_{\beta,\beta+\gamma}\big(-xt^\beta\big)e^{-t\theta}\ud t \\
      & =\int_{-\infty}^{\infty} \left[\one_{t>0}\: t^{\beta+\gamma-1}E_{\beta,\beta+\gamma}\big(-xt^\beta\big)e^{-t\sigma}\right]e^{-it\tau}\ud t\\
      & =\cF\left[\one_{t>0}\: t^{\beta+\gamma-1}E_{\beta,\beta+\gamma}\big(-xt^\beta\big)e^{-t\sigma}\right](\tau).
    \end{split}
  \end{equation}
Therefore, applying the identity theorem for analytic functions, we can show that \eqref{e:lap-MlF-proof} holds for all $\sigma=\mathfrak{R}(\theta)>0$.
This proves the first assertion.

  \smallskip

  {\em (ii):} Assume that $\beta\in(0,2)$, $\gamma\in[0,1)$ and $\tau\neq0$. Define
  \begin{align*}
     & f(\sigma,t):=t^{\beta+\gamma-1}E_{\beta,\beta+\gamma}\big(-xt^\beta\big)e^{-t\sigma}\cos(t\tau),\quad f(t):=f(0,t), \\
     & g(\sigma,t):=t^{\beta+\gamma-1}E_{\beta,\beta+\gamma}\big(-xt^\beta\big)e^{-t\sigma}\sin(t\tau),\quad g(t):=g(0,t).
  \end{align*}
  If $\sigma\ge0$,   noting that  $\gamma<1$,
  \begin{equation}\label{e:int-fg}
    \cF\left[\one_{t>0}\:t^{\beta+\gamma-1}E_{\beta,\beta+\gamma}\big(-xt^\beta\big)e^{-t\sigma}\right](\tau)
    =\int_{0}^{\infty}f(\sigma,t)\ud t-i\int_{0}^{\infty}g(\sigma,t)\ud t,
  \end{equation}
  is well defined due to \eqref{e:asym-MLF}.

In view of \eqref{e:lap-MLF}, in order to prove \eqref{e:fourier-MLF}, it suffices to show that the last two integrals in  \eqref{e:int-fg}
are continuous at $\sigma=0$. To this end, we divide $\beta\in(0,2)$ and $\gamma\in[0,1)$ into  the following two cases:
  \begin{equation*}
    \begin{cases}
      (1):\beta\in(0,2), \gamma\in(0,1) \text{ or } \beta\in(0,1],\gamma=0; \\
      (2):\beta\in(1,2),\gamma=0.
    \end{cases}
  \end{equation*}
By the asymptotic property \eqref{e:asym-MLF} and the fact that $E_{1,1}(z)=e^{z}$, we have for $\beta\in(0,2)$,  as $z\to\infty$,
 \begin{equation}\label{e:asymp-proof-fourier}
  E_{\beta,\beta+\gamma}(-z)=
  \begin{cases}
    \dfrac{1}{\Gamma(\gamma)z}-\dfrac{1}{\Gamma(\gamma-\beta)z^2}+o\left(\dfrac{1}{z^2}\right), & \mbox{if } \beta\neq1 \text{ or } \gamma\neq0 \\
    e^{-z}, & \mbox{if } \beta=1 \text{ and } \gamma=0.
  \end{cases}
 \end{equation}

For case (1), noting $1/\Gamma(0)=1/\Gamma(-1)=0$ and $\Gamma(x)<0$ for $x\in(-1,0)$,   by \eqref{e:asymp-proof-fourier} there exists a
constant $T$ such that for  $t>T$, $t^{\beta+\gamma-1}E_{\beta,\beta+\gamma}\big(-xt^\beta\big)$ is positive and
bounded by a decreasing function that tends to zero as $t\to\infty$. Let $n\in\mathbb{N}$ such that $\frac{3\pi/2+2n\pi}{|\tau|}>T$.
For any $\sigma>0$, the integral
\begin{equation*}
   \int_{\frac{3\pi/2+2n\pi}{|\tau|}}^{\infty}f(\sigma,t)\ud t=\sum_{k=2n}^{\infty}\int_{\frac{3\pi/2+k\pi}{|\tau|}}^{\frac{5\pi/2+k\pi}{|\tau|}}f(\sigma,t)\ud t
 \end{equation*}
 is an alternating series with a positive leading term. Thus, for any $\e>0$, there exists a constant $N_\e$, such that for any $\sigma>0$ and
 $n\in\left\{n\in \mathbb N:\frac{3\pi/2+2n\pi}{|\tau|}>T,n>N_\e\right\}$,
  \begin{equation*}
  0 < \int_{\frac{3\pi/2+2n\pi}{|\tau|}}^{\infty}f(\sigma,t)\ud t<\int_{\frac{3\pi/2+2n\pi}{|\tau|}}^{\frac{5\pi/2+2n\pi}{|\tau|}}f(\sigma,t)\ud t
  <\int_{\frac{3\pi/2+2n\pi}{|\tau|}}^{\frac{5\pi/2+2n\pi}{|\tau|}}f(t)\ud t <\frac{\e}{4}.
 \end{equation*}
 Let $n_0=\min\left\{n\in \mathbb N:\frac{3\pi/2+2n\pi}{|\tau|}>T,n>N_\e\right\}$. Thus, we have
  \begin{equation*}
    \left|\int_{\frac{3\pi/2+2n_0\pi}{|\tau|}}^{\infty} \big(f(\sigma,t)-f(t) \big)\ud t\right| <\int_{\frac{3\pi/2+2n_0\pi}{|\tau|}}^{\infty}f(\sigma,t)\ud t
    +\int_{\frac{3\pi/2+2n_0\pi}{|\tau|}}^{\frac{5\pi/2+2n_0\pi}{|\tau|}}f(t)\ud t<\frac{\e}{2}.
  \end{equation*}
  For the integral on the finite interval $\left[0,\frac{3\pi/2+2n_0\pi}{|\tau|}\right]$, we have if $\sigma$ is sufficiently small, say
  $\sigma<\delta_\e$ for some $\delta_\e$ depending on $\e$,
 \begin{equation*}
    \left|\int_{0}^{\frac{3\pi/2+2n_0\pi}{|\tau|}} \big( f(\sigma,t)-f(t) \big)\ud t\right|<\frac{\e}{2}.
  \end{equation*}
  Therefore, for any $\e>0$ there exists a constant $\delta_\e$ such that when $\sigma<\delta_\e$,
  \begin{equation*}
    \left|\int_{0}^{\infty} \big( f(\sigma,t)-f(t) \big) \ud t\right|\le \left|\int_{\frac{3\pi/2+2n_0\pi}{|\tau|}}^{\infty} \big( f(\sigma,t)-f(t) \big)\ud t\right|
    +\left|\int_{0}^{\frac{3\pi/2+2n_0\pi}{|\tau|}} \big( f(\sigma,t)-f(t) \big)\ud t\right|<\e.
  \end{equation*}
This shows that $\int_{0}^{\infty}f(\sigma,t)\ud t$ is continuous at $\sigma=0$.

Now we consider case (2). Noting $1/\Gamma(0)=0$ and $\Gamma(x)>0$ for $x\in(-2,-1)$,  by \eqref{e:asymp-proof-fourier} there
exists a constant $T$ such that as $t>T$, $t^{\beta+\gamma-1}E_{\beta,\beta+\gamma}\big(-xt^\beta\big)$ is negative and
$\left|t^{\beta+\gamma-1}E_{\beta,\beta+\gamma}\big(-xt^\beta\big)\right|$ is decreasing to zero as $t\to \infty$. Then, we can apply
the same argument as for case (1) to prove that $\int_{0}^{\infty}f(\sigma,t)\ud t$ is continuous at $\sigma=0$. Similarly, we can
also prove that $\int_{0}^{\infty}g(\sigma,t)\ud t$ is continuous at $\sigma=0$ for both cases (1) and (2). Thus, the integral
\eqref{e:int-fg} is continuous at $\sigma=0$.

  On the other hand, using Euler's formula, we have
  \begin{align*}
    (i\tau)^{\beta+\gamma}+x(i\tau)^\gamma=\left(|\tau|e^{i\tilde{\omega}}\right)^{\beta+\gamma}+x\left(|\tau|e^{i\tilde{\omega}}\right)^\gamma,
  \end{align*}
  where
  \begin{equation*}
    \tilde{\omega}=\text{arg}(i\tau)=
    \begin{cases}
      -\pi/2, & \mbox{if } \tau<0, \\
      \pi/2, & \mbox{if } \tau>0.
    \end{cases}
  \end{equation*}
 For $\beta\in(0,2)$,
  \begin{equation*}
    \beta\tilde{\omega}\in
    \begin{cases}
      (-\pi,0), & \mbox{if } \tau<0, \\
      (0,\pi), & \mbox{if } \tau>0.
    \end{cases}
  \end{equation*}
  Since $x>0$ and $\tau\neq0$, it follows that $(i\tau)^{\beta+\gamma}+x(i\tau)^\gamma\neq0$. Hence, there are no poles in
  $\dfrac{(i\tau)^{-\gamma}}{(i\tau)^\beta+x}$ and the fraction $\dfrac{(\sigma+i\tau)^{-\gamma}}{(\sigma+i\tau)^\beta+x}$ is
  continuous at $\sigma=0$.

  Consequently, when $\beta\in(0,2)$, $\gamma\in[0,1)$ and $\tau\neq0$, the  Fourier transform  \eqref{e:fourier-MLF} follows
  from setting $\sigma=0$ in both sides of  \eqref{e:lap-MLF}.
\end{proof}

The following two lemmas are based on the results in \cite[Sections 1.5 and 1.8]{kilbas2004h} which are summarized in \cite[Appendix A.2]{chen2017space}.

\begin{lemma}\label{le:partial-x-ptx}
  Suppose $\alpha>0$, $\beta\in(0,2)$, $\gamma\ge0$, and $d\in\mathbb{N}$. Let
  \begin{equation*}
    g(x)= x^{-2} \FoxH{2,2}{3,4}{ x^\alpha} {(d,\alpha),\:(1,1),\:(\beta+\gamma,\beta)}{(d/2,\alpha/2),\:(1,1),\:(1,\alpha/2),\:(d+1,\alpha)},
    \quad x>0.
  \end{equation*}
  Then, there exist positive and finite constants $K_1 < K_2$ such that 
  \begin{equation}\label{Eq:small}
    |g(x)|\le C (x^d +x^{\alpha-2}|\log x|), \quad \hbox{ for all } 0< x \le K_1,
  \end{equation}
and 
 \begin{equation}\label{Eq:large}
 |g(x)|\le C x^{-2}, \quad \hbox{ for all } x \ge K_2.
\end{equation}
\end{lemma}
\begin{proof}
  Let
  \begin{equation*}
    f(x)=\FoxH{2,2}{3,4}{ x} {(d,\alpha),\:(1,1),\:(\beta+\gamma,\beta)}{(d/2,\alpha/2),\:(1,1),\:(1,\alpha/2),\:(d+1,\alpha)}.
  \end{equation*}
  Then
   \begin{equation}\label{Eq:g}
  g(x)=x^{-2}f(x^\alpha).
    \end{equation}
    Noting \cite[(1.1.1)]{kilbas2004h} we have
   $a_1=d, a_2=1, a_3=\beta+\gamma,  \alpha_1=\alpha, \alpha_2=1, \alpha_3=\beta, b_1=d/2, b_2=1, b_3=1, b_4=d+1,
   \beta_1=\alpha/2, \beta_2=1, \beta_3=\alpha/2, \beta_4=\alpha$,
    and applying \cite[(1.1.7)]{kilbas2004h}, we have
  \begin{equation*}
    a^*=\alpha+1-\beta+\alpha/2+1-\alpha/2-\alpha=2-\beta>0,
  \end{equation*}
which implies Condition (A5) in the Appendix of \cite{chen2017space}. Thus, we can apply Theorem  1.7  in
\cite[Appendix A.2]{chen2017space} to obtain the asymptotic property of $f(x)$ as $x\to 0^+$.  Note that the poles
\begin{equation*}
  \left\{b_{jl}=-\frac{b_j+l}{\beta_j}, l=0,1,\dots\right\}, \text{ for }j=1,2;
\end{equation*}
may not be simple (see \cite[(A7)]{chen2017space} for the definition), namely,
\begin{equation*}
\text{ for } i,j\in\{1,2\}, ~ \exists i\neq j \text{ such that } \beta_j(b_i+k)=\beta_i(b_j+l) \text{ for some } k,l\in \mathbb N\cup \{0\}.
\end{equation*}
We apply (A11) in \cite[Theorem  1.7]{chen2017space} to obtain that as $x\to 0^+,$
 \begin{equation}\label{e:asymp-fx-zoro-1}
   \begin{split}
      f(x)\sim & \sum_{l\in\N\cup\{0\},l\notin A_1} h_{1l}^* x^\frac{d+2l}{\alpha}+ \sum_{l\in\N\cup\{0\},l\notin A_2} h_{2l}^*x^{1+l}\\
    &  + \sum_{l\in\N\cup\{0\},l\in A_1}\sum_{i=0}^{1} H_{1li}^* x^\frac{d+2l}{\alpha} (\log x)^i
        + \sum_{l\in\N\cup\{0\},l\in A_2 }\sum_{i=0}^{1} H_{2li}^* x^{1+l}(\log x)^i,
   \end{split}
 \end{equation}
where
  \begin{equation*}
    A_1:=\left\{l: l= \frac{\alpha(k+1)-d}{2},k=0,1,2,\cdots\right\} \text{ and } A_2:= \left\{l: l= \frac{d+2k}{\alpha}-1,k=0,1,2,\cdots\right\},
  \end{equation*}
  and $h_{jl}^*$, $H_{jli}^*$ are defined by (A12) and (A13) in \cite{chen2017space}, respectively.
  For the first summation in \eqref{e:asymp-fx-zoro-1}, we have
  \begin{equation*}
    h_{10}^* =\frac{2}{\alpha} \frac{\Gamma(1-d/\alpha)\Gamma(1)\Gamma(d/\alpha)}{\Gamma(\beta+\gamma-d\beta/\alpha)\Gamma(d/2)\Gamma(0)}=0,
  \end{equation*}
   and for the third summation in \eqref{e:asymp-fx-zoro-1}, we have
  \begin{equation*}
    \frac{d+2l}{\alpha}=k+1, \text{ for } l\in A_1.
  \end{equation*}
  Hence,  as $x\to 0^+,$
   \begin{equation}\label{e:asymp-fx-zoro-2}
     \begin{split}
        f(x)\sim & \sum_{l\in\N,l\notin A_1} h_{1l}^* x^\frac{d+2l}{\alpha}+ \sum_{l\in\N\cup\{0\},l\notin A_2} h_{2l}^*x^{1+l}\\
    &  + \sum_{l\in\N\cup\{0\},l\in A_1}\sum_{i=0}^{1} H_{1li}^* x^{k+1} (\log x)^i
        + \sum_{l\in\N\cup\{0\},l\in A_2 }\sum_{i=0}^{1} H_{2li}^* x^{1+l}(\log x)^i.
     \end{split}
   \end{equation}
  Thus, by \eqref{e:asymp-fx-zoro-2}, we derive that for $x$ small enough,
  \begin{equation*}
    |f(x)|\le C\left(x^{(d+2)/\alpha}+x+x|\log x|+x|\log x| \right) \le C\left(x^{(d+2)/\alpha}+x|\log x|\right).
  \end{equation*}
  This and (\ref{Eq:g}) imply (\ref{Eq:small}).

 On the other hand, by \cite[(2.1.7)]{kilbas2004h} we have
  \begin{align*}
    f(x)&=\FoxH{2,2}{3,4}{ x} {(d,\alpha),\:(1,1),\:(\beta+\gamma,\beta)}{(d/2,\alpha/2),\:(1,1),\:(1,\alpha/2),\:(d+1,\alpha)}
    =-\FoxH{3,1}{3,4}{ x} {(1,1),\:(\beta+\gamma,\beta),\:(d,\alpha)}{(d+1,\alpha),\:(d/2,\alpha/2),\:(1,1),\:(1,\alpha/2)}.
  \end{align*}
  Since the second Fox H-function satisfies (A6) in \cite{chen2017space}, one can use \cite[(A8)]{chen2017space} to see that
  \begin{equation*}
    f(x)\sim -\sum_{k=0}^{\infty} h_{1k} x^{-k}, \quad x\to+\infty,
  \end{equation*}
  where $h_{1k}$ is given by \cite[(A9)]{chen2017space}. Thus, $|f(x)|\le C$ for all $x$ large enough.  This and
  (\ref{Eq:g}) yield (\ref{Eq:large}).
\end{proof}

\begin{lemma}\label{le:partial-t-ptx}
  Suppose $\alpha>0$, $\beta\in(0,2)$, $\gamma\ge0$, and $d\in\mathbb{N}$. Let
  \begin{equation*}
    h(x)=x^{-d}\FoxH{2,1}{2,3}{ x^\alpha}{(0,1),\:(\beta+\gamma,\beta)}{(d/2,\alpha/2),\:(1,1),\:(1,\alpha/2)}, \quad x>0.
  \end{equation*}
  Then, there exist positive and finite constants $K_3 < K_4$ such that 
  \begin{equation*}
    |h(x)|\le C(1+x^{\alpha-d}|\log x|), \quad \hbox{ for all } 0< x \le K_3,
  \end{equation*}
  and
  \begin{equation*}
    |h(x)|\le C x^{-\alpha-d}, \quad \hbox{ for all } x \ge K_4.
  \end{equation*}
\end{lemma}
\begin{proof}
  For $h(x)$, by \cite[(1.1.7)]{kilbas2004h}, we have    $a^*=1-\beta+\alpha/2+1-\alpha/2=2-\beta>0.$   Thus, we can use Appendix A.2 in
  \cite{chen2017space} to obtain the asymptotic property of $h(x)$. We omit the details since the proof is similar to that of Lemma
  \ref{le:partial-x-ptx}.
\end{proof}

The following lemma is borrowed from \cite[Theorem A.5.1]{Durrett2019Probability}.
\begin{lemma}\label{le:change-patial-int}
  Let $(Y,\mathcal Y,\mu)$ be a measurable  space with $\mu:\mathcal Y\to[0,\infty]$. Let $f$ be a complex-valued function defined on
   $\R\times Y$ and $\delta\in(0,\infty)$ be a constant.   Suppose that,  for  $x\in(x_0-\delta,x_0+\delta)$,
  \begin{enumerate}[(i)]
    \item $u(x)=\displaystyle \int_{Y}f(x,y)\mu(\ud y)$ with $\displaystyle\int_{Y}|f(x,y)|\mu(\ud y)<\infty$;
    \medskip
    
    \item for fixed $y$, $\displaystyle \frac{\partial}{\partial x}f(x,y)$ exists and is a continuous function of $x$;
    \medskip
    
    \item $v(x)=\displaystyle\int_{Y}\frac{\partial}{\partial x}f(x,y) \mu(\ud y)$ is continuous at $x=x_0$;
    \medskip
    
    \item $\displaystyle\int_{Y}\int_{-\delta}^{\delta}\left|\frac{\partial}{\partial x}f(x_0+z,y) \right|\ud z\mu(\ud y)<\infty$;
  \end{enumerate}
  then we have \[u'(x_0)=v(x_0).\]
\end{lemma}
\begin{proof}
  Letting $|h|\le\delta$ and using (i) and (ii), we have
  \begin{align*}
    u(x_0+h)-u(x_0) &= \int_{Y} f(x_0+h,y)-f(x_0,y) \mu(\ud y)\\
    &=\int_{Y}\int_{0}^{h}\frac{\partial}{\partial x}f(x_0+z,y)\ud z\mu(\ud y).
  \end{align*}
In view of (iv), we can apply Fubini’s theorem to get
  \begin{align*}
    u(x_0+h)-u(x_0) &= \int_{0}^{h}\int_{Y}\frac{\partial}{\partial x}f(x_0+z,y)\mu(\ud y)\ud z= \int_0^h v(x_0+z) \ud z.
  \end{align*}
  The above equation implies
  \begin{equation*}
    \frac{u(x_0+h)-u(x_0)}{h} =\frac{1}{h} \int_{0}^{h} v(x_0+z)\ud z.
  \end{equation*}
  Since $v$ is continuous at $x_0$ by (iii), letting $h\to 0$ gives the desired result.
\end{proof}

The following lemma is used in the proof of Proposition \ref{prop:reg-V-add}.
\begin{lemma}\label{le:partial-fourier}
 Recall that the fundamental solution $p(t,x)$ is given in \eqref{e:de-ptx}. For all $t > s$, we have
  \begin{equation}\label{e:7-1}
    \cF\left[\frac{\partial}{\partial t} p(t-s,x-\cdot)\right](\xi)= \frac{\partial}{\partial t} \cF p(t-s,x-\cdot)(\xi),
  \end{equation}
   and  if we assume $\alpha>1$, we also have
  \begin{equation}\label{e:7-2}
    \cF\left[\frac{\partial}{\partial x_i} p(t-s,x-\cdot)\right](\xi)= \frac{\partial}{\partial x_i} \cF p(t-s,x-\cdot)(\xi),\text{ for all $i=1,\dots, d$.}
  \end{equation}
\end{lemma}

\begin{proof}
{\em Step1: } First we prove the assertion \eqref{e:7-1}. By  \cite[(2.2.1)]{kilbas2004h} and \cite[(2.1.1)]{kilbas2004h}, we have
\begin{equation}\label{e:udt-ptx}
\begin{aligned}
  &\frac{\ud}{\ud x}   \FoxH{2,1}{2,3}{x}{(1,1),\:(\beta+\gamma,\beta)}{(d/2,\alpha/2),\:(1,1),\:(1,\alpha/2)}\\
  &=x^{-1} \FoxH{2,2}{3,4}{x}{(0,1),\:(1,1),\:(\beta+\gamma,\beta)}{(d/2,\alpha/2),\:(1,1),\:(1,\alpha/2),\:(1,1)}\\
  &=x^{-1} \FoxH{2,1}{2,3}{x}{(0,1),\:(\beta+\gamma,\beta)}{(d/2,\alpha/2),\:(1,1),\:(1,\alpha/2)}.
\end{aligned}
\end{equation}
Using \eqref{e:de-ptx} and \eqref{e:udt-ptx}, we have
\begin{align}\label{e:partial-t-ptx}
  &\frac{\partial}{\partial t}  p(t-s,x) \notag \\
  &= \frac{\partial}{\partial t}\left[ \pi^{-d/2}|x|^{-d}(t-s)^{\beta+\gamma-1}\times \FoxH{2,1}{2,3}{\frac{ |x|^\alpha}{2^{\alpha} (t-s)^\beta}}
       {(1,1),\:(\beta+\gamma,\beta)}{(d/2,\alpha/2),\:(1,1),\:(1,\alpha/2)}\right]\notag\\
       &= (\beta+\gamma-1)\pi^{-d/2}|x|^{-d}(t-s)^{\beta+\gamma-2}\times \FoxH{2,1}{2,3}{\frac{ |x|^\alpha}{2^{\alpha} (t-s)^\beta}}
       {(1,1),\:(\beta+\gamma,\beta)}{(d/2,\alpha/2),\:(1,1),\:(1,\alpha/2)}\notag\\
       &\quad -\beta \pi^{-d/2}|x|^{-d} (t-s)^{\beta+\gamma-2}\FoxH{2,1}{2,3}{\frac{ |x|^\alpha}{2^{\alpha} (t-s)^\beta}}{(0,1),\:(\beta+\gamma,\beta)}
       {(d/2,\alpha/2),\:(1,1),\:(1,\alpha/2)}.
\end{align}
Then, by a change of variable and (\ref{e:partial-t-ptx}),
 \begin{align*}
    & \int_{\R^d} \frac{\partial}{\partial t}  p(t-s,x-y)e^{-i\langle y,\xi\rangle}\ud y=\int_{\R^d} \frac{\partial}{\partial t}  p(t-s,y)e^{-i\langle x-y,\xi\rangle}\ud y\\
    &=(t-s)^{\beta+\gamma-2} \int_{\R^d}\bigg[ (\beta+\gamma-1)\pi^{-d/2}|y|^{-d}\times \FoxH{2,1}{2,3}{\frac{ |y|^\alpha}{2^{\alpha} }}
       {(1,1),\:(\beta+\gamma,\beta)}{(d/2,\alpha/2),\:(1,1),\:(1,\alpha/2)}\\
       &\quad -\beta \pi^{-d/2}|y|^{-d} \FoxH{2,1}{2,3}{\frac{ |y|^\alpha}{2^{\alpha} }}{(0,1),\:(\beta+\gamma,\beta)}{(d/2,\alpha/2),\:(1,1),\:(1,\alpha/2)}\bigg]
       e^{-i\langle x-y,\xi\rangle}\ud y,
 \end{align*}
 the above function is continuous in $t \in (s, \infty)$. By Lemma \ref{le:change-patial-int}, it is sufficient to show that for any $t>s$
  \begin{equation}\label{e:step1-1}
    \int_{\R^d} |p(t-s,x-y)|\ud y<\infty,
  \end{equation}
  and for any $t>s$, there exists $\delta>0$ such that
  \begin{equation}\label{e:step1-2}
    \int_{\R^d}\ud y\int_{-\delta}^{\delta} \ud r\left|\frac{\partial}{\partial t}p(t-s+r,x-y)\right| <\infty.
  \end{equation}

 It follows from \cite[Lemma 4.3]{chen2019nonlinear} that  there is a constanrt $K_5> 0$ such that for all $0< x \le K_5$, 
\begin{equation}\label{e:asymp-p1x-0}
  p(1,x)\le
  \begin{cases}
    C|x|^{\alpha-d}, & \mbox{if } \alpha<d, \\
    C |\log x|, & \mbox{if } \alpha=d, \\
    C, & \alpha>d.
  \end{cases}
\end{equation}
By \cite[Lemma 4.5]{chen2019nonlinear}, there is a constant $K_6> 0$ such that for all $x \ge K_6$,
\begin{equation}\label{e:asymp-p1x-infty}
  p(1,x)\le C|x|^{-\alpha-d}.
\end{equation}
By  \eqref{e:asymp-p1x-0} and \eqref{e:asymp-p1x-infty}, we have
  \begin{equation*}
  \begin{split}
   &\int_{\R^d}\left| p(t-s,x-y)\right|\ud y =\int_{\R^d}\left| p(t-s,x)\right|\ud x\\
   &= \int_{\R^d}\pi^{-d/2}|x|^{-d}(t-s)^{\beta+\gamma-1}\times \left|\FoxH{2,1}{2,3}{\frac{ |x|^\alpha}{2^{\alpha} (t-s)^\beta}}
       {(1,1),\:(\beta+\gamma,\beta)}{(d/2,\alpha/2),\:(1,1),\:(1,\alpha/2)}\right| \ud x\\
   &= (t-s)^{\beta+\gamma-1} \int_{\R^d} \pi^{-d/2}|x|^{-d}\times \left|\FoxH{2,1}{2,3}{\frac{ |x|^\alpha}{2^{\alpha} }}
       {(1,1),\:(\beta+\gamma,\beta)}{(d/2,\alpha/2),\:(1,1),\:(1,\alpha/2)}\right| \ud x\\
   &=  (t-s)^{\beta+\gamma-1} \int_{\R^d}  \left|p(1,x)\right| \ud x<\infty,
  \end{split}
\end{equation*}
which proves \eqref{e:step1-1}.

To verify \eqref{e:step1-2}, we use \eqref{e:partial-t-ptx} to bound the integral by
\begin{align*}
  &\int_{\R^d}\int_{-\delta}^{\delta}\left|\frac{\partial}{\partial t}p(t-s+r,x-y)\right|\ud r \ud y
  =\int_{-\delta}^{\delta}\int_{\R^d}\left|\frac{\partial}{\partial t}  p(t-s+r,x)\right|\ud r\ud x \le I_1+I_2,
\end{align*}
where
\begin{equation*}
\begin{split}
  I_1 &=\int_{-\delta}^{\delta}\ud r\int_{\R^d}\ud x |\beta+\gamma-1|\pi^{-d/2}|x|^{-d}(t-s+r)^{\beta+\gamma-2}\\
  &\qquad \qquad \times \left|\FoxH{2,1}{2,3}{\frac{ |x|^\alpha}{2^{\alpha} (t-s+r)^\beta}}
       {(1,1),\:(\beta+\gamma,\beta)}{(d/2,\alpha/2),\:(1,1),\:(1,\alpha/2)}\right|
\end{split}
\end{equation*}
and
\begin{equation*}
\begin{split}
  I_2 &=  \int_{-\delta}^{\delta}\int_{\R^d} \beta \pi^{-d/2}|x|^{-d} (t-s+r)^{\beta+\gamma-2}\left|\FoxH{2,1}{2,3}
  {\frac{ |x|^\alpha}{2^{\alpha} (t-s+r)^\beta}}{(0,1),\:(\beta+\gamma,\beta)}{(d/2,\alpha/2),\:(1,1),\:(1,\alpha/2)}\right|\ud r\ud t.
\end{split}
\end{equation*}
For the integral $I_1$, we have
\begin{align*}
  I_1 
       &=|\beta+\gamma-1| \int_{-\delta}^{\delta}(t-s+r)^{\beta+\gamma-2}\ud r\\
       &\qquad \quad\times\int_{\R^d}\pi^{-d/2}|x|^{-d} \left|\FoxH{2,1}{2,3}{\frac{ |x|^\alpha}{2^{\alpha} }}
       {(1,1),\:(\beta+\gamma,\beta)}{(d/2,\alpha/2),\:(1,1),\:(1,\alpha/2)}\right|\ud x\\
       &= \left((t-s+\delta)^{\beta+\gamma-1}-(t-s-\delta)^{\beta+\gamma-1}\right)\int_{\R^d} |p(1,x)| \ud x<\infty,
\end{align*}
where the integral with respect to $x$ is finite due to  the upper bounds \eqref{e:asymp-p1x-0} and \eqref{e:asymp-p1x-infty}.

On the other hand,
\begin{align*}
  I_2 
  &=\beta \pi^{-d/2}\int_{-\delta}^{\delta} (t-s+r)^{\beta+\gamma-2}\ud r \int_{\R^d}|x|^{-d} \left|\FoxH{2,1}{2,3}{\frac{ |x|^\alpha}{2^{\alpha} }} {(0,1),\:(\beta+\gamma,\beta)}{(d/2,\alpha/2),\:(1,1),\:(1,\alpha/2)}\right|\ud x\\
  &=\beta \pi^{-d/2}\left((t-s+\delta)^{\beta+\gamma-1}-(t-s-\delta)^{\beta+\gamma-1}\right)\\
  &\quad \times\int_{\R^d}|x|^{-d} \left|\FoxH{2,1}{2,3}{\frac{ |x|^\alpha}{2^{\alpha} }} {(0,1),\:(\beta+\gamma,\beta)}{(d/2,\alpha/2),\:(1,1),\:(1,\alpha/2)}\right|\ud x
  <\infty,
\end{align*}
where the integral with respect to $x$ is also finite thanks to Lemma \ref{le:partial-t-ptx}. This proves \eqref{e:step1-2}.

{\em Step 2:} Now we prove the second assertion \eqref{e:7-2}. By \cite[(2.2.1)]{kilbas2004h}, we have
\begin{align}\label{e:udxi-ptx}
   & \frac{\ud}{\ud x}\left[x^{-d}\times \FoxH{2,1}{2,3}{C x^\alpha}
       {(1,1),\:(\beta+\gamma,\beta)}{(d/2,\alpha/2),\:(1,1),\:(1,\alpha/2)}\right]\notag\\
       = &\, x^{-d-1}\FoxH{2,2}{3,4}{C x^\alpha}
       {(d,\alpha),\:(1,1),\:(\beta+\gamma,\beta)}{(d/2,\alpha/2),\:(1,1),\:(1,\alpha/2),\:(d+1,\alpha)}.
\end{align}
Then, by \eqref{e:de-ptx} and \eqref{e:udxi-ptx}, for $i=1, \dots, d$,
\begin{align}\label{e:exchange-y-neq-x}
  &\frac{\partial}{\partial x_i}p(t-s,x-y)= \pi^{-d/2}|x-y|^{-d-1}(t-s)^{\beta+\gamma-1}\notag\\
  &\quad\times \left|\FoxH{2,2}{3,4}{\frac{ |x-y|^\alpha}{2^{\alpha} (t-s)^\beta }}
       {(d,\alpha),\:(1,1),\:(\beta+\gamma,\beta)}{(d/2,\alpha/2),\:(1,1),\:(1,\alpha/2),\:(d+1,\alpha)}\right|\times\frac{x_i-y_i}{|x-y|},
\end{align}
which is continuous in $x_i$ when $y\neq x$.
We have obtained that $\int_{\R^d} |p(t-s,x-y)|\ud y<\infty$ in Step 1 and
\begin{align*}
  &\int_{\R^d} \frac{\partial}{\partial x_i}p(t-s,x-y)e^{-i\langle y,\xi\rangle}\ud y=\int_{\R^d}  \pi^{-d/2}|x-y|^{-d-1}(t-s)^{\beta+\gamma-1} \\
  &\quad \times\left|\FoxH{2,2}{3,4}{\frac{ |x-y|^\alpha}{2^{\alpha} (t-s)^\beta }}
       {(d,\alpha),\:(1,1),\:(\beta+\gamma,\beta)}{(d/2,\alpha/2),\:(1,1),\:(1,\alpha/2),\:(d+1,\alpha)}\right|\times\frac{x_i-y_i}{|x-y|}e^{-i\langle y,\xi\rangle} \ud y\\
  &=\int_{\R^d}  \pi^{-d/2}|y|^{-d-1}(t-s)^{\beta+\gamma-1} \\
  &\quad \times\left|\FoxH{2,2}{3,4}{\frac{ |y|^\alpha}{2^{\alpha} (t-s)^\beta }}
       {(d,\alpha),\:(1,1),\:(\beta+\gamma,\beta)}{(d/2,\alpha/2),\:(1,1),\:(1,\alpha/2),\:(d+1,\alpha)}\right|\times\frac{y_i}{|y|}e^{-i\langle x-y,\xi\rangle} \ud y,
\end{align*}
which is also continuous in $x_i$ by dominated convergence theorem and Lemma \ref{le:partial-x-ptx}, noting that we have assumed $\alpha>1$. Using the changes of variables $\tilde{y_i}=y_i-z$,  we get
\begin{equation*}
  \begin{split}
   &  \int_{\R^d}  \int_{-\delta}^{\delta} \left|\frac{\partial}{\partial x_i}p(t-s,x_1-y_1,\dots,x_i-y_i+z,\dots,x_d-y_d)\right|\ud z\ud y\\
   &=\int_{-\delta}^{\delta} \int_{\R^d} \left|\frac{\partial}{\partial x_i}p(t-s,x-y)\right|\ud y\ud z=2\delta\int_{\R^d}\left|\frac{\partial}{\partial x_i}p(t-s,x-y)\right|\ud y\\
   &= 2\delta\int_{\R^d}\pi^{-d/2}|x-y|^{-d-2}|x_i-y_i|t^{\beta+\gamma-1} \left|\FoxH{2,2}{3,4}{\frac{ |x-y|^\alpha}{2^{\alpha} t^\beta }}
       {(d,\alpha),\:(1,1),\:(\beta+\gamma,\beta)}{(d/2,\alpha/2),\:(1,1),\:(1,\alpha/2),\:(d+1,\alpha)}\right|  \ud x\\
   &=C\int_{0}^{\infty} r^{-2}t^{\beta+\gamma-1}\times \left|\FoxH{2,2}{3,4}{\frac{ r^\alpha}{2^{\alpha} t^\beta }}
       {(d,\alpha),\:(1,1),\:(\beta+\gamma,\beta)}{(d/2,\alpha/2),\:(1,1),\:(1,\alpha/2),\:(d+1,\alpha)}\right| \ud r\\
   &=C  t^{\beta+\gamma-1-\beta/\alpha} \int_{0}^{\infty}  r^{-2}\times \left|\FoxH{2,2}{3,4}{\frac{ r^\alpha}{2^{\alpha} }}
       {(d,\alpha),\:(1,1),\:(\beta+\gamma,\beta)}{(d/2,\alpha/2),\:(1,1),\:(1,\alpha/2),\:(d+1,\alpha)}\right| \ud r.
  \end{split}
\end{equation*}
 Noting that condition \eqref{e:con-space-add} implies $
\alpha>1$ and then applying  Lemma \ref{le:partial-x-ptx}, we prove that  the last integral is convergent.
By Lemma \ref{le:change-patial-int}, we have that
\begin{equation*}
  \int_{\R^d\backslash\{x\}} \frac{\partial}{\partial x_i}p(t-s,x-y)e^{-i\langle y,\xi\rangle}\ud y= \frac{\partial}{\partial x_i} \int_{\R^d\backslash\{x\}} p(t-s,x-y)e^{-i\langle y,\xi\rangle}\ud y,
\end{equation*}
which is equivalent to \eqref{e:7-2}. This completes the proof.
\end{proof}

The following lemma is adapted from \cite[Proposition 2.6]{khoshnevisan2023small} which deals with the case $\lambda =3/4$.
\begin{lemma}\label{le:an-asym}
 Define $a_1=1$ and $a_{n+1}=a_n+ca_n^{\lambda}$ for  $n\in\N$, where $c>0$ and $0<\lambda<1$ are constants. Then,
  \begin{equation*}
     a_n\sim\left(cn(1-\lambda)\right)^\frac{1}{1-\lambda} \text{ as } n\to \infty.
  \end{equation*}
\end{lemma}
\begin{proof}
Let
\begin{equation*}
  f(t):=a_{\lfloor t\rfloor}+c(t-\lfloor t\rfloor)a_{\lfloor t\rfloor}^{\lambda},~\text{ for }  t\ge1,
\end{equation*}
 where $\lfloor t\rfloor$ denotes the largest integer that is not bigger than $t$. Note that $f$ is differentiable on $[1,\infty)\setminus\N$, and
\begin{equation}\label{e:der-f}
  f'(s)=ca_{\lfloor s\rfloor}^{\lambda}=c \big(h(s) f(s)\big)^{\lambda},\quad \text{for} \quad s\in[1,\infty)\setminus\N,
\end{equation}
where
\begin{equation*}
  h(t)=\frac{a_{\lfloor t\rfloor}}{f(t)}=\frac{a_{\lfloor t\rfloor}}{a_{\lfloor t\rfloor}+c(t-\lfloor t\rfloor)a_{\lfloor t\rfloor}^{\lambda}}.
\end{equation*}
Since $\lambda\in(0,1)$ and  $a_{\lfloor t\rfloor}\to\infty$ as $t\to\infty$, we have $h(t)\to1$ as $t\to\infty$. We can write \eqref{e:der-f}
as ${\ud f}/{f^\lambda}=ch^\lambda \ud s$ and integrate from $1$ to $t$  to get
\begin{equation*}
  f(t)=\left(1+c(1-\lambda)\int_{1}^{t}h(s)^\lambda \ud s\right)^\frac{1}{1-\lambda},
\end{equation*}
for all $t\in[1,\infty)\setminus\N$ and hence all $t\ge1$ by continuity. Since $f(n)=a_n$ for all $n\in\N$ and $\lim_{s\to\infty}h(s)=1$,
this proves the result.
\end{proof}

The following lemma is a reformulation of Talagrand \cite[Lemma 2.2]{Talagrand1995Hausdorff} given by Ledoux \cite[ (7.11)-(7.13) on p.257]{ledoux2006isoperimetry}.
\begin{lemma}\label{le:talagrand}
Let $S$ be a set with the distance $d(s,t)=\E[|Z(s)-Z(t)|^2]^{1/2}$, where $\{Z(t)\}_{t\in S}$ is a Gaussian process. Let $\cN_d(S,\e)$
denote the smallest number of (open) $d$-balls of radius $\e$ needed to cover $S$. If there is a decreasing function $\Psi: (0, \delta] \to (0, \infty)$
such that  $N(S,\e)\le\Psi(\e)$ for all $\e \in (0, \delta]$. Assume that for some constant $C_2 \ge C_1>1$  we have $C_1 \Psi(\e) \le \Psi(\e/2)\le C_2\Psi(\e)$
for all $\e\in (0, \delta]$. Then for all $\e \in (0, \delta]$,
\begin{equation*}
\mathbb{P}\bigg\{\sup_{s,t\in S}|Z(t)-Z(s)|\le \e \bigg\}  \ge \exp\left(-\frac{\Psi(\e)}{K}\right),
  \end{equation*}
where the constant $K$ depends on $C_1$ and $C_2$ only.
\end{lemma}

\end{appendix}

\bigskip

\noindent{\bf Acknowledgements.}
The research of  J. Song is partially supported by National Natural Science Foundation of China (No. 12471142). The research of R. Wang is partially supported by  the NSF of Hubei Province  (2024AFB683).
The research of Y. Xiao is partially supported by the NSF grant DMS-2153846.

\bibliographystyle{plain}
\bibliography{ref-2024}

\begin{thebibliography}{10}

\bibitem{Xiao2017JDE}
Hassan Allouba and Yimin Xiao.
\newblock L-{K}uramoto-{S}ivashinsky {SPDE}s vs. time-fractional {SPIDE}s:
  exact continuity and gradient moduli, {$1/2$}-derivative criticality, and
  laws.
\newblock {\em J. Differential Equations}, 263(2):1552--1610, 2017.

\bibitem{Athreya2021Smallball}
Siva Athreya, Mathew Joseph, and Carl Mueller.
\newblock Small ball probabilities and a support theorem for the stochastic
  heat equation.
\newblock {\em Ann. Probab.}, 49(5):2548--2572, 2021.

\bibitem{balan2016intermittency}
Raluca~M. Balan and Daniel Conus.
\newblock Intermittency for the wave and heat equations with fractional noise
  in time.
\newblock {\em Ann. Probab.}, 44(2):1488--1534, 2016.

\bibitem{balan2015spdes}
Raluca~M. Balan, Maria Jolis, and Llu\'{\i}s Quer-Sardanyons.
\newblock S{PDE}s with affine multiplicative fractional noise in space with
  index {$\frac14<H<\frac12$}.
\newblock {\em Electron. J. Probab.}, 20:Paper No. 54, 36, 2015.

\bibitem{Balan2008Thestochastic}
Raluca~M. Balan and Ciprian~A. Tudor.
\newblock The stochastic heat equation with fractional-colored noise: existence
  of the solution.
\newblock {\em ALEA Lat. Am. J. Probab. Math. Stat.}, 4:57--87, 2008.

\bibitem{Balan2010thestocha}
Raluca~M. Balan and Ciprian~A. Tudor.
\newblock The stochastic wave equation with fractional noise: a random field
  approach.
\newblock {\em Stochastic Process. Appl.}, 120(12):2468--2494, 2010.

\bibitem{Carfagnini2022application}
Marco Carfagnini.
\newblock An application of the {G}aussian correlation inequality to the small
  deviations for a {K}olmogorov diffusion.
\newblock {\em Electron. Commun. Probab.}, 27:Paper No. 20, 7, 2022.

\bibitem{Chen2024Smallball}
Jiaming Chen.
\newblock Small ball probabilities for the stochastic heat equation with
  colored noise.
\newblock {\em Stochastic Process. Appl.}, 177:Paper No. 104455, 2024.

\bibitem{chen2017nonlinear}
Le~Chen.
\newblock Nonlinear stochastic time-fractional diffusion equations on
  {$\Bbb{R}$}: moments, {H}\"{o}lder regularity and intermittency.
\newblock {\em Trans. Amer. Math. Soc.}, 369(12):8497--8535, 2017.

\bibitem{chen2022interpolating}
Le~Chen and Nicholas Eisenberg.
\newblock Interpolating the stochastic heat and wave equations with
  time-independent noise: solvability and exact asymptotics.
\newblock {\em Stoch. Partial Differ. Equ. Anal. Comput.}, pages 1--51, 2022.

\bibitem{Chen2023Interpolating}
Le~Chen and Nicholas Eisenberg.
\newblock Interpolating the stochastic heat and wave equations with
  time-independent noise: solvability and exact asymptotics.
\newblock {\em Stoch. Partial Differ. Equ. Anal. Comput.}, 11(3):1203--1253,
  2023.

\bibitem{chen2024moments}
Le~Chen, Yuhui Guo, and Jian Song.
\newblock Moments and asymptotics for a class of spdes with space-time white
  noise.
\newblock {\em Trans. Amer. Math. Soc.}, 377(06):4255--4301, 2024.

\bibitem{chen2017space}
Le~Chen, Guannan Hu, Yaozhong Hu, and Jingyu Huang.
\newblock Space-time fractional diffusions in {G}aussian noisy environment.
\newblock {\em Stochastics}, 89(1):171--206, 2017.

\bibitem{Chen2018Intermittency}
Le~Chen, Yaozhong Hu, Kamran Kalbasi, and David Nualart.
\newblock Intermittency for the stochastic heat equation driven by a rough time
  fractional {G}aussian noise.
\newblock {\em Probab. Theory Related Fields}, 171(1-2):431--457, 2018.

\bibitem{chen2019nonlinear}
Le~Chen, Yaozhong Hu, and David Nualart.
\newblock Nonlinear stochastic time-fractional slow and fast diffusion
  equations on {$\Bbb R^d$}.
\newblock {\em Stochastic Process. Appl.}, 129(12):5073--5112, 2019.

\bibitem{Chen2018Temporal}
Xia Chen, Yaozhong Hu, Jian Song, and Xiaoming Song.
\newblock Temporal asymptotics for fractional parabolic {A}nderson model.
\newblock {\em Electron. J. Probab.}, 23:Paper No. 14, 39, 2018.

\bibitem{Durrett2019Probability}
Rick Durrett.
\newblock {\em Probability---theory and examples}, volume~49 of {\em Cambridge
  Series in Statistical and Probabilistic Mathematics}.
\newblock Cambridge University Press, Cambridge, fifth edition, 2019.

\bibitem{Foondun2023Smallball}
Mohammud Foondun, Mathew Joseph, and Kunwoo Kim.
\newblock Small ball probability estimates for the {H}\"older semi-norm of the
  stochastic heat equation.
\newblock {\em Probab. Theory Related Fields}, 185(1-2):553--613, 2023.

\bibitem{guo2023stochastic}
Yuhui Guo, Jian Song, and Xiaoming Song.
\newblock Stochastic fractional diffusion equations with {G}aussian noise rough
  in space.
\newblock {\em Bernoulli}, 30(3):1774--1799, 2024.

\bibitem{hms11}
H.~J. Haubold, A.~M. Mathai, and R.~K. Saxena.
\newblock Mittag-{L}effler functions and their applications.
\newblock {\em J. Appl. Math.}, pages Art. ID 298628, 51, 2011.

\bibitem{Herrell2020sharp}
Randall Herrell, Renming Song, Dongsheng Wu, and Yimin Xiao.
\newblock Sharp space-time regularity of the solution to stochastic heat
  equation driven by fractional-colored noise.
\newblock {\em Stoch. Anal. Appl.}, 38(4):747--768, 2020.

\bibitem{hu2017stochastic}
Yaozhong Hu, Jingyu Huang, Khoa L\^{e}, David Nualart, and Samy Tindel.
\newblock Stochastic heat equation with rough dependence in space.
\newblock {\em Ann. Probab.}, 45(6B):4561--4616, 2017.

\bibitem{Hu2012Feynman}
Yaozhong Hu, Fei Lu, and David Nualart.
\newblock Feynman-{K}ac formula for the heat equation driven by fractional
  noise with {H}urst parameter {$H<1/2$}.
\newblock {\em Ann. Probab.}, 40(3):1041--1068, 2012.

\bibitem{hu2009stochastic}
Yaozhong Hu and David Nualart.
\newblock Stochastic heat equation driven by fractional noise and local time.
\newblock {\em Probab. Theory Related Fields}, 143(1-2):285--328, 2009.

\bibitem{hns2011Feynman}
Yaozhong Hu, David Nualart, and Jian Song.
\newblock Feynman-{K}ac formula for heat equation driven by fractional white
  noise.
\newblock {\em Ann. Probab.}, 39(1):291--326, 2011.

\bibitem{khoshnevisan2023small}
Davar Khoshnevisan, Kunwoo Kim, and Carl Mueller.
\newblock Small-ball constants, and exceptional flat points of {SPDE}s.
\newblock {\em arXiv preprint arXiv:2312.05789}, 2023.

\bibitem{kilbas2004h}
Anatoly~A. Kilbas and Megumi Saigo.
\newblock {\em {$H$}-transforms}, volume~9 of {\em Analytical Methods and
  Special Functions}.
\newblock Chapman $\&$ Hall/CRC, Boca Raton, FL, 2004.
\newblock Theory and applications.

\bibitem{kilbas2006theory}
Anatoly~A. Kilbas, Hari~M. Srivastava, and Juan~J. Trujillo.
\newblock {\em Theory and applications of fractional differential equations},
  volume 204 of {\em North-Holland Mathematics Studies}.
\newblock Elsevier Science B.V., Amsterdam, 2006.

\bibitem{Lata2017Royen}
Rafa\l{} Lata{\l}a and Dariusz Matlak.
\newblock Royen's proof of the {G}aussian correlation inequality.
\newblock In {\em Geometric aspects of functional analysis}, volume 2169 of
  {\em Lecture Notes in Math.}, pages 265--275. Springer, Cham, 2017.

\bibitem{ledoux2006isoperimetry}
Michel Ledoux.
\newblock Isoperimetry and {G}aussian analysis.
\newblock {\em Lectures on Probability Theory and Statistics: Ecole d’Et{\'e}
  de Probabilit{\'e}s de Saint-Flour XXIV—1994}, pages 165--294, 2006.

\bibitem{Lee2023Chung-type}
Cheuk~Yin Lee and Yimin Xiao.
\newblock Chung-type law of the iterated logarithm and exact moduli of
  continuity for a class of anisotropic {G}aussian random fields.
\newblock {\em Bernoulli}, 29(1):523--550, 2023.

\bibitem{Li1999Gaussian}
Wenbo~V. Li.
\newblock A {G}aussian correlation inequality and its applications to small
  ball probabilities.
\newblock {\em Electron. Comm. Probab.}, 4:111--118, 1999.

\bibitem{Li2001Gaussian}
Wenbo~V. Li and Qi-Man Shao.
\newblock Gaussian processes: inequalities, small ball probabilities and
  applications.
\newblock In {\em Stochastic processes: theory and methods}, volume~19 of {\em
  Handbook of Statist.}, pages 533--597. North-Holland, Amsterdam, 2001.

\bibitem{Xiao2010Chung}
Nana Luan and Yimin Xiao.
\newblock Chung's law of the iterated logarithm for anisotropic {G}aussian
  random fields.
\newblock {\em Statist. Probab. Lett.}, 80(23-24):1886--1895, 2010.

\bibitem{Xiao2013Fernique}
Mark~M. Meerschaert, Wensheng Wang, and Yimin Xiao.
\newblock Fernique-type inequalities and moduli of continuity for anisotropic
  {G}aussian random fields.
\newblock {\em Trans. Amer. Math. Soc.}, 365(2):1081--1107, 2013.

\bibitem{mmv01}
Jean M{\'e}min, Yulia Mishura, and Esko Valkeila.
\newblock Inequalities for the moments of {W}iener integrals with respect to a
  fractional {B}rownian motion.
\newblock {\em Statist. Probab. Lett.}, 51(2):197--206, 2001.

\bibitem{Mueller2002Hitting}
Carl Mueller and Roger Tribe.
\newblock Hitting properties of a random string.
\newblock {\em Electron. J. Probab.}, 7:no. 10, 29, 2002.

\bibitem{nualart06}
David Nualart.
\newblock {\em The Malliavin calculus and related topics}, volume 1995.
\newblock Springer, 2006.

\bibitem{Olver2010NIST}
Frank W.~J. Olver, Daniel~W. Lozier, Ronald~F. Boisvert, and Charles~W. Clark,
  editors.
\newblock {\em N{IST} handbook of mathematical functions}.
\newblock U.S. Department of Commerce, National Institute of Standards and
  Technology, Washington, DC; Cambridge University Press, Cambridge, 2010.
\newblock With 1 CD-ROM (Windows, Macintosh and UNIX).

\bibitem{podlubny1998fractional}
Igor Podlubny.
\newblock {\em Fractional differential equations}, volume 198 of {\em
  Mathematics in Science and Engineering}.
\newblock Academic Press, Inc., San Diego, CA, 1999.
\newblock An introduction to fractional derivatives, fractional differential
  equations, to methods of their solution and some of their applications.

\bibitem{Royen2014simpleproof}
Thomas Royen.
\newblock A simple proof of the {G}aussian correlation conjecture extended to
  some multivariate gamma distributions.
\newblock {\em Far East J. Theor. Stat.}, 48(2):139--145, 2014.

\bibitem{Shao2003Gaussian}
Qi-Man Shao.
\newblock A {G}aussian correlation inequality and its applications to the
  existence of small ball constant.
\newblock {\em Stochastic Process. Appl.}, 107(2):269--287, 2003.

\bibitem{song2017class}
Jian Song.
\newblock On a class of stochastic partial differential equations.
\newblock {\em Stochastic Process. Appl.}, 127(1):37--79, 2017.

\bibitem{Talagrand1995Hausdorff}
Michel Talagrand.
\newblock Hausdorff measure of trajectories of multiparameter fractional
  {B}rownian motion.
\newblock {\em Ann. Probab.}, 23(2):767--775, 1995.

\bibitem{Tudor2017Sample}
Ciprian~A. Tudor and Yimin Xiao.
\newblock Sample paths of the solution to the fractional-colored stochastic
  heat equation.
\newblock {\em Stoch. Dyn.}, 17(1):1750004, 20, 2017.

\bibitem{Xiao2009Sample}
Yimin Xiao.
\newblock Sample path properties of anisotropic {G}aussian random fields.
\newblock In {\em A minicourse on stochastic partial differential equations},
  volume 1962 of {\em Lecture Notes in Math.}, pages 145--212. Springer,
  Berlin, 2009.

\bibitem{yaglom1957some}
Akiva~M. Yaglom.
\newblock Some classes of random fields in n-dimensional space, related to
  stationary random processes.
\newblock {\em Theory of Probability $\&$ Its Applications}, 2(3):273--320,
  1957.

\end{thebibliography}

\end{document}